\documentclass[sn-mathphys-num]{sn-jnl}
\usepackage[T1]{fontenc}
\usepackage{lmodern}
\usepackage{amsmath,amsfonts,amsthm,amssymb}
\usepackage{tikz-cd}
\usepackage{hyperref}
\usepackage{geometry}
\usepackage{nicefrac,xfrac}
\usepackage{cleveref}
\usepackage{physics}
\usepackage{chemfig} %for CRNs
\usepackage{enumerate}
\usepackage{enumitem}
\usepackage{graphicx}%
\usepackage{multirow}%
\usepackage{mathrsfs}%
\usepackage[title]{appendix}%
\usepackage{textcomp}%
\usepackage{manyfoot}%
\usepackage{booktabs}%
\usepackage{algorithm}%
\usepackage{algorithmicx}%
\usepackage{algpseudocode}%
\usepackage{listings}%
\usepackage{ragged2e}

\hypersetup{
  colorlinks   = true, %Colours links instead of ugly boxes
  urlcolor     = blue, %Colour for external hyperlinks
  linkcolor    = magenta, %Colour of internal links
  citecolor   = magenta %Colour of citations
}
\usepackage{xcolor,colortbl}
\usepackage{array}

\makeatletter
\newcommand{\thickhline}{%
    \noalign {\ifnum 0=`}\fi \hrule height 1.155pt
    \futurelet \reserved@a \@xhline
}
\newcolumntype{"}{@{\hskip\tabcolsep\vrule width 1.155pt\hskip\tabcolsep}}
\makeatother
\newcommand{\N}{\mathbb{N}}

\newcommand{\R}{\mathbb{R}}

\newcommand{\Vector}[1]{\ensuremath{\boldsymbol{#1}}}

\DeclareMathOperator{\conv}{conv}

\definecolor{cellone}{RGB}{204, 23, 217}
\definecolor{celltwo}{RGB}{0, 139, 0}
\definecolor{cellthree}{RGB}{0, 125, 188}

\definecolor{NiceBlue}{rgb}{0.2,0.2,0.75}
\newif\ifstartedinmathmode
\newcommand{\struc}[1]{{\relax\ifmmode\startedinmathmodetrue\else\startedinmathmodefalse\fi\color{NiceBlue}{\ifstartedinmathmode #1 \else\textit{#1}\fi}}}

\newtheorem{thm}{Theorem}[section]

\newtheorem{lemma}[thm]{Lemma}

\theoremstyle{definition}

\newtheorem{eg}[thm]{Example}

\begin{document}
\title{Empirically Exploring the Space of Monostationarity in Dual Phosphorylation\thanks{\scriptsize Supported by the MSRI-MPI Leipzig Summer Graduate School on Algebraic Methods for Biochemical Reaction Networks.}}
\author[1]{\fnm{May} \sur{Cai}}\email{ mcai@gatech.edu}
\equalcont{These authors contributed equally to this work.}

\author*[2]{\fnm{Matthias} \sur{Himmelmann}}\email{himmelmann1@uni-potsdam.de}
\equalcont{These authors contributed equally to this work.}

\author[3]{\fnm{Birte} \sur{Ostermann}}\email{birte.ostermann@tu-braunschweig.de}
\equalcont{These authors contributed equally to this work.}

\affil[1]{\orgdiv{School of Mathematics}, \orgname{Georgia Tech}, \orgaddress{\city{Atlanta}, \postcode{30332}, \state{Georgia}, \country{USA}}}
\affil*[2]{\orgdiv{Institute of Mathematics}, \orgname{University of Potsdam}, \orgaddress{\city{Potsdam}, \postcode{14476}, \country{Germany}}}
\affil[3]{\orgdiv{Institute for Analysis and Algebra}, \orgname{TU Braunschweig}, \orgaddress{\city{Braunschweig}, \postcode{38106}, \country{Germany}}}

\abstract{The dual phosphorylation network provides an essential component of intracellular signaling, affecting the expression of phenotypes and cell metabolism. For particular choices of kinetic parameters, this system exhibits multistationarity, a property that is relevant in the decision-making of cells. Determining which reaction rate constants correspond to monostationarity and which produce multistationarity is an open problem. The system's monostationarity is linked to the nonnegativity of a specific polynomial. A previous study by Feliu et al. provides a sufficient condition for monostationarity via a decomposition of this polynomial into nonnegative circuit polynomials. However, this decomposition is not unique. We extend their work by a systematic approach to classifying such decompositions in the dual phosphorylation network. Using this result classification, we provide a qualitative comparison of the decompositions into nonnegative circuit polynomials via empirical experiments and improve on previous conditions for the region of monostationarity.
} 

\maketitle
\section*{Acknowledgements}
This work emerged from a group project during the MSRI-MPI Leipzig Summer Graduate School on Algebraic Methods for Biochemical Reaction Networks. We are grateful for the opportunity to participate in that inspiring research environment and thank Nicole Kitt as well as our other former group project members for insightful discussions. Special thanks to Elisenda Feliu and Timo de Wolff for proposing the project to us and for their support through various helpful comments and suggestions.
MC was partially supported by NSF-DMS grant \#1855726. MH was funded by the German Research Foundation (DFG) grant 195170736-TRR109. BO was partially funded by the German Federal Ministry for Economic Affairs and Climate Action (BMWK), project ProvideQ. 

\section{Introduction}
\label{section:introduction}
Chemical reaction networks (CRNs) provide a model for the behavior of chemical systems in which molecules react to create new species. A multitude of biological processes are represented using chemical reactions, for instance, cellular metabolism and gene regulation \cite{10.3389/fgene.2019.00549}. Under the assumption of mass action kinetics, CRNs are typically described using ordinary differential equations. Often, these systems have a stochastic component, though we restrict ourselves to deterministic models that are associated with differential algebraic equations. This assumption makes an analytical and algebraic investigation of these systems feasible. As early as 1972, mathematical techniques have been used to study the steady states of CRNs \cite{Horn1972, Feinberg1972}. A steady state refers to a composition of molecular species where the reactions balance each other and no changes in the concentrations of the reactants occur over time. 

CRNs can possess a single or multiple steady states depending on the specific reaction rates and reactants. The ability to achieve different steady states depending on external inputs is associated with a ``switch-like'' behavior. This is a driving factor in cellular decision-making \cite{LAURENT1999418, Ozbudak2004}. Characterizing the specific regions of parameter values in a given CRN where multiple steady states are attainable is generally a hard problem.

Alternatively, we can consider those parameter values where only a single steady state is admissible. Since we are only considering differential algebraic equations, the task of finding steady states can be turned into the algebraic problem of finding the positive real zeros of a polynomial system. Furthermore, the question of deciding on the \struc{monostationarity} of a CRN can be translated to the question of deciding whether a specific $n$-variate polynomial is nonnegative on $\mathbb{R}^n_{\geq 0}$. It is a fundamental yet NP-hard problem in algebraic geometry and polynomial optimization to determine the nonnegativity of a polynomial; see \cite{Laurent:Survey, Blekherman:Parrilo:Thomas, Lasserre:IntroductionPolynomialandSemiAlgebraicOptimization,Theobald:Book:RealAlgGeom}. Therefore, relaxations are typically considered to find certificates of nonnegativity which are sufficient criteria for effectively proving the polynomial's nonnegativity. A widely-used certificate is to show that a polynomial can be written as a sum of squares of polynomials (SOS), which is a property that can be tested computationally by semidefinite programming \cite{parrilo2000structured}. 
In order to certify the nonnegativity of the specific polynomial arising in the context of dual phosphorylation, we consider a different certificate, namely \struc{sums of nonnegative circuit polynomials (SONC)} which Iliman and de Wolff introduced \cite{SONCinitial}. 
Circuit polynomials generalize a result on agiforms by Reznick \cite{ReznickAGIforms} and were independently developed by Pantea, Koeppl and Craciun \cite{CRN_SONC_Pantea} in the context of biochemical reaction networks. 
In particular, SONC certificates allow us to derive symbolic bounds for the nonnegativity of the specific polynomial we are investigating in this article. 

\struc{Circuit polynomials} constitute a certain class of polynomials with a single negative coefficient. These polynomials each have an associated invariant called \struc{circuit number} that is easily computable via solving a linear system, even for symbolic coefficients. This invariant can be used to check the nonnegativity of a circuit polynomial \cite{SONCinitial}, and thus provides a useful tool for determining the multistationarity of CRNs.
If a polynomial can be decomposed into nonnegative circuit polynomials, it is nonnegative. We call such a decomposition a \struc{circuit cover}.

Determining the nonnegativity of a multivariate polynomial is a strategy for solving polynomial minimization problems by finding the largest $\lambda\in \mathbb{R}$ such that $g(\Vector{x})-\lambda\geq 0$ for $g\in \mathbb{R}[\Vector{x}]$ and all $\Vector{x}\in \mathbb{R}^n$. In particular, the decomposition of $g(\Vector{x})-\lambda$ into nonnegative circuit polynomials relaxes the original optimization problem and provides a lower bound for the minimum. 
There are various algorithms based on this approach (e.g. \cite{Iliman:deWolff:GP,dualityofnonnegativecircuits,Gennadiy,MAGRON2023346}). Besides their application in polynomial optimization, SONC certificates have, for instance, been used to find Lyapunov functions for solving dynamical systems \cite{Heuer:deWolff}. 

In this article, we investigate the \struc{dual phosphorylation} of proteins. This reaction is essential for the synthesis of molecules, the transport of free energy, cellular signaling and movements such as muscle contractions. It occurs in all forms of life, from amoebas and plants to humans \cite{cohen}. Dual phosphorylation is a cyclical reaction and is depicted in \Cref{fig:CRN}. 
The dual phosphorylation network poses an interesting challenge, as it is sufficiently large to make intuitive approaches infeasible, while it is small enough for a rigorous mathematical investigation. This statement is emphasized by several articles that have used mathematical tools to examine the dual phosphorylation network (e.g. \cite{bb6e59ad-64dd-32f7-8a45-283020c4203c,Flockerzi2014,ConradiFeliuMinchevaWiuf2017}). In particular, several algebraic sufficient conditions on the monostationarity of CRNs are available \cite{ConradiMincheva2014,Conradi2024,Feliu:Kaihnsa:Yueruek:deWolff:nSite,BIHAN2020367}. Many of these results rely on the theory on nonnegative polynomials, establishing a link between algebraic geometry and CRNs. 
Here, we are building on results from Feliu et al. \cite{DualPhosphorSONC} to empirically explore the region of monostationarity associated with the dual phosphorylation network using nonnegative circuit polynomials.
By systematically classifying circuit covers, we generate improved sufficient conditions for this chemical reaction network's monostationarity.

\begin{figure}[h!]
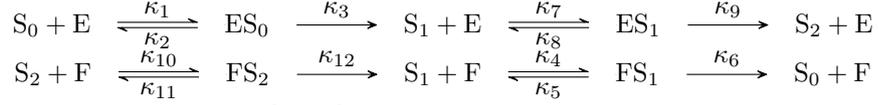

    \centering
    \schemestart
    \chemfig{@{a}S_0 + E} \arrow{<=>[$\kappa_1$][$\kappa_2$]}[0] \chemfig{@{b}ES_0}  \arrow{->[$\kappa_3$]}[0]
    \chemfig{@{c}S_1 + E} \arrow{<=>[$\kappa_7$][$\kappa_8$]}[0] 
    \chemfig{@{d}ES_1} \arrow{->[$\kappa_9$]}[0]
    \chemfig{@{e}S_2+E}
    \schemestop\\
    \schemestart
    \chemfig{@{f}S_2+F} \arrow{<=>[$\kappa_{10}$][$\kappa_{11}$]}[0] \chemfig{@{g}FS_2}  \arrow{->[$\kappa_{12}$]}[0]
    \chemfig{@{h}S_1 + F} \arrow{<=>[$\kappa_4$][$\kappa_5$]}[0] 
    \chemfig{@{i}FS_1} \arrow{->[$\kappa_6$]}[0]
    \chemfig{@{j}S_0+F}
    \schemestop
    \caption{The dual phosphorylation CRN \cite{WangSontag, ConradiMincheva2014}: The species $S_0, S_1, S_2$ denote the three phosphoforms of substrate $S$. The enzymes kinase and phosphatase are $E$ and $F$ respectively. The reaction rate constants are the parameters $\kappa_1,\dots, \kappa_{12} \in \mathbb{R}_{\geq 0}$.}
    \label{fig:CRN}
\end{figure}

Following Feliu et al. \cite{DualPhosphorSONC}, we first introduce the concept of a CRN in \Cref{section:CRN} and establish relevant results on nonnegative circuit polynomials in \Cref{section:SONC}. By combining these two topics, we investigate existing sufficient criteria on the monostationarity of the dual phosphorylation network in \Cref{section:dual-phosphorylation}. In \Cref{section:empirical data}, we build on this connection to produce novel results about the dual phosphorylation network's region of monostationarity. In particular, \Cref{prop:16-circuit-covers} presents the characterization of all pure circuit covers, in which each vertex of the polynomial's Newton polytope appears exactly once. This allows us to introduce an empirical approach in \Cref{section:experimental-setup} to compare these pure covers' quality. In \Cref{section:the-best-circuit-cover}, we ascertain which such pure cover corresponds to the largest region of monostationarity and investigate how these pure covers relate to each other in \Cref{section:containment}. These results allow us to state new sufficient conditions for the monostationarity of the dual phosphorylation network in \Cref{cor:newsufficientconditions}. Since pure circuit covers do not necessarily capture the entire region of nonnegativity, we also consider weighted covers in \Cref{sec:weighted} that combine multiple pure covers.

\section{Preliminaries}
\label{section:preliminaries}
Before we discuss our experimental setup and the results of our investigation, we introduce several concepts that relate the monostationarity of chemical reaction networks and the nonnegativity of polynomials. As our results extend the work of Feliu, Kaihnsa, de Wolff and Yürük \cite{DualPhosphorSONC}, we closely follow their notation and theoretical concepts.

\subsection{Chemical Reaction Networks}
\label{section:CRN}
A \struc{chemical reaction network (CRN)} consists of a set of \struc{species} and a set of \struc{reactions}, where each reaction is defined by its reactant and product species. In particular, we consider a network of reactions that transform chemical complexes between themselves. Chemical complexes are made up of individual species, and each reaction is governed by a reaction rate constant. See Dickenstein \cite{dickenstein2016biochemical} for a detailed survey. The paradigm of \struc{mass-action kinetics} assumes that a reaction rate is proportional to the occurrences of collisions between reactants. These quantities are related to the concentrations of the reactants. Thus we can consider the behavior of the reactions over time as a system of ordinary differential equations (ODEs) in the positive reals. Each ODE only involves polynomial terms with species concentrations as variables and the \struc{reaction rate parameters $\kappa_i \in \R_{> 0}$} appearing in the coefficients.

The specific CRN we consider is the \struc{dual phosphorylation network}, as depicted in Figure~\ref{fig:CRN}. The species $S_0,\, S_1$ and $S_2$ denote the three phosphoforms of substrate $S$. The enzymes kinase and phosphatase are denoted by $E$ and $F$, respectively.
We label the concentrations of the species by $x_1\,=\,E,\, x_2\,=\,F,\, x_3\,=\,S_0,\, x_4\,=\,S_1,\, x_5\,=\,S_2,\, x_6~=~ ES_0,\, x_7\,=\,FS_1,\, x_8\,=\,ES_1,\, x_9\,=\,FS_2$.
Mass-action kinetics predicts that the system of ODEs modeling the dual phosphorylation network (cf. \cite{DualPhosphorSONC}) is given by

\begin{align*}
\dv{x_1}{t} &= -\kappa_1x_1x_3 - \kappa_7x_1x_4 + \kappa_2x_6 + \kappa_3 x_7 + \kappa_8x_8 + \kappa_9x_8 & \dv{x_8}{t} &= \kappa_1x_1x_3 - \kappa_2x_6 - \kappa_3x_6\\
\dv{x_2}{t} &= -\kappa_4x_2x_4 - \kappa_{10}x_2x_5 + \kappa_5x_7 + \kappa_6x_7 + \kappa_{11}x_9 + \kappa_{12}x_9 & \dv{x_7}{t} &= \kappa_4x_2x_4 - \kappa_5x_7 - \kappa_6x_7\\
\dv{x_3}{t} &= -\kappa_1x_1x_3 + \kappa_2x_6 + \kappa_6x_7 & \dv{x_8}{t} &= \kappa_7x_1x_4 - \kappa_8x_8 - \kappa_9x_8\\
\dv{x_4}{t} &= -\kappa_4x_2x_4 - \kappa_7x_1x_4 + \kappa_3x_6 + \kappa_5x_7 \kappa_8x_8 + \kappa_{12}x_9 & \dv{x_9}{t} &= \kappa_{10}x_2x_5 - \kappa_{11}x_9 - \kappa_{12}x_9.\\
\dv{x_5}{t} &= -\kappa_{10}x_2x_5 + \kappa_9x_8 + \kappa_{11}x_9 &&
\end{align*}

The reaction network admits an additional three linear constraints coming from the conservation of the enzymes $E$ and $F$, and the total amount of the substrates $S_i$, giving rise to equations
\begin{align*}
    x_1 + x_6 + x_6 = E_{tot},\; & x_2 + x_7 + x_9 = F_{tot},\; x_3 + x_4 + x_5+ x_6 + x_7 + x_8 + x_9 = S_{tot}.
\end{align*}
Here, $E_{tot}$, $F_{tot}$, and $S_{tot}$ represent the total amount of kinase, phosphatase, and substrate present in the system. 
These twelve equations in total describe the predicted behavior of the system. 
We are interested in the \struc{steady states} of this system, where the concentrations of species are stable over time. Mathematically, these are the solutions to the reaction network system of ODEs. This gives us a variety generated by the nine polynomials and three linear constraints, and for specific parameters, $\kappa_1, \ldots, \kappa_{12}, E_{tot}, F_{tot}, S_{tot}$, we are interested in the number of positive steady-state solutions $(x_1, \ldots, x_9)$. There is always at least one, and there can be as many as three \cite[Thm. 4]{WangSontag}. We say that a system is \struc{multistationary} if it has multiple steady states. This property is interesting because it implies dynamic behavior in the underlying biological processes, and we are interested in the conditions under which it may or may not occur.

In particular, we say that a choice of reaction rate constants $(\kappa_1, \ldots, \kappa_{12}) \in \R^{12}_{> 0}$ \struc{enables} multistationarity when there exist choices of $E_{tot},\, F_{tot}$ and $S_{tot}$ so that the steady state system has at least two positive solutions. When this occurs, we call the system \struc{multistationary}. When there is only one positive solution, we say the vector $\kappa$ \struc{precludes} multistationarity, or enables \struc{monostationarity}.

\subsection{Sums of Nonnegative Circuit Polynomials}
\label{section:SONC}
The problem of determining the monostationarity of a CRN corresponds to deciding whether a particular polynomial attains a nonnegative value for each point in the positive orthant. The usual approach is to find \struc{certificates of nonnegativity}, which are sufficient criteria that are effective to find and that prove the polynomial's nonnegativity. Given the specific polynomial arising in the context of dual phosphorylation, we consider \struc{Sums of Nonnegative Circuit Polynomials (SONC)} developed by Iliman and de Wolff \cite{SONCinitial} (see also Reznick \cite{ReznickAGIforms} and Pantea et al. \cite{CRN_SONC_Pantea}) as a certificate of nonnegativity. The SONC approach allows one to derive sufficient conditions for the nonnegativity of the polynomial in the case of symbolic coefficients that emerge from the reaction rate parameters in the CRN.

In this section, we introduce the basic notation and terminology concerning SONC polynomials we use throughout the article.
We consider the polynomial ring of \struc{real $n$-variate polynomials} $\R[x_1,\dots, x_n]$ and abbreviate it as \struc{$\R[\Vector{x}]$}.
A real $n$-variate polynomial $p$ in $\Vector{x}:=(x_1,\dots,x_n)$ in the multi-index notation can be expressed as $p(\Vector{x}) = 
 \sum_{\Vector{\alpha}\in \mathbb{N}^n}c_{\Vector{\alpha}}
 \Vector{x}^{\Vector{\alpha}} \in \R[\Vector{x}]$ such that $c_{\Vector{\alpha}}=0$ for all but finitely many \struc{exponents} $\Vector{\alpha}\in \mathbb{N}^n$. A term $\Vector{x}^{\Vector{\alpha}}$ is called \struc{monomial}.
 For every exponent $\Vector{\alpha}$ there is a \struc{coefficient} $c_{\Vector{\alpha}} \in \R$. 
 We consider the finite set $\struc{\text{supp}(p)} = \{\Vector{\alpha} \in \mathbb{N}^n \text { with } c_{\Vector{\alpha}} \neq 0\}$ and call it the \struc{support} of $p$. The \struc{convex hull} of a finite set $A \subset \R^n$ is the set of all convex combinations of the elements. We denote it as $\struc{\conv({A})} := \{\sum_{i=1}^{|A|}\lambda_i \Vector{\alpha_i}\text{ with }0\leq\lambda_i\leq1,\sum_{i=1}^{|A|}\lambda_i=1,\;\Vector{\alpha_i} \in A\text{ for } 1\leq i\leq |A|\}$. 
 The \struc{Newton polytope} $\mathcal{N}(p) \subset \N^n$ of polynomial $p$ in $n$ variables is the convex hull of the support $\text{conv}(\text{supp}(p))$.
A polynomial $p(\Vector{x})$ is \struc{nonnegative} if $p(\Vector{x}) \geq 0$ for all $x\in \mathbb{R}^n$.
In this article, we are only interested in the case if $p(\Vector{x})\geq 0$ for $x\in \mathbb{R}_{>0}^n$ on the positive orthant.

We now define \struc{circuit polynomials} that are crucial in this article and Feliu et al. \cite{DualPhosphorSONC} for certifying the nonnegativity of a polynomial via a \struc{decomposition} into \struc{sums of nonnegative circuit polynomials (SONC)} \cite{SONCinitial} (see also \cite{ReznickAGIforms, CRN_SONC_Pantea}). Consider a polynomial $p \in \mathbb{R}[\Vector{x}]$ of the form
\begin{align*}
    p(\Vector{x}) = c_{\Vector{\beta}}\Vector{x}^{\Vector{\beta}} + \sum_{i = 0}^r c_{\Vector{\alpha(i)}}\Vector{x}^{\Vector{\alpha(i)}},
\end{align*}
with $r\leq n$, coefficients $c_{\Vector{\alpha(i)}}\in \mathbb{R}_{>0}$ and $c_{\Vector{\beta}} \in \mathbb{R}$, and exponents $\Vector{\alpha(i)},\;\Vector{\beta} \in \mathbb{N}^n$.
The polynomial $p$ is a circuit polynomial if the Newton polytope $\mathcal{N}(p)$ is a simplex with vertices $\Vector{\alpha(0)},\dots,\Vector{\alpha(r)}$ and $\Vector{\beta}$ is contained in its relative interior. Note that this definition differs slightly from the original definition in Iliman and de Wolff \cite{SONCinitial} where all vertices are required to be even, i.e.\ $\Vector{\alpha}(i) \subseteq (2\mathbb{N})^n$ for all $\; 0\leq i \leq r$. It is a necessary condition for a polynomial $p(\Vector{x})$ to be nonnegative for $\Vector{x}$ in $\mathbb{R}^n$ that the vertices of $\mathcal{N}(p)$ are even \cite[p.365]{ReznickPSDForms}. Here, because we restrict $\Vector{x}$ to $\mathbb{R}_{>0}^n$, we follow the definition from Feliu et al. \cite{DualPhosphorSONC} and allow uneven vertices $\Vector{\alpha}(i) \subseteq \mathbb{N}^n$ for all $\; 0\leq i \leq r$.

For an intuition on circuit polynomials we give an example:
\begin{eg}
\label{exampleSONC}
    Consider $p(x,y) = x^4y^2 + x^2 + y - cx^2y$, with one variable coefficient $c \in \R,\; c\neq 0$.
    The support $\text{supp}(p)$ is the set $\{\Vector{\alpha(0)} = (4,2),\, \Vector{\alpha(1)} = (2,0),\, \Vector{\alpha(2)} = (0,1),\, \Vector{\beta} = (2,1)\}$ and the Newton polytope $\mathcal{N}(p)$ is a simplex, see Figure \ref{Fig:Ex:NewtonPolytope}. We observe that $p$ is a circuit polynomial, as the only exponent that potentially has a negative coefficient $-c$ is in the interior of $\mathcal{N}(p)$.
\end{eg}
\begin{figure}[h!]
\begin{center}
\includegraphics[width=.31\linewidth]{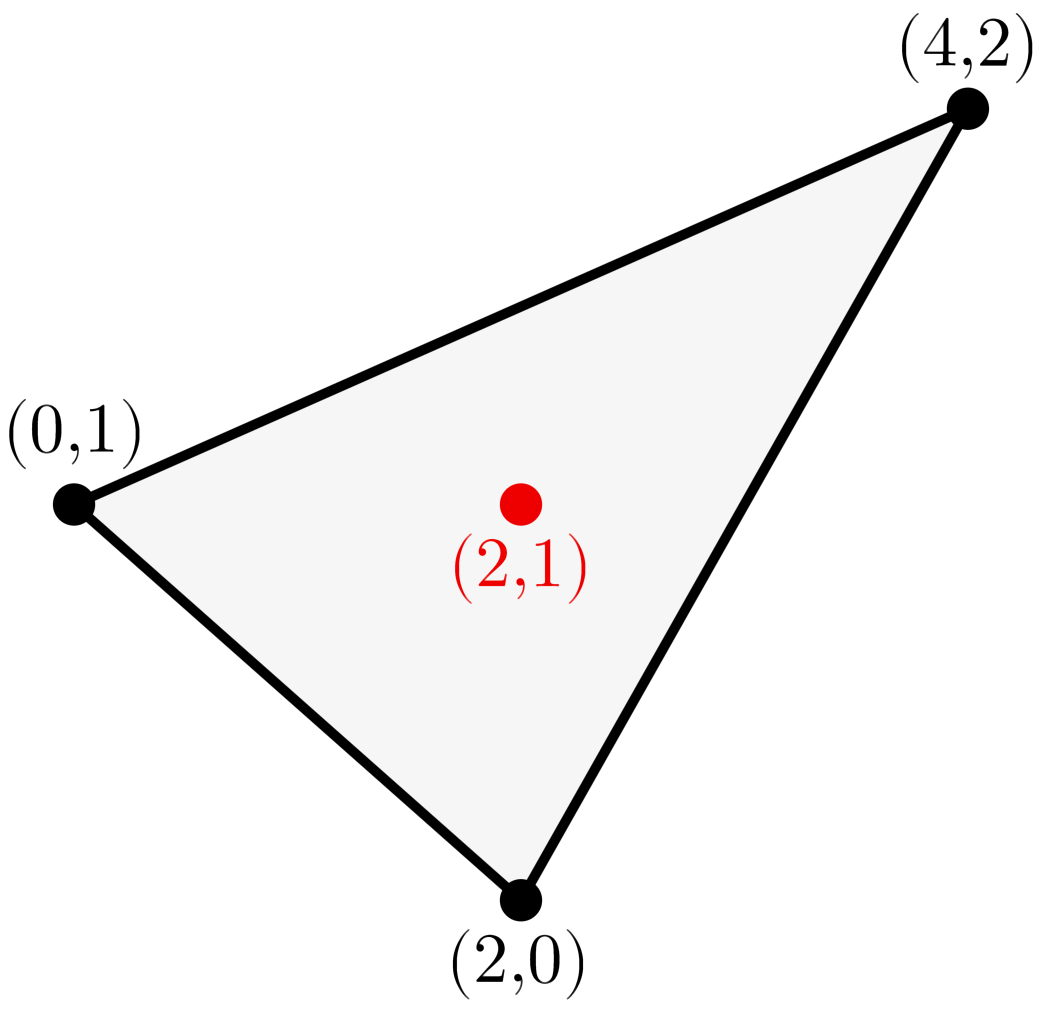}
\end{center}
\caption{Newton polytope of the circuit polynomial defined in Example \ref{exampleSONC}. 
}\label{Fig:Ex:NewtonPolytope}
\end{figure}
Every circuit polynomial $p$ has an associated \struc{circuit number} $\struc{\Theta_p}$ that allows to determine the nonnegativity of $p$. 
The circuit number is defined as
\begin{equation}
\label{eq:circuit-number}
\Theta_p = \prod_{i=0}^{r}\Bigg(\frac{c_{\Vector{\alpha(i)}}}{\lambda_i}\Bigg)^{\lambda_{i}}
\end{equation}
with $\lambda_0,\dots,\lambda_r$ denoting the \struc{barycentric coordinates} of $\Vector{\beta}$ with respect to the vertices $\mbox{\Vector{\alpha(0)},\dots,$ $\Vector{\alpha(r)}}$. 
In other words, $\Vector{\beta} = \sum_{i=0}^r \lambda_i\Vector{\alpha(i)}$, with $\sum_{i=0}^r \lambda_i = 1$ and $0 < \lambda_i < 1$ for $i=0,\dots, r$.

The following theorem provides a criterion to check the nonnegativity of a circuit polynomial based on its circuit number. Again, due to restricting to the positive orthant, we slightly adjust the original version in Iliman and de Wolff \cite[Thm. 3.8]{SONCinitial} and follow Feliu et al. \cite[Thm. 2.8]{DualPhosphorSONC}:
\begin{thm}{\cite[Thm. 2.8]{DualPhosphorSONC}}
\label{theorem_circuit_number}
A circuit polynomial $p(\Vector{x})~=~ c_{\Vector{\beta}}\Vector{x}^{\Vector{\beta}}~+~ \sum_{i = 0}^r c_{\Vector{\alpha(i)}}\Vector{x}^{\Vector{\alpha(i)}}$ on the positive orthant is nonnegative if and only if
\begin{align}
    -c_{\Vector{\beta}} \leq \Theta_p.
\end{align}
\end{thm}
Returning to Example \ref{exampleSONC}, we can easily compute the circuit number $\Theta_p$ of $p$ by solving a system of linear equations for the barycentric coordinates: $\frac{1}{3}\Vector{\alpha(0)}+\frac{1}{3}\Vector{\alpha(1)}+\frac{1}{3}\Vector{\alpha(2)} = \Vector{\beta}$. Thus, we have $\Theta_p = (\frac{1}{\nicefrac{1}{3}})^{\nicefrac{1}{3}}(\frac{1}{\nicefrac{1}{3}})^{\nicefrac{1}{3}}(\frac{1}{\nicefrac{1}{3}})^{\nicefrac{1}{3}} = 3$. Using Theorem \ref{theorem_circuit_number}, we can now determine that the circuit polynomial $p$ is nonnegative if and only if $c \leq 3$.
This example suggests a potential approach for certifying the nonnegativity of polynomials with partially symbolic coefficients. Using the circuit number and Theorem \ref{theorem_circuit_number}, one can derive inequalities that symbolically characterize the nonnegativity of a circuit polynomial.
 We thus aim to show the nonnegativity of a general polynomial $f$ by finding a decomposition of $f$ into a sum of nonnegative circuit polynomials, i.e.\ finding a SONC decomposition of $f$. 
Note that the circuit polynomial from Example \ref{exampleSONC} has the same support as one of the circuit polynomials used in the following sections.

Following and extending the approach by Feliu et al. \cite{DualPhosphorSONC}, in this work we use the SONC polynomials and their circuit numbers as initially introduced by Iliman and de Wolff \cite{SONCinitial} to derive sufficient conditions for nonnegativity in the context of dual phosphorylation via decomposing a given polynomial into SONC. Alongside this perspective on SONC polynomials, numerous other works have advanced both the theoretical and computational aspects of polynomial optimization:
In the context of \struc{signomials} or \struc{exponential sums}, i.e. functions $f(\Vector{x}) = \sum_{\Vector{\alpha}\in A}c_{\Vector{\alpha}}e^{\langle \Vector{x},{\Vector{\alpha}}\rangle}$
with $\langle\, \text{--},\text{--}\,\rangle$ denoting the standard inner product, $\Vector{x} \in \mathbb{R}^n$ and the finite support $A\subset \mathbb{R}^n$, Chandrasekaran and Shah \cite{Chandrasekaran:Shah:SAGE} have introduced a certificate of nonnegativity for signomials via \struc{sums of arithmetic and geometric mean exponentials (SAGE)}.
The general idea is to decompose them into \struc{AGE functions} that only contain one term with a negative coefficient and are nonnegative by the arithmetic-geometric-mean inequality.
 Wang \cite{Wang} has shown for polynomials, i.e. $A\subset \mathbb{N}^n$, that if we consider the same support sets $A$ then there exists a SONC decomposition if and only if there exists a SAGE decomposition. Chandrasekaran, Murray and Wiermann \cite{Chandrasekaran:Murray:Wiermann} proved this result for signomials.

In practice, the SONC and SAGE decompositions can be used in polynomial optimization to compute a lower bound to a minimum of a polynomial.
It is a priori not clear how to choose the circuit polynomials involved in the decomposition to obtain the optimal lower bound, especially in the case of symbolic coefficients.
In the nonparametrized case there exists algorithms to compute the decompositions:
In the SAGE approach, finding a SAGE decomposition can be via \struc{relative entropy programming (REP)} providing an optimal lower bound \cite{Chandrasekaran:Shah:SAGE, Chandrasekaran:Murray:Wiermann}. 
Dressler, Iliman and de Wolff developed \struc{geometric programming} \cite{Iliman:deWolff:GP, constrainedoptimizationapproach} and REP \cite{Dressler:Iliman:deWolff:Positivstellensatz} approaches to solve polynomial optimization problems via SONC decompositions. 
While this approach gives a lower bound on the minimal value of a polynomial, one has to make prior assumptions about the involved circuit polynomials. 
Papp \cite{dualityofnonnegativecircuits} proposes an algorithm via column generation to choose a the circuit polynomials for an optimal decomposition.
The theoretical description by Averkov \cite{Gennadiy} as well as the algorithmic approach by Magron and Wang \cite{MAGRON2023346} address SONC decompositions via \struc{Second Order Cone Programming (SOCP)}.

\subsection{Monostationarity and Circuit Polynomials}
\label{section:dual-phosphorylation}
In what follows, we briefly summarize the preliminaries and key aspects from Feliu et al. \cite[§3.2]{DualPhosphorSONC} in which the authors classify the region of monostationarity in dual phosphorylation via circuit polynomials. In the upcoming sections, we extend their approach empirically.

Prior to Feliu et al. \cite{DualPhosphorSONC}, conditions on the multistationarity of dual phosphorylation were given by Conradi and Mincheva et al. \cite{ConradiMincheva2014, ConradiFeliuMinchevaWiuf2017}. Recall the reaction rate constants of the CRN in Figure \ref{fig:CRN} are given as the parameters \begin{align} (\kappa_1, \kappa_2, \kappa_3, \kappa_4, \kappa_5, \kappa_6, \kappa_7, \kappa_8, \kappa_9, \kappa_{10}, \kappa_{11}, \kappa_{12}) =: \Vector{\kappa} \in \mathbb{R}_{>0}^{12}. 
\end{align} Additionally, we consider the \struc{Michaelis-Menten constants}
    \begin{align}K_1 = \frac{\kappa_2+\kappa_3}{\kappa_1}, && K_2 = \frac{\kappa_5+\kappa_6}{\kappa_4}, && K_3 = \frac{\kappa_8 + \kappa_9}{\kappa_7}, && K_4 = \frac{\kappa_{11}+\kappa_{12}}{\kappa_{10}} \in \mathbb{R}_{>0} \label{eq:Michaelis-Menten}
    \end{align}
to define the vector 
\begin{align}
    \Vector{\eta} := (K_1, K_2, K_3, K_4, \kappa_3, \kappa_6, \kappa_9, \kappa_{12}) \in \mathbb{R}_{>0}^{8} \label{eq:eta}
\end{align}
as the \struc{image $\pi({\Vector{\kappa}})$} of a surjective and continuous map $\pi: \mathbb{R}_{>0}^{12} \rightarrow \mathbb{R}_{>0}^8$.
See Conradi et al. or Yürük \cite{ConradiFeliuMinchevaWiuf2017, dissYuruk} for a detailed description how to derive the the following polynomial $p_{\Vector{\eta}}(\Vector{x})$ given in \eqref{eq:polynomial_eta} with variables $\Vector{x}=(x_1,x_2,x_3)$ and coefficients based on $\Vector{\eta}$ from the CRN. 
We have:
\begin{align}
\begin{split}
    p_{\Vector{\eta}}(\Vector{x}) &= K_2\kappa_3(\struc{\kappa_3\kappa_{12} - \kappa_6\kappa_9})\Big(K_2K_4\kappa_3\kappa_9x_1^4x_3^2 + K_1K_3\kappa_6\kappa_{12}(x_1^3x_2^2x_3 + x_1^2x_2^3x_3+ x_1^2x_2^2x_3^2)\\
    &+K_2K_3\kappa_3\kappa_{12}x_1^3x_2x_3^2\Big) + K_1K_2K_3\kappa_3\kappa_6\kappa_{12}\Big(\struc{(K_2+K_3)\kappa_3\kappa_{12} - (K_1+K_4)\kappa_6\kappa_9}\Big)x_1^2x_2^2x_3\\
    &+K_1\kappa_6\Big(K_2^2K_4\kappa_3^2\kappa_9^2x_1^4x_3 + 2K_2K_3K_4\kappa_3^2\kappa_9\kappa_{12}x_1^3x_2x_3 + K_1K_2K_3\kappa_3\kappa_6\kappa_{12}(\kappa_9 + \kappa_{12})x_1^2x_2^3\\
    &+K_1K_2K_3K_4\kappa_3\kappa_6\kappa_9\kappa_{12}x_1^2x_2^2 + K_1K_3^2\kappa_6\kappa_{12}^2(\kappa_3 + \kappa_6)x_1x_2^4 + 2K_1K_2K_3\kappa_3\kappa_6\kappa_{12}^2x_1x_2^3x_3\\
    &+ K_1K_2K_3^2\kappa_3\kappa_6\kappa_{12}^2x_1x_2^3 + K_1K_3^2\kappa_6^2\kappa_{12}^2x_2^4x_3 + K_1^2K_3^2\kappa_6^2\kappa_{12}^2x_2^4\Big).\label{eq:polynomial_eta}
\end{split}
\end{align}

The following Lemma \ref{lem:mono_nonnegativity} (cf. \cite[Prop. 2.2]{DualPhosphorSONC}) gives a criterion to determine the monostationarity of the CRN based on the signs of values that the polynomial $p_{\Vector{\eta}}(\Vector{x})$ attains. For a fixed set of parameters in $\Vector{\eta}$ as given in \eqref{eq:eta}, the following holds:
\begin{lemma}[{\cite[Prop. 2.2]{DualPhosphorSONC}}, {\cite{ConradiMincheva2014, ConradiFeliuMinchevaWiuf2017}}]
\label{lem:mono_nonnegativity}
Consider the polynomial $p_{\Vector{\eta}}(\Vector{x})$ from \eqref{eq:polynomial_eta}:
    If $p_{\Vector{\eta}}(\Vector{x}) > 0$ for all $x_1, x_2, x_3 > 0$, then any set of parameters $\Vector{\kappa}$ of the preimage $\pi^{-1}({\Vector{\eta})}$ does not enable multistationarity, i.e.\ the CRN is monostationary.
\end{lemma}
As shown in \cite[Prop 2.11]{DualPhosphorSONC}, in the setting we consider, the polynomial $p_{\Vector{\eta}}(\Vector{x})$ is either positive 
or attains negative values. 
That means if we find the parameters in $\Vector{\eta}$ such that $p_{\Vector{\eta}}(\Vector{x}) \geq 0$ for all $x_1,x_2,x_3 \in \mathbb{R}_{>0}$, the CRN is monostationary by Lemma \ref{lem:mono_nonnegativity}.
Conversely, if the polynomial attains negative values, multistationarity is enabled, see \cite[Prop. 2.2]{DualPhosphorSONC} for a detailed version. 

Importantly, Lemma \ref{lem:mono_nonnegativity} translates deciding the monostationarity of dual phosphorylation to deciding nonnegativity of the polynomial $p_{\Vector{\eta}}(\Vector{x})$.
We now tackle the question of nonnegativity of $p_{\Vector{\eta}}(\Vector{x})$ by using SONC decompositions that we introduced in Section \ref{section:SONC}.

Recall the definition of $p_{\Vector{\eta}}(\Vector{x})$ in \eqref{eq:polynomial_eta}.
Since we only consider $x_1, x_2, x_3$ in $\mathbb{R}_{>0}$, we find that $p_{\Vector{\eta}}(\Vector{x})$ cannot attain negative values if all coefficients of $p_{\Vector{\eta}}(\Vector{x})$ are nonnegative. 
Due to the positivity of all parameters in $\Vector{\eta}$, there are in total only six monomials of $p_{\Vector{\eta}}(\Vector{x})$ that can possibly have a negative coefficient. Those are determined by multiples of either one of the functions 
\begin{align}
\label{eq:def-a-and-b}
    \struc{a(\Vector{\eta}):= \kappa_3\kappa_{12} - \kappa_6\kappa_9} &&
    \struc{b(\Vector{\eta}):= (K_2 + K_3)\kappa_3\kappa_{12} - (K_1 + K_4)\kappa_6\kappa_9}.
\end{align}

Therefore, we now list all possible cases that determine the signs which $a(\Vector{\eta})$ and $b(\Vector{\eta})$ can attain. While cases 1) and 2) have already been discussed by Conradi and Mincheva \cite{ConradiMincheva2014}, all cases are considered by Feliu et al. \cite{DualPhosphorSONC}:
\begin{enumerate}[label={\arabic*)}]
    \item $a(\Vector{\eta}) \geq 0$ and $b(\Vector{\eta})\geq 0$: All coefficients of $p_{\Vector{\eta}}(\Vector{x})$ are nonnegative, thus the CRN is monostationary by Lemma \ref{lem:mono_nonnegativity}.
    \item $a(\Vector{\eta}) < 0$: The CRN enables multistationarity, cf. \cite[Prop.~2.11]{DualPhosphorSONC}.
    \item $a(\Vector{\eta})= 0$ and $b(\Vector{\eta}) < 0$: This case is the subject of \cite[Thm. 3.1(ii)]{DualPhosphorSONC}, where the authors give a necessary and sufficient condition for when multistationarity is enabled.
    \item $a(\Vector{\eta}) > 0$ and $b(\Vector{\eta}) < 0$: This case is the subject of \cite[Thm. 3.5]{DualPhosphorSONC}, the authors give a sufficient condition for monotationarity through the use of SONC certificates.
\end{enumerate}
We take a closer look at case 4. The empirical analysis that we perform in \Cref{section:empirical data} along with the remaining upcoming sections focus on this case.
Hence, from now on we assume that {$a(\Vector{\eta})>0$} and {$b(\Vector{\eta})<0$}. We note that the unique monomial corresponding to the coefficient which is a multiple of $b(\Vector{\eta})$ is $x_1^2x_2^2x_3$.
The exponent vector of this monomial is $(2,2,1)$.
Now let $H$ be the face of the Newton polytope $\mathcal{N}(p_{\Vector{\eta}})$ of $p_{\Vector{\eta}}(\Vector{x})$ that contains $(2,2,1)$, see Figure \ref{fig:Newton-polytope}. This two-dimensional face is a hexagon whose six vertices correspond to exponent vectors with positive coefficients. Among the points that lie in the interior of $H$, only $(2,2,1)$ has a negative coefficient.

\begin{figure}[h!]
    \centering
    \includegraphics[width=0.48\linewidth]{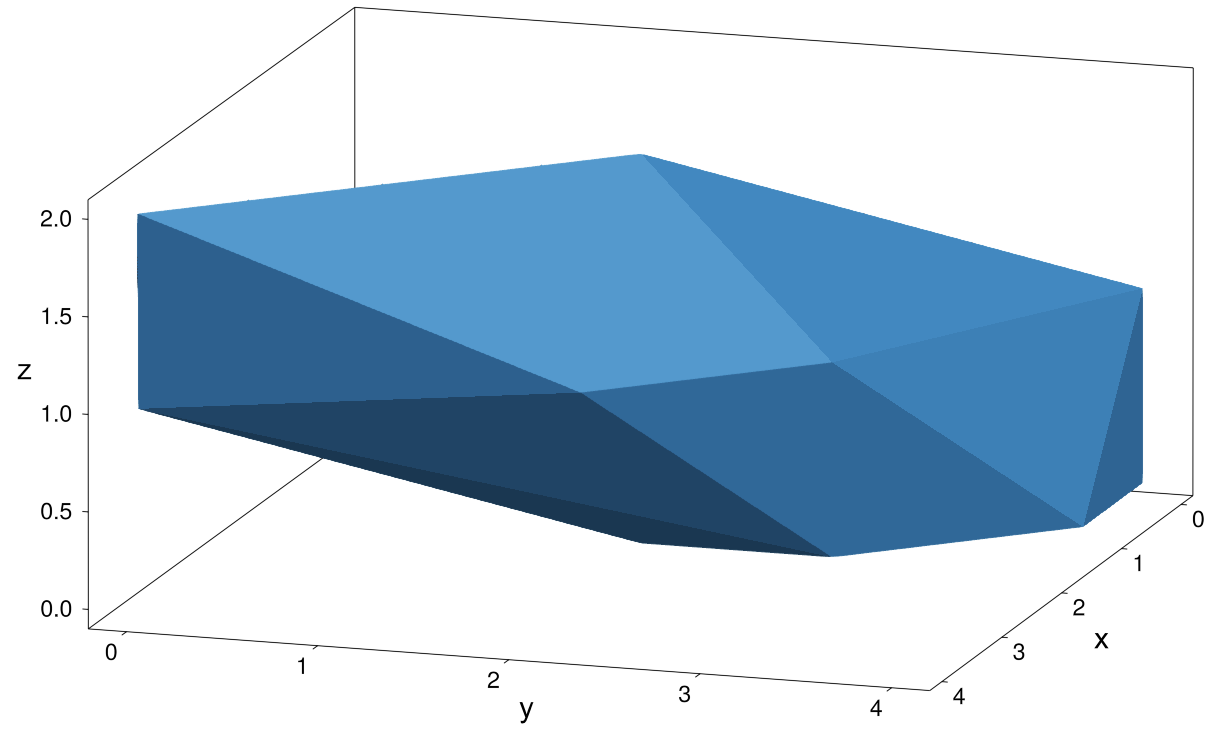}
    \hfill
        \includegraphics[width=0.48\linewidth]{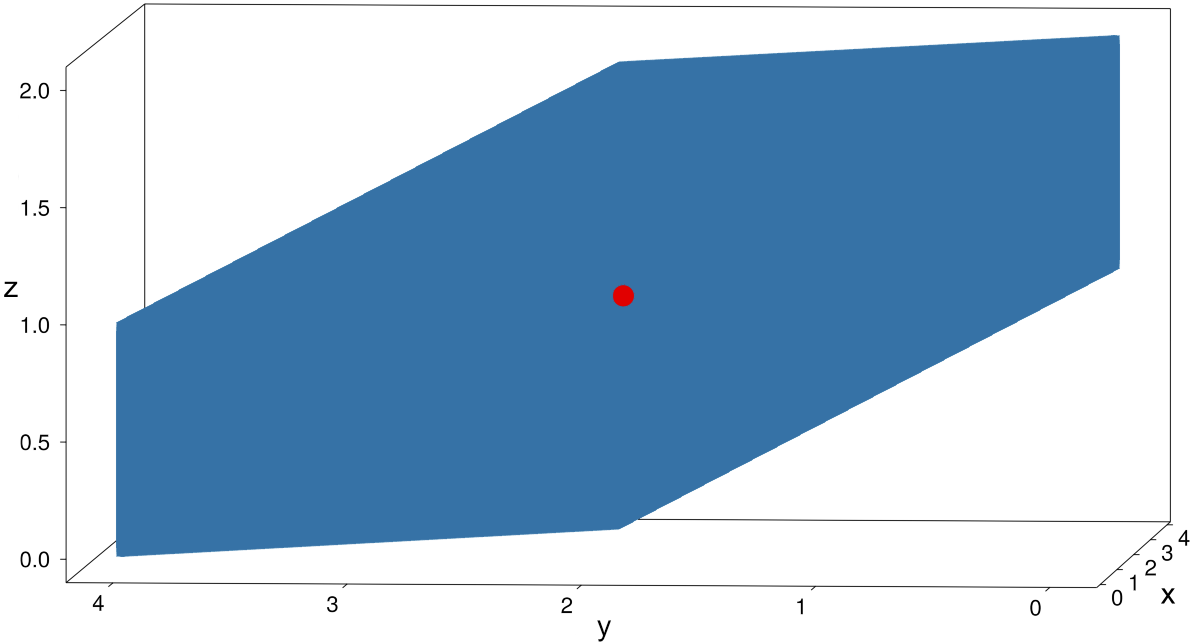}
    \caption{ The Newton polytope with its flat, hexagonal side (right) from two perspectives. 
    The negative point, which potentially leads to the polynomial attaining negative values, is marked in red.}
    \label{fig:Newton-polytope}
\end{figure}

By \cite[Prop. 2.11]{DualPhosphorSONC}, given that $a(\Vector{\eta})\geq 0$, the system is monostationary if and only if the polynomial that is supported on $H$ is nonnegative over $\mathbb{R}_{>0}^3$.
This means, in order to determine monostationarity of the system, we may restrict $p_{\Vector{\eta}}(\Vector{x})$ to the monomials with exponents whose exponent vectors belong to $H$. Denote by $p_{\Vector{\eta}, H}(\Vector{x})$ this restriction of $p_{\Vector{\eta}}(\Vector{x})$.
Moreover, we may use the fact that $p_{\Vector{\eta},H}(\Vector{x})$ is a \struc{homogeneous} polynomial in the variables $x_1$ and $x_2$, so the sign of the values stays the same if we set $x_2=1$ \cite{DualPhosphorSONC}.
 In what follows, we consider
\begin{align}
\begin{split}
    p_{\Vector{\eta},H}(x_1, x_3) &= K_2\kappa_3(\struc{\kappa_3\kappa_{12} - \kappa_6\kappa_9})\Big(K_2K_4\kappa_3\kappa_9x_1^4x_3^2 + K_1K_3\kappa_6\kappa_{12}x_1^2x_3^2+K_2K_3\kappa_3\kappa_{12}x_1^3x_3^2\Big)\\ &+ K_1K_2K_3\kappa_3\kappa_6\kappa_{12}\Big(\struc{(K_2+K_3)\kappa_3\kappa_{12} - (K_1+K_4)\kappa_6\kappa_9}\Big)x_1^2x_3\\
    &+K_1\kappa_6\Big(K_2^2K_4\kappa_3^2\kappa_9^2x_1^4x_3 + 2K_2K_3K_4\kappa_3^2\kappa_9\kappa_{12}x_1^3x_3\\
    &+K_1K_2K_3K_4\kappa_3\kappa_6\kappa_9\kappa_{12}x_1^2  + 2K_1K_2K_3\kappa_3\kappa_6\kappa_{12}^2x_1x_3\\
    &+ K_1K_2K_3^2\kappa_3\kappa_6\kappa_{12}^2x_1 + K_1K_3^2\kappa_6^2\kappa_{12}^2x_3 + K_1^2K_3^2\kappa_6^2\kappa_{12}^2\Big).\label{eq:polynomial_eta_H}
\end{split}
\end{align}
We infer the monostationarity of the CRN from the nonnegativity of $p_{\Vector{\eta},H}(x_1, x_3)$ for $x_1, x_3$ in $\mathbb{R}_{>0}$.
Note that mapping $x_2\mapsto 1$ in $p_{\Vector{\eta}, H}(\Vector{x})$ corresponds to projecting $H$ to the $x$-$z$-plane by setting the $y$-coordinate to be $0$.
We label the coordinates $(x_1, 0, x_3)=:(x_1,x_3)$ of the projected face $H$ as
\begin{align}
\label{eq:hexagon-coordinates}
\begin{split}
    \Vector{\Vector{\alpha_1}}:=(0,0), 
    ~\Vector{\Vector{\alpha_2}}:=(2,0),
    ~\Vector{\Vector{\alpha_3}}:=(4,1),
    ~\Vector{\Vector{\alpha_4}}:=(4,2), 
    ~\Vector{\Vector{\alpha_5}}:=(2,2),
    ~\Vector{\Vector{\alpha_6}}:=(0,1),\\
    \Vector{\Vector{\beta_1}}:= (1,0),
    ~\Vector{\Vector{\beta_2}}:=(3,2), 
    ~\Vector{\Vector{\iota_1}}:=(1,1),
    ~\Vector{\Vector{\iota_2}}:=(3,1)
    ~\text{ and }~\Vector{m}:=(2,1).
\end{split}
\end{align}
In Figure \ref{fig:hexagonalface_LABELS}(a) we display $H$ along with the labels from Equation \eqref{eq:hexagon-coordinates}. 
\begin{figure}[h!]
    \centering
    \includegraphics[width=0.3\linewidth]{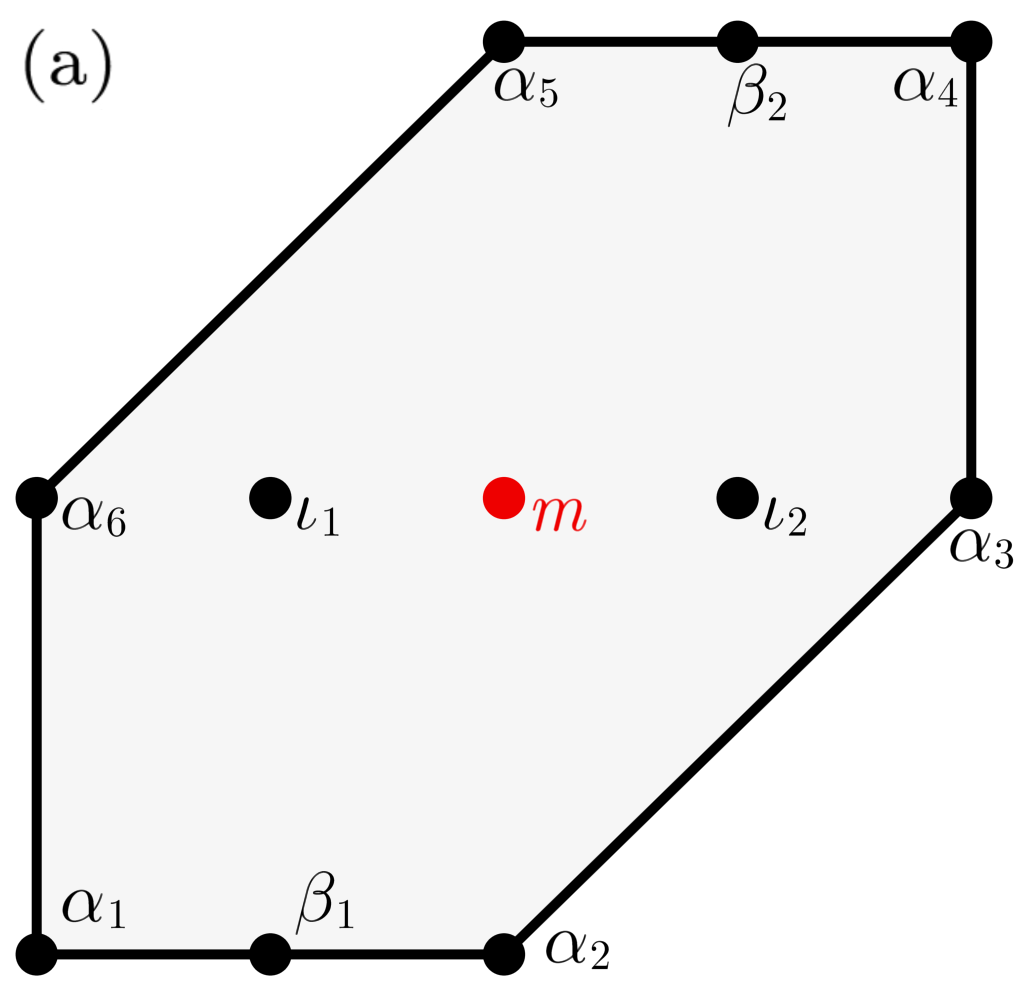}
    \hspace*{16mm}
    \includegraphics[width=0.3\linewidth]{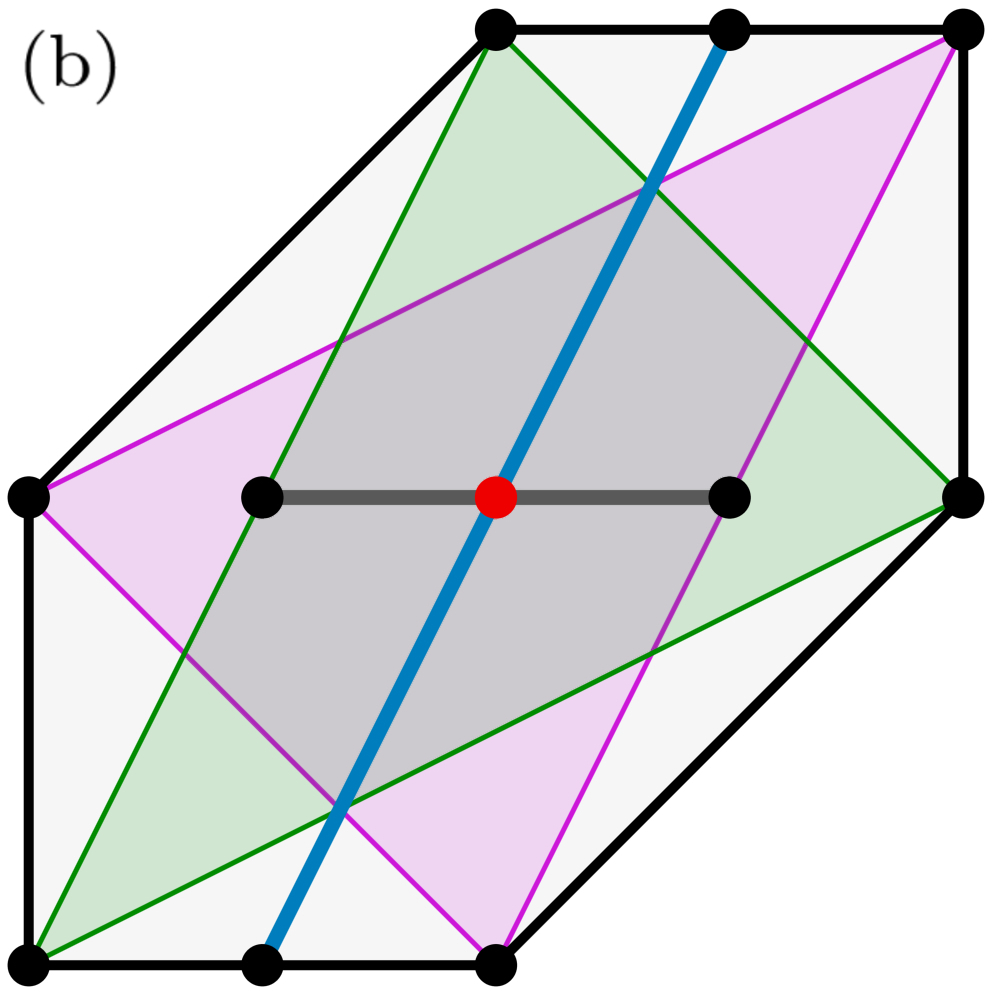}

    \caption{(a) Face $H$ of the Newton polytope of $p_{\Vector{\eta}}(\Vector{x})$ with labeled vertices. Note that the labeling of the hexagon's vertices is different from the original article \cite{DualPhosphorSONC}. (b) Specific vertex cover for $p_{\Vector{\eta},H}(x_1, x_3)$ used in Feliu et al. \cite{DualPhosphorSONC}, later referred to as cover $\mathcal{CC}(9)$.}
    \label{fig:hexagonalface_LABELS}
\end{figure}
We call $\Vector{m} = (2,1)$ the \struc{negative point} that corresponds to the exponent vector of the only monomial $x_1^2x_3$ of $p_{\Vector{\eta},H}(x_1, x_3)$ with a negative coefficient. This point is exactly obtained by projecting the vertex $(2,2,1)$.
The idea in Feliu et al. \cite[§3.2]{DualPhosphorSONC} to determine monostationarity via SONC is the following one:
\begin{enumerate}
    \item Decompose $p_{\Vector{\eta},H}(x_1, x_3)$ into circuit polynomials containing the negative point $\Vector{m}$ in the interior. 
    \item Symbolically calculate the circuit number of those circuit polynomials according to Theorem \ref{theorem_circuit_number}, in order to certify their nonnegativity in dependence of any $\Vector{\eta}$. 
    \item The monostationarity is a consequence of Lemma \ref{lem:mono_nonnegativity}. 
\end{enumerate} 
We briefly sketch this idea with an example:
\begin{eg}
\label{ex:one_circuit_in_hexagon}
Recall Example \ref{exampleSONC}. 
The circuit polynomial in Figure \ref{Fig:Ex:NewtonPolytope} contains the negative point $\Vector{m} = (2,1)$ in its interior, while the vertices are $\Vector{\alpha_2} = (2,0),~ \Vector{\alpha_6} = (0,1),~ \Vector{\alpha_4} = (4,2)$. 
Considering a subset of the monomials in $p_{\Vector{\eta},H}(x_1, x_3)$, we can find a circuit polynomial $f_{\Vector{\eta},H}(x_1,x_3)$ that also has the support $\{\Vector{m}, \Vector{\alpha_2}, \Vector{\alpha_6}, \Vector{\alpha_4}\}$. In particular, we have
\begin{eqnarray}
    f_{\Vector{\eta},H}(x_1,x_3)&=&{b(\Vector{\eta})}K_1K_2K_3\kappa_3\kappa_6\kappa_{12}x_1^2x_3
    +K_1^2K_2K_3K_4\kappa_3\kappa_6^2\kappa_9\kappa_{12}x_1^2
     + K_1^2K_3^2\kappa_6^3\kappa_{12}^2x_3\\
     &~&+\,{a(\Vector{\eta})}K_2^2K_4\kappa_3^2\kappa_9x_1^4x_3^2.\nonumber\label{eq:example_poly_circuit}
\end{eqnarray}
Calculating the corresponding circuit number $\Theta_{f}$, we have by Theorem \ref{theorem_circuit_number} that $f_{\Vector{\eta},H}(x_1,x_3)$ is nonnegative if and only if
\begin{align}
    -b(\Vector{\eta})K_1K_2K_3\kappa_3\kappa_6\kappa_{12} 
    ~\leq ~\Theta_{f} ~:= ~3K_1K_2K_3\kappa_3\kappa_6\kappa_{12}\sqrt[3]{K_1K_4^2\kappa_6^2\kappa_9^2a(\Vector{\eta})}.\label{ex:circuit_number_with_symbols}
\end{align}
Any set of parameters that satisfies inequality \eqref{ex:circuit_number_with_symbols} leads to the nonnegativity of $f_{\Vector{\eta},H}(x_1,x_3)$. Since all other monomials in $p_{\Vector{\eta},H}(x_1, x_3)$ except for the one corresponding to $\Vector{m}$ have positive coefficients, they cannot cause negative values of $p_{\Vector{\eta},H}(x_1, x_3)$. Therefore, if \eqref{ex:circuit_number_with_symbols} holds, we can certify nonegativity of $p_{\Vector{\eta},H}(x_1, x_3)$ and monostationarity via the circuit polynomial $f_{\Vector{\eta},H}(x_1,x_3)$.
\end{eg}
In Example \ref{ex:one_circuit_in_hexagon} we only made use of one circuit polynomial to certify the nonnegativity of $p_{\Vector{\eta},H}(x_1, x_3)$ and did not involve the other points in the hexagon.
However, we can enlarge the right-hand side of inequality \eqref{ex:circuit_number_with_symbols} by using multiple circuit polynomials consisting of different supports, each having $\Vector{m}$ in their interior.
Hence the idea is to construct a SONC decomposition of $p_{\Vector{\eta},H}(x_1, x_3)$.
See Figure \ref{fig:hexagonalface_LABELS}(b), for a possible selection of four circuit polynomials only sharing $\Vector{m}$ in their support, that was used in the monostationarity analysis of Feliu et al. \cite{DualPhosphorSONC}.
Since each of these circuit polynomials $f_{\Vector{\eta},H}$, $g_{\Vector{\eta},H}$, $h_{\Vector{\eta},H}$, $q_{\Vector{\eta},H}$ has their individual symbolic circuit number $\Theta_{f}$, $\Theta_{g}$, $\Theta_{h}$, $\Theta_{q}$ that each gives a condition on the negative coefficient $b(\Vector{\eta})K_1K_2K_3\kappa_3\kappa_6\kappa_{12}$ of $\Vector{m}$, we certify nonnegativity of $p_{\Vector{\eta},H}(x_1, x_3)$ by Theorem \ref{theorem_circuit_number} if
\begin{align}
    -b(\Vector{\eta})K_1K_2K_3\kappa_3\kappa_6\kappa_{12} ~\leq~  \Theta_{f} +\Theta_{g} + \Theta_{h} + \Theta_{q}.\label{eq:circuit_cover_inequ}
\end{align}
Confer \cite[Thm. 3.5]{DualPhosphorSONC} for a simplified expression of \eqref{eq:circuit_cover_inequ}.

We refer to a decomposition of $p_{\Vector{\eta},H}(x_1, x_3)$ into circuit polynomials that each contain $\Vector{m}$ as the negative point in the interior as a \struc{circuit cover} or simply \struc{cover}. This decomposition is called a \struc{pure cover} if each point in the face $H\setminus \{\Vector{m}\}$ is used in exactly one circuit polynomial, see for example Figure \ref{fig:hexagonalface_LABELS}(b). If a point in $H\setminus \{\Vector{m}\}$ is part of more than one circuit polynomial, we define this decomposition as a \struc{weighted cover}. The name derives from the idea that if a point $\Vector{\alpha}$ is used more than once, the corresponding coefficient $c_{\Vector{\alpha}}$ has to contribute to more than one circuit polynomial and its circuit number. Therefore, a weighting, which is a convex combination, determines the fractions of $c_{\Vector{\alpha}}$ that contribute to each involved circuit polynomial. 
We say that we \struc{split} the coefficient $c_{\Vector{\alpha}}$ among the circuit polynomials, see Section \ref{sec:weighted}. 
Note that we do not consider any decompositions into circuit polynomials that do not use all points $H\setminus \{\Vector{m}\}$.
This would only yield smaller sums of circuit numbers.
As we aim to classify an as large as possible region of monostationarity, a larger sum of circuit numbers is beneficial by Theorem \ref{theorem_circuit_number}.

\section{Results}
\label{section:empirical data}
In the previous sections, we have introduced theoretical results for the nonnegativity of polynomials and how they can be used to assess the monostationarity of CRNs. By applying these tools to computer-generated experimental data, we gain novel insights into the dual phosphorylation network. As laid out in Section \ref{section:dual-phosphorylation}, sufficient conditions for enabling and precluding multistationarity are available. The region of monostationarity of the dual phosphorylation network is already characterized in most cases (cf. \Cref{section:dual-phosphorylation}). 
Still, there exists a hexagonal face $H$, where the space of reaction rate constants has not yet been fully classified (see Figure \ref{fig:hexagonalface_LABELS}). 
If the corresponding polynomial $p_{\Vector{\eta},H}$ is nonnegative there always exists a SONC decomposition for $p_{\Vector{\eta},H}$ and vice versa (see \cite[Thm. 3.9]{Wang}).
However, an a priori chosen cover does not necessarily completely classify the polynomial's region of nonnegativity.
In that case, a given SONC certificate only provides a sufficient criterion on the nonnegativity of $p_{\Vector{\eta},H}$ (see Equation (\ref{eq:circuit_cover_inequ})), not a necessary one. 

To answer the question of how the different SONC decompositions of $H$ compare to each other, we intend to use a Monte-Carlo approach to estimate the sizes of the associated regions of monostationarity. This allows us to investigate multiple invariants such as the number of points that are contained in a unique cover or the pure cover that occupies the largest region of monostationarity.

\subsection{Classification of Pure Covers on the Hexagon}
\label{section:classification}
Given a Newton polytope of a polynomial with vertices corresponding to positive exponents and negative points in the interior, it is in general not clear how to decompose it into circuit polynomials.
The numerical geometric programming approach by Dressler et al. \cite{constrainedoptimizationapproach} first chooses a non-unique triangulation of the Newton Polytope in order to lower bound the polynomial's minimum.
Likewise in the setting of symbolic coefficients, the chosen pure cover in Feliu et al. \cite{DualPhosphorSONC} is not unique and was selected for symmetry and simplicity reasons, see \cite[Remark 3.6]{DualPhosphorSONC}. 
While it provides a sufficient condition on the parameters in $\Vector{\eta}$ to lead to monostationarity, it does not fully include the entire parameter space of the monostationarity region.
In this section, we classify all possible pure covers to then empirically compare their quality based on the parameter space that leads to monostationarity to the one from Feliu et al. \cite{DualPhosphorSONC} and among each other in the upcoming sections.

In this ad-hoc classification, we refer heavily to the labeling in Figure \ref{fig:hexagonalface_LABELS}. In particular, we justify that Figure \ref{fig:all16configurations} contains all 16 valid pure covers of the polytope. In the following, we denote the pure cover with the label $i$ by \struc{$\mathcal{CC}(i)$}.
\begin{thm}
\label{prop:16-circuit-covers}
    There are 16 pure covers of the Newton polytope of $p_{\Vector{\eta}, H}(x_1, x_3)$ as given in \eqref{eq:polynomial_eta_H}, that are enumerated in Figure \ref{fig:all16configurations}.
\end{thm}
\begin{figure}[h!]
    \centering 
    \includegraphics[width=0.243\linewidth]{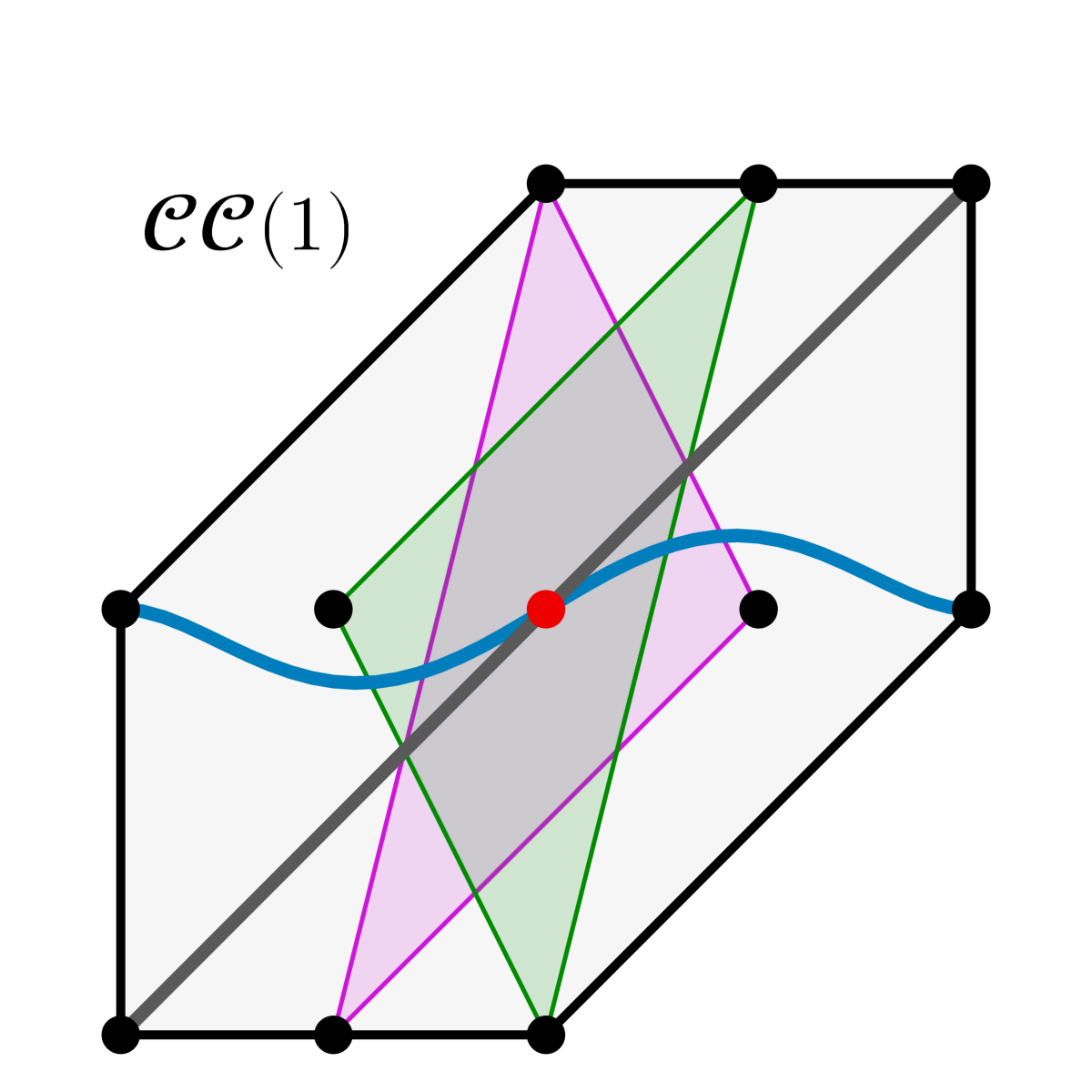}
    \hfill 
    \includegraphics[width=0.243\linewidth]{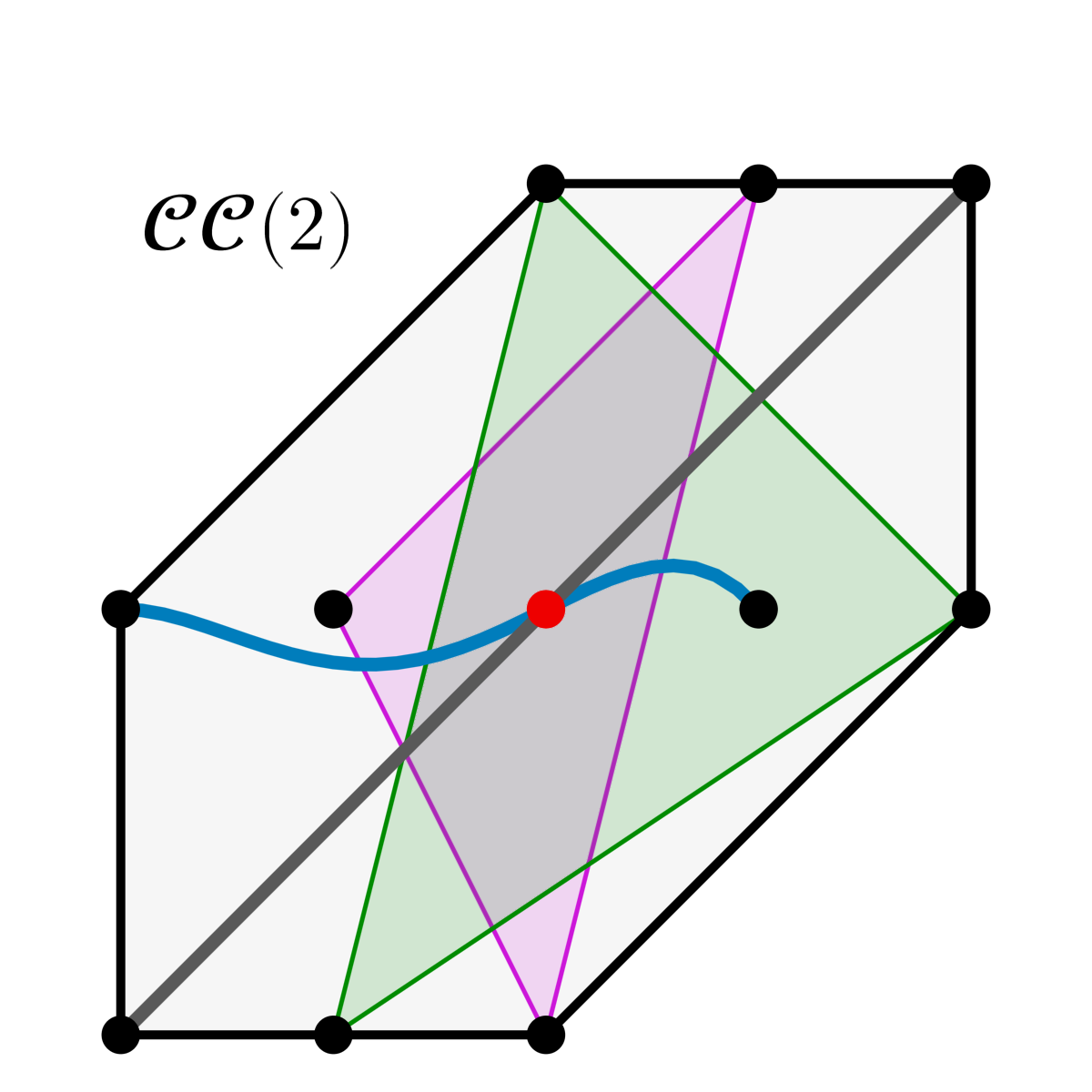}
    \hfill 
    \includegraphics[width=0.243\linewidth]{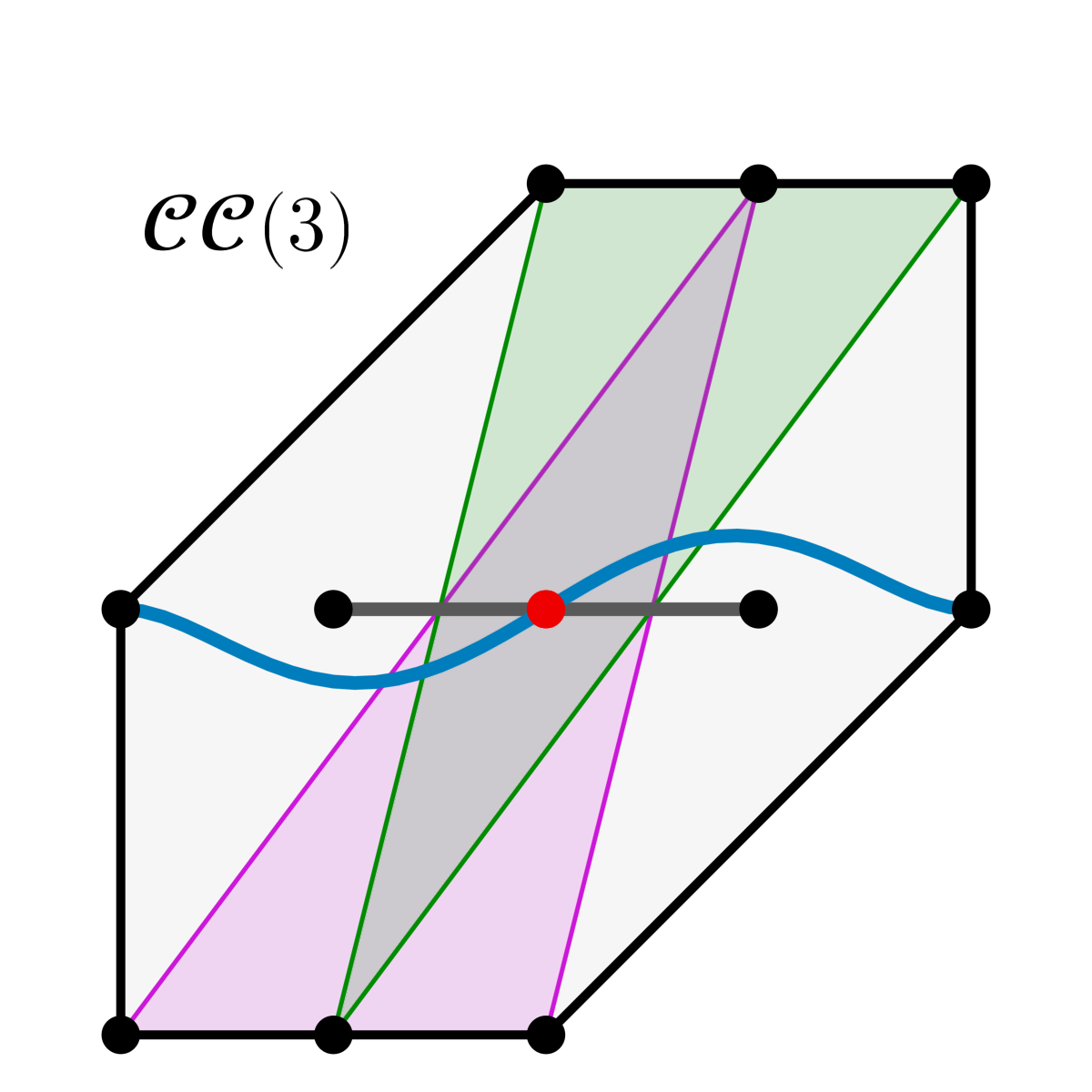}
    \hfill 
    \includegraphics[width=0.243\linewidth]{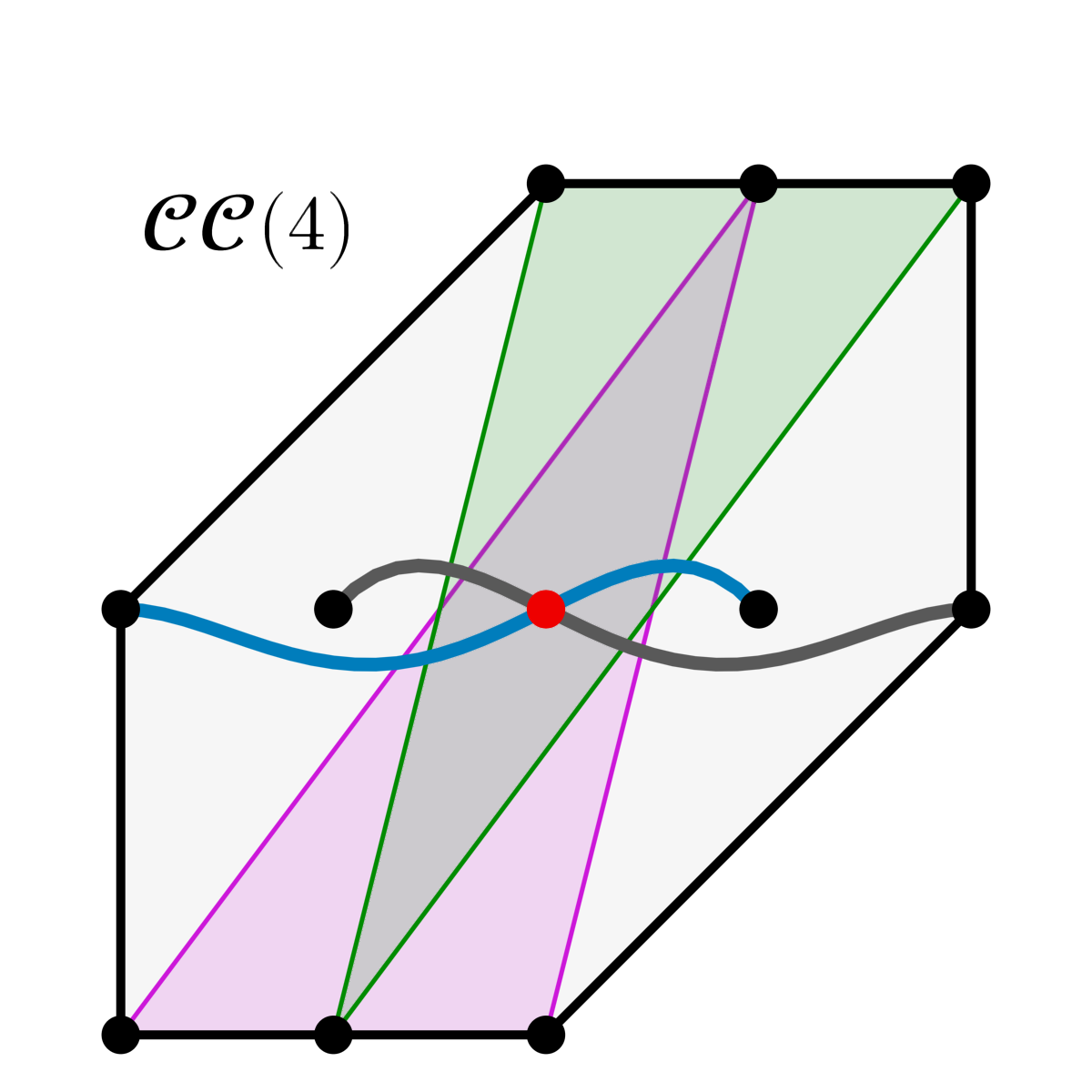}
    
    \includegraphics[width=0.243\linewidth]{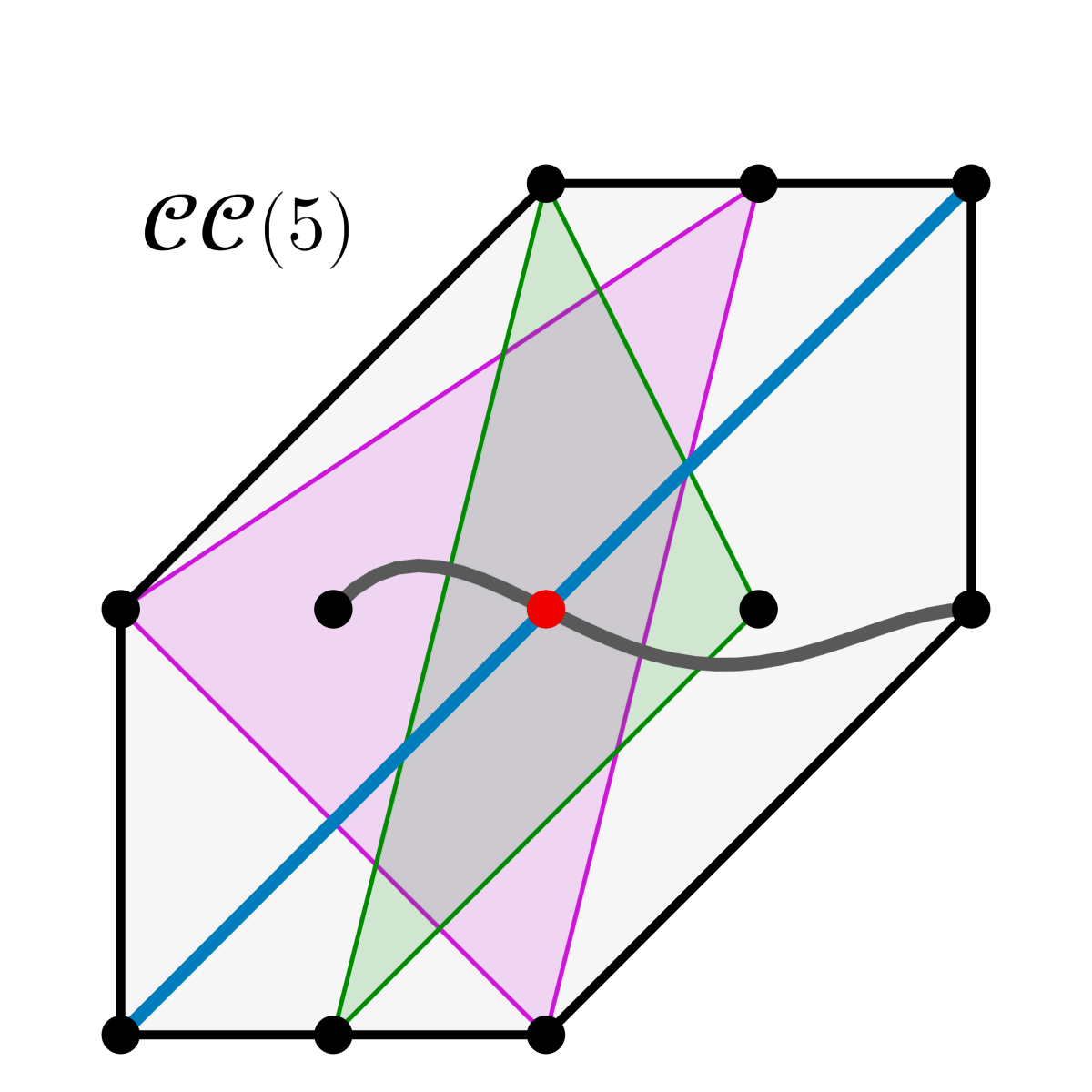}
    \hfill
    \includegraphics[width=0.243\linewidth]{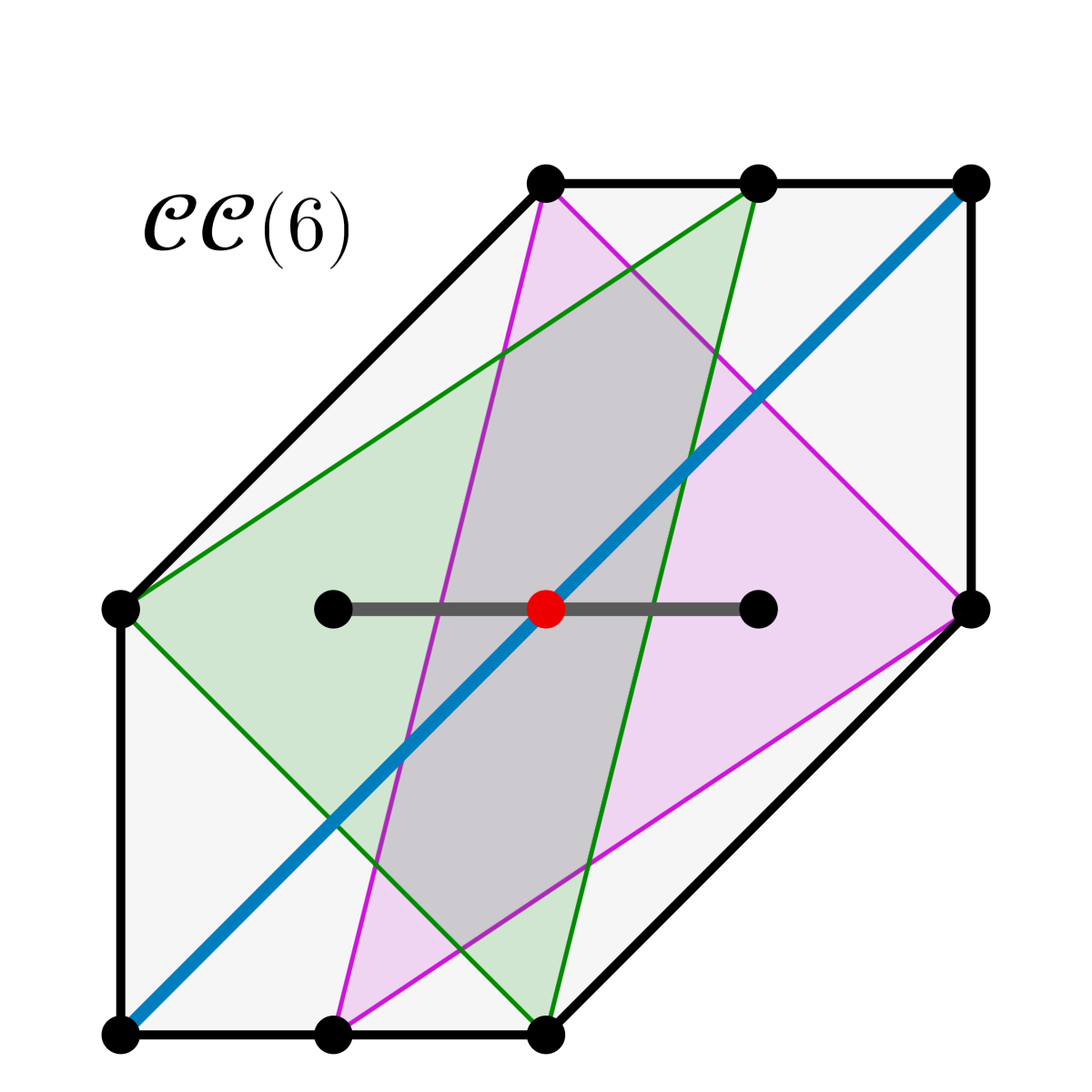}
    \hfill
    \includegraphics[width=0.243\linewidth]{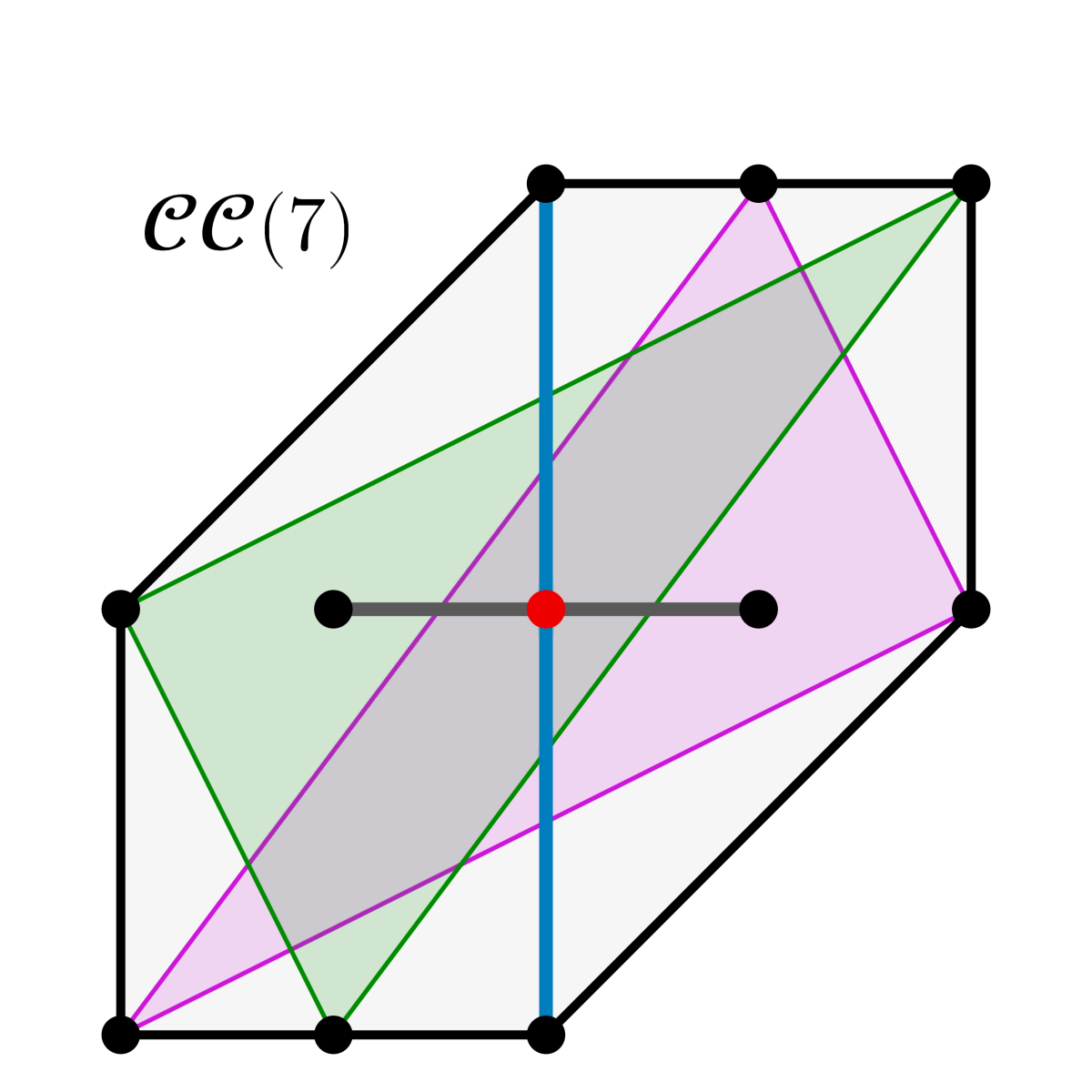}
    \hfill
    \includegraphics[width=0.243\linewidth]{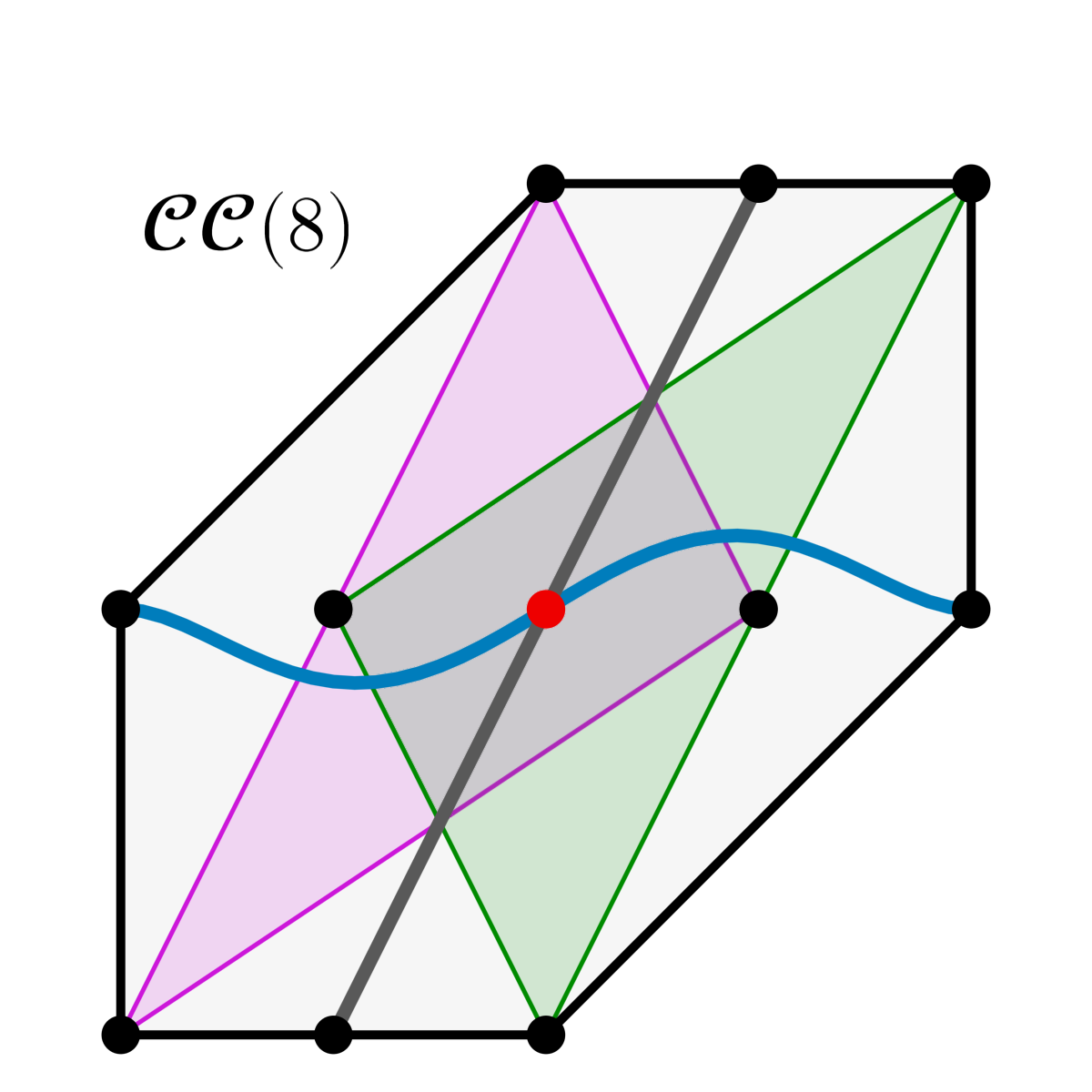}

    \includegraphics[width=0.243\linewidth]{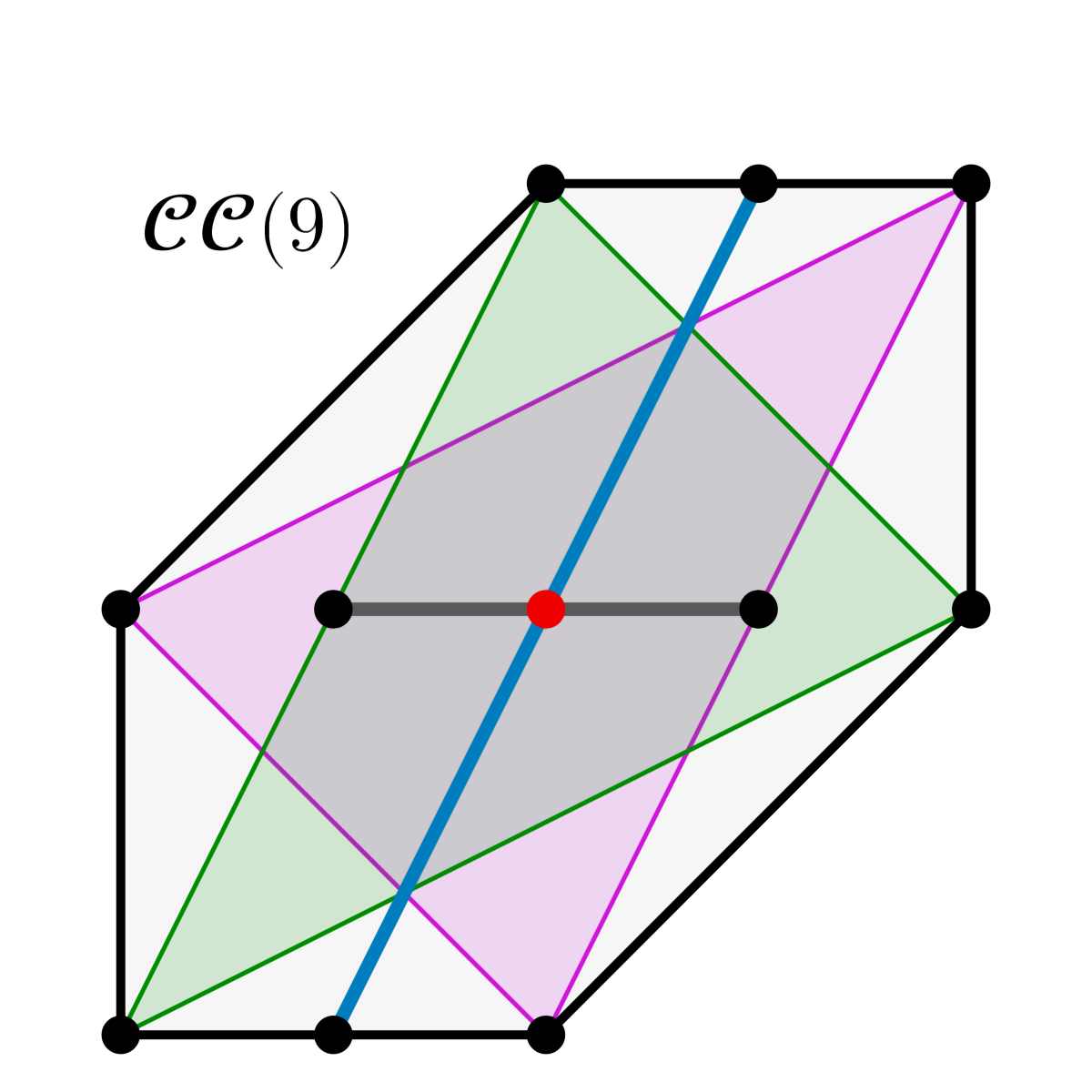}
    \hfill
    \includegraphics[width=0.243\linewidth]{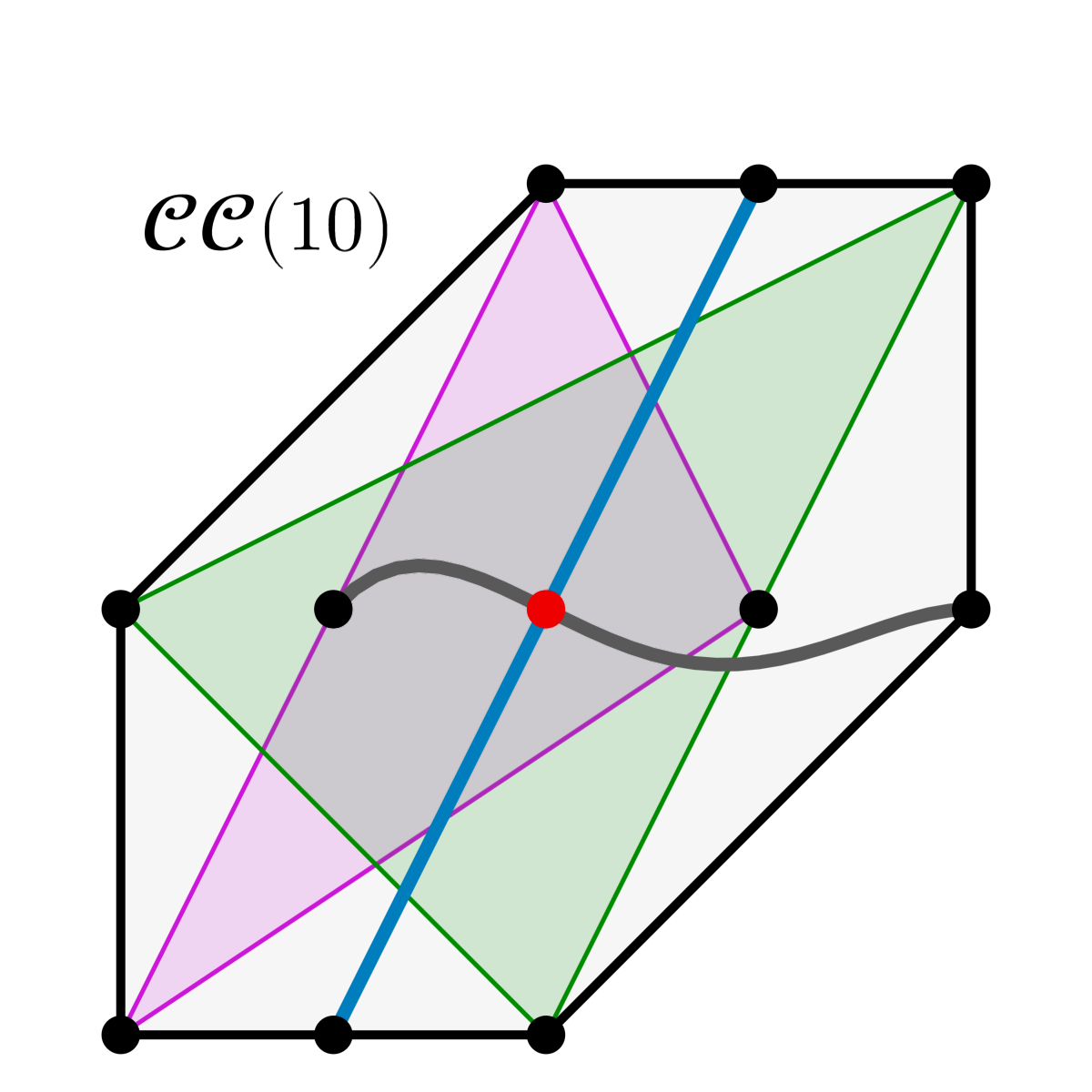}
    \hfill
    \includegraphics[width=0.243\linewidth]{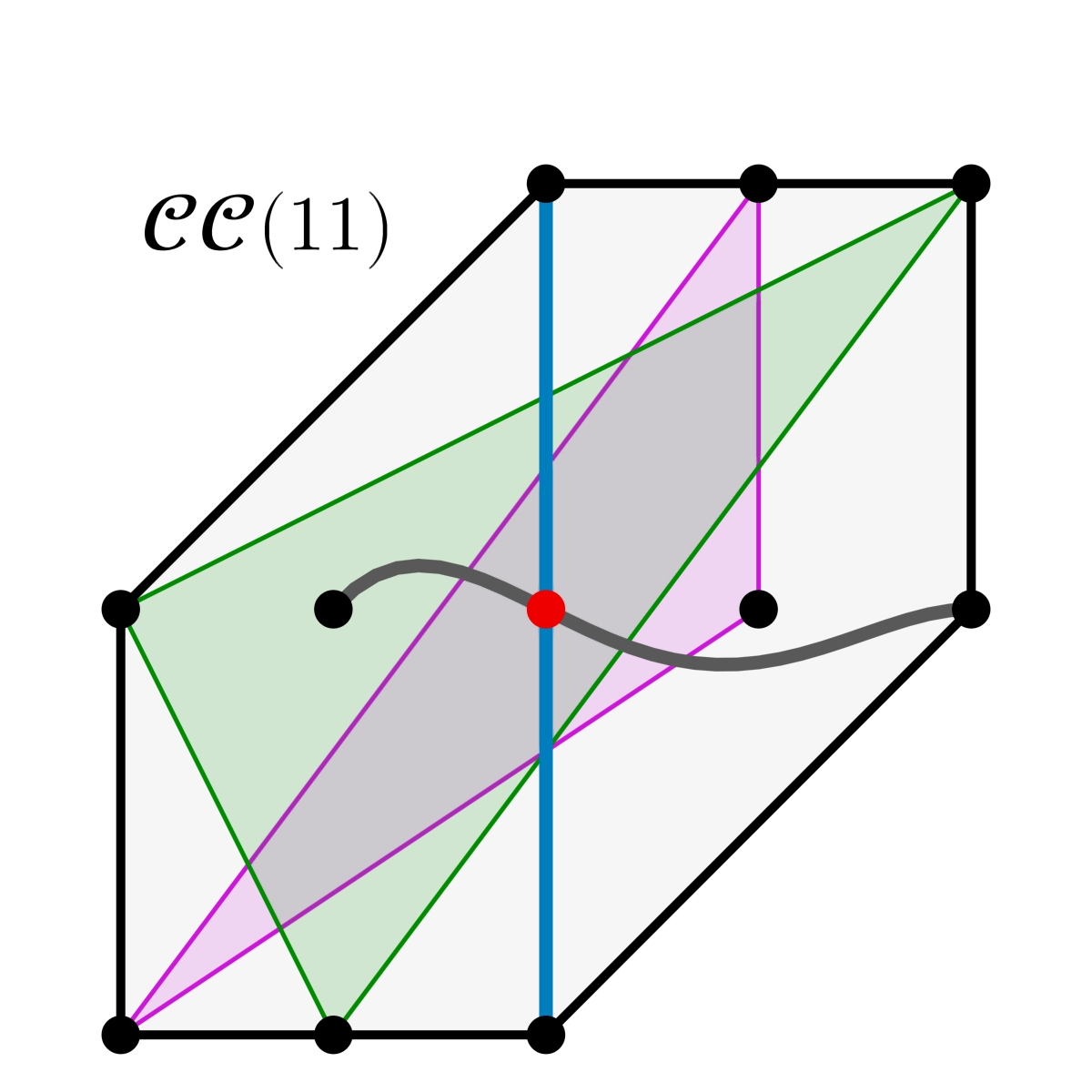}
    \hfill
    \includegraphics[width=0.243\linewidth]{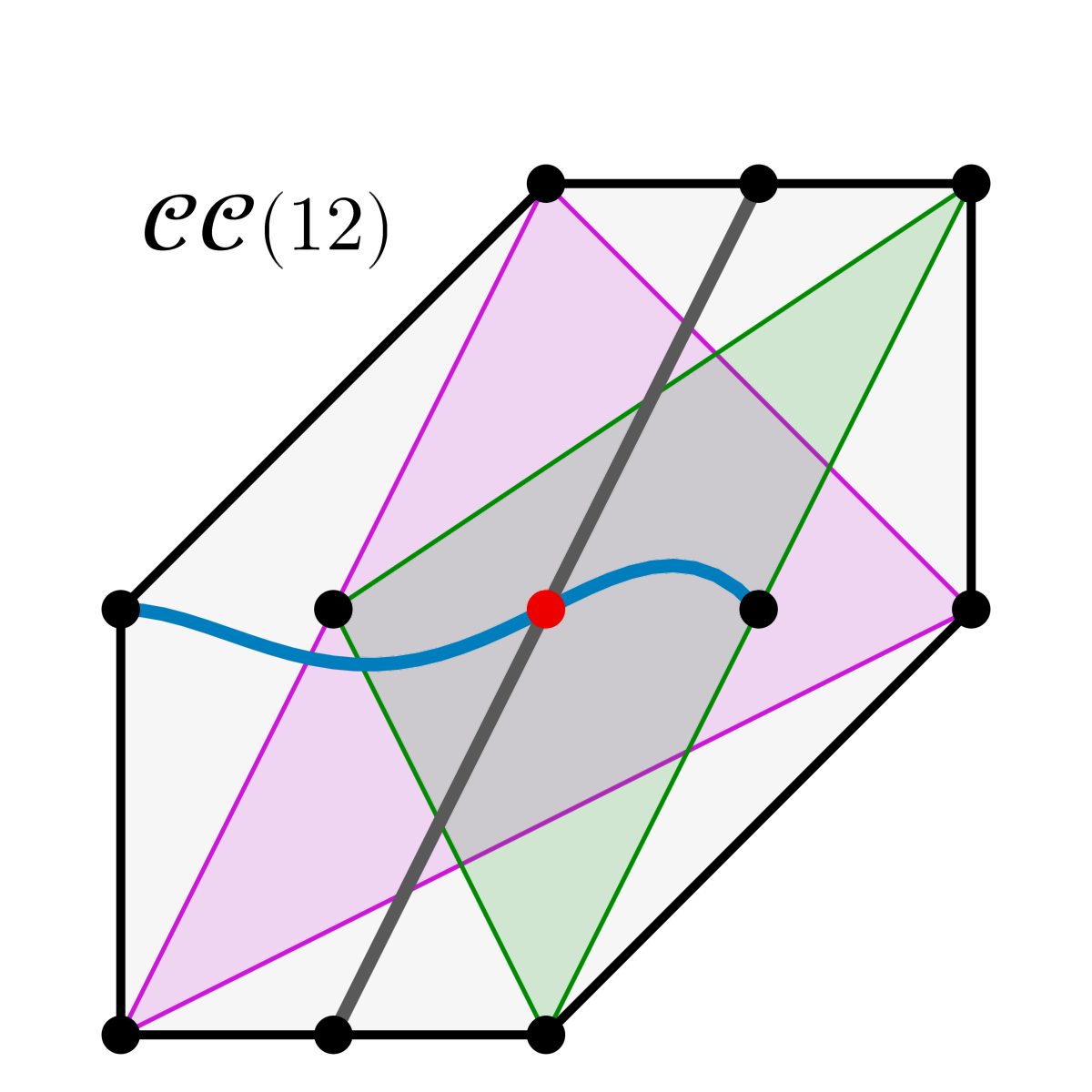}

    \includegraphics[width=0.243\linewidth]{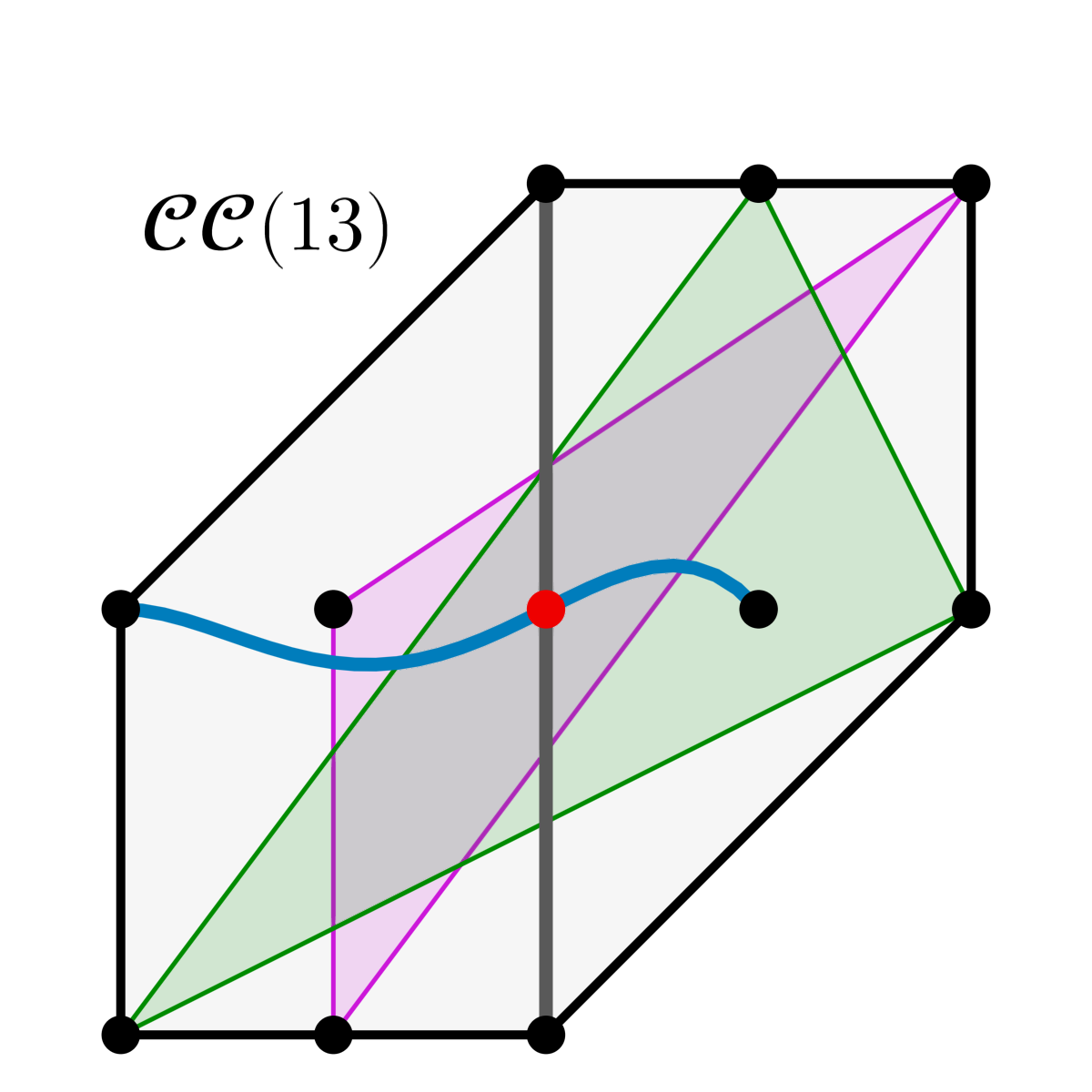}
    \hfill
    \includegraphics[width=0.243\linewidth]{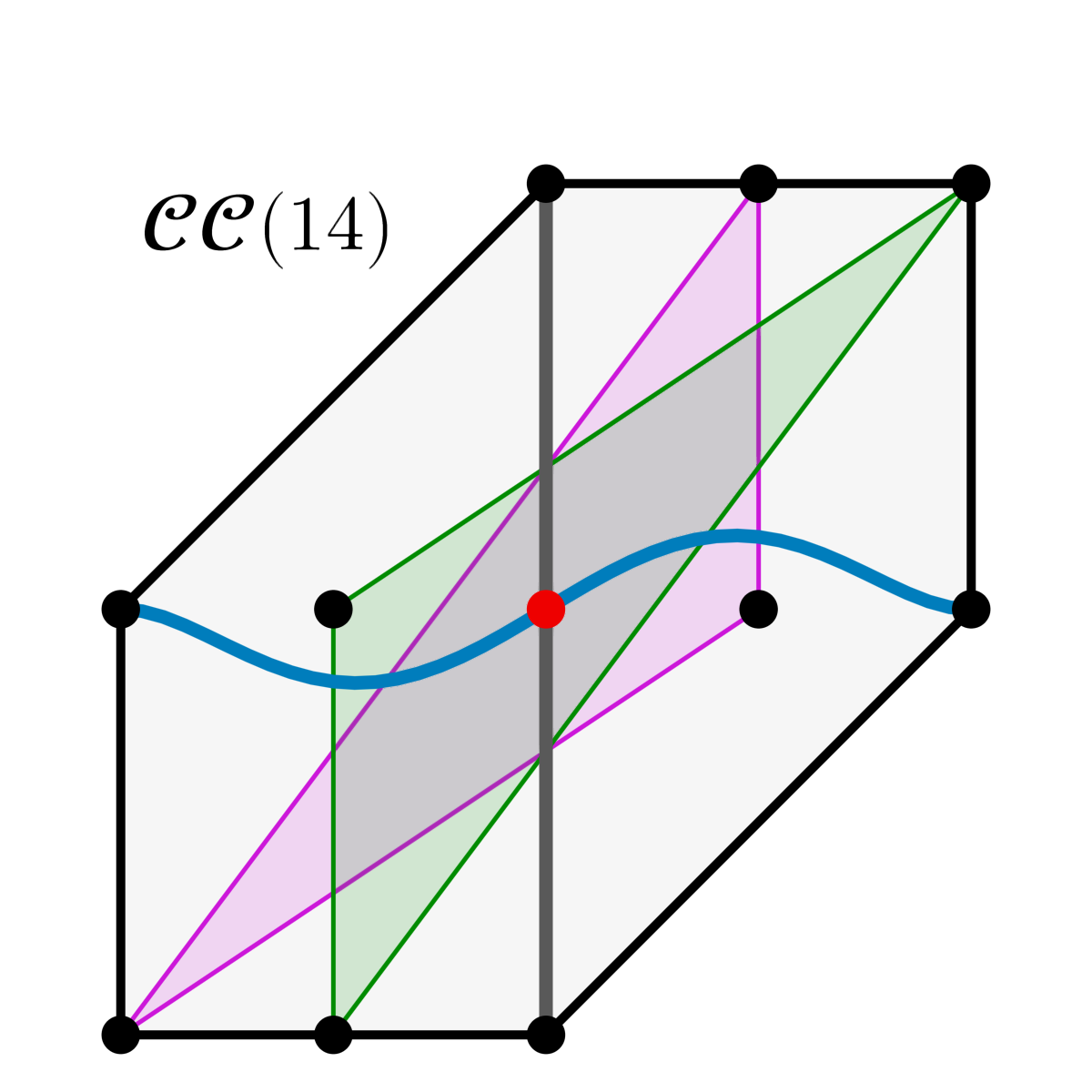}
    \hfill
    \includegraphics[width=0.243\linewidth]{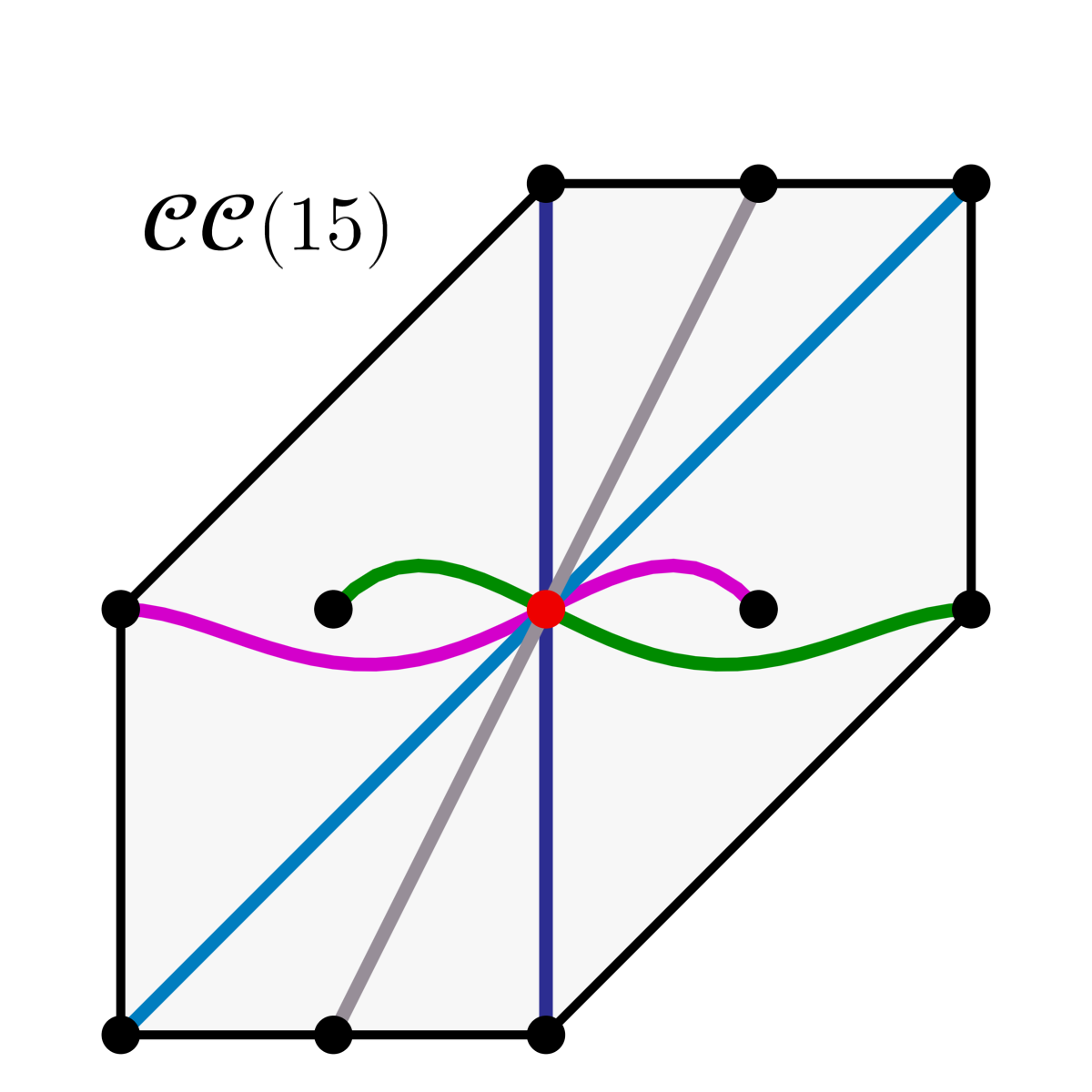}
    \hfill
    \includegraphics[width=0.243\linewidth]{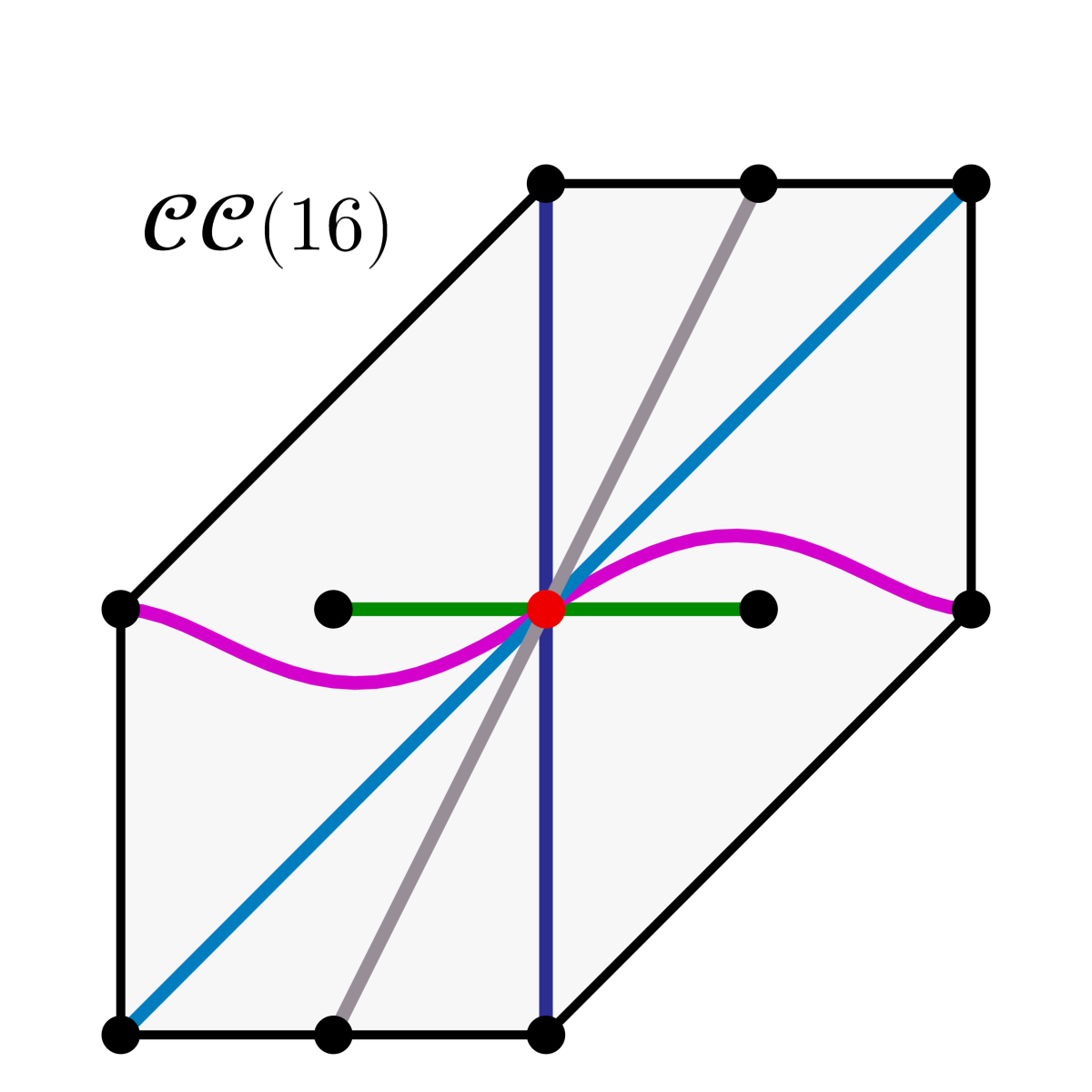}

    \caption{All possible pure covers of the integral hexagon associated with $p_{\Vector{\eta},H}(x_1,x_3)$ defined in \eqref{eq:polynomial_eta_H}. 
    }
    \label{fig:all16configurations}
\end{figure}
\begin{proof}
Every pure cover consists of a partition of $V = \{\Vector{\alpha_1}, \Vector{\alpha_2}, \Vector{\alpha_3}, \Vector{\alpha_4}, \Vector{\alpha_5}, \Vector{\alpha_6}, \Vector{\beta_1}, \Vector{\beta_2}, \Vector{\iota_1}, \Vector{\iota_2}\}$ into simplices that all contain $\Vector{m}$ in their relative interior. Specifically, we must partition the hexagon $H$ into a combination of two-element line segments and three-element triangles that include the center point $\Vector{m}$. Using the fact that there are ten points to partition, a pure cover must either consist of five line segments, or two line segments and two triangles. 

From Figure~\ref{fig:hexagonalface_LABELS}(a), we see that any line segment or triangle contained strictly in $\{\Vector{\alpha_1}, \Vector{\alpha_2}, \Vector{\alpha_6}, \Vector{\beta_1}, \Vector{\iota_1}\}$ cannot contain $\Vector{m}$ in the relative interior. Also, there are unique line segments containing $\Vector{\alpha_1}$, $\Vector{\beta_1}$, and $\Vector{\alpha_2}$ respectively that cross $\Vector{m}$, while $\Vector{\alpha_6}$ and $\Vector{\iota_1}$ each can produce a valid line segment with either $\Vector{\iota_2}$ or $\Vector{\alpha_3}$. 
Thus, we can classify the pure covers by first determining the number of line segments and then selecting which elements from $\{\Vector{\alpha_1}, \Vector{\alpha_2}, \Vector{\alpha_6}, \Vector{\beta_1}, \Vector{\iota_1}\}$ will form these line segments. Exactly two pure covers use five line segments, characterized by pairing $\Vector{\alpha_6}$ with either $\Vector{\iota_2}$ or $\Vector{\alpha_3}$. These covers are labeled as $\mathcal{CC}(15)$ and $\mathcal{CC}(16)$ in Figure \ref{fig:all16configurations}.

Now we catalog the possible triangles to appear in a pure cover. Each triangle can have at most one vertex on the middle row, and needs at least one vertex on both the middle and top rows, in order to contain $\Vector{m}$ in its interior. The two triangles that can have multiple vertices on the same horizontal row are $\{\Vector{\alpha_1}, \Vector{\alpha_2}, \Vector{\beta_2}\}$ and $\{\Vector{\alpha_5}, \Vector{\alpha_4}, \Vector{\beta_1}\}$, and they each preclude any other triangles to appear with them. They are used in covers $\mathcal{CC}(3)$ and $\mathcal{CC}(4)$. 

Every other triangle must contain one vertex on each row. Consider how to construct such a triangle. First, we choose the vertex on the bottom row out of $\Vector{\alpha_1}, \Vector{\beta_1},$ and $\Vector{\alpha_2}$. Assume without loss of generality 
we choose $\Vector{\alpha_1}$. Then the vertex on the top row cannot be the one directly opposite it from $\Vector{m}$, $\Vector{\alpha_4}$, leaving us two remaining options for the vertex on the top. Assume further without loss of generality we choose $\Vector{\alpha_5}$ for the second vertex. This determines which ``side'' of $\Vector{m}$ the vertex on the middle row can be, leaving us with the choice of $\Vector{\iota_2}$ or $\Vector{\alpha_3}$ to finish the triangle. Each such triangle can only be paired with two other triangles. In our example, $\Vector{\alpha_4}$ and $\Vector{\alpha_2}$ must be part of the other triangle, since they cannot be part of a line segment as their opposites $\Vector{\alpha_1}$ and $\Vector{\alpha_5}$. The only choice is which vertex of the middle row to use, $\Vector{\alpha_6}$ or $\Vector{\iota_1}$.

In fact, once we choose the bottom and top vertices of the non-horizontal line segment, we are left with exactly four possible covers: the remaining bottom and top vertices must be in different triangles from their opposite-to-$\Vector{m}$ vertex, and it remains to choose which vertices along the middle row belong as part of the triangles and which form a line segment. Since we have $3$ ways of picking the bottom and top vertex of the line segment, this gives us $3 \cdot 4 = 12$ covers. Combined with covers $\mathcal{CC}(3)$ and $\mathcal{CC}(4)$, where the triangles have no vertices on the middle rows, as well as the five-line-segment covers $\mathcal{CC}(15)$ and $\mathcal{CC}(16)$, that yields all 16 possible covers.
\end{proof}
This theorem shows that there are exactly 16 possible pure covers of this integral hexagon $H$. They are enumerated in \Cref{fig:all16configurations}.
With this additional information, we intend to investigate which of these covers corresponds to the largest region in the semialgebraic set defined by the inequalities $a(\Vector{\eta})>0$ and $b(\Vector{\eta})<0$ (see Equation \eqref{eq:def-a-and-b}). These numbers are derived from the potentially negative coefficients of the bivariate polynomial $p_{\Vector{\eta},H}$ supported on the hexagon $H$ (see Equation \eqref{eq:polynomial_eta_H}).

\subsection{Experimental Setup}
\label{section:experimental-setup}
To investigate which of these 16 pure covers occupy the largest region of monostationarity, we draw uniformly distributed samples from hypercubes $[0,N]^{12}$ in the positive orthant having one vertex at the origin and side length $N>0$. To compare the region close to 0 with the behavior far away from 0, we run the experiments for $N=0.1,\,1,\,10,\,100$. 
For each collection of positive reaction rate constants from the semialgebraic set $a(\Vector{\eta}) > 0$ and $b(\Vector{\eta}) < 0$, we determine whether the circuit polynomial corresponding to each of the 16 pure covers is positive by relying on Theorem \ref{theorem_circuit_number}. In that case, multistationarity is not enabled (cf. Lemma \ref{lem:mono_nonnegativity}). The more points that are covered by this sufficient criterion, the better we understand the region of monostationarity in dual phosphorylation.

By drawing approximately $52\cdot 10^6$ samples from the region where $a(\Vector{\eta})> 0$ and $b(\Vector{\eta})<0$, we make sure that our numbers are statistically significant with a confidence level of $99\%$ and desired precision given by $\pm 0.005\%$ \cite{determiningsamplesizes}. The code for all experiments, along with the raw data, is available through the Julia package \texttt{CRNHexagon.jl}\footnote{The code is available in the repository \url{https://github.com/matthiashimmelmann/CRNHexagon}.}. For each hypercube $[0,N]^{12}$, each combination of reaction rate constants is compared across all 16 covers, making the results comparable. These experiments are run on a Windows 11 64-Bit machine with an Intel i7-10750H @ 2.6GHz processor and 16GB RAM. 

\subsection{The Best Pure Cover}
\label{section:the-best-circuit-cover}
For each of the 16 possible pure covers of the Newton polytope's hexagonal face $H$ (see \Cref{prop:16-circuit-covers}), we depict the ratio of points where monostationarity is enabled according to Lemma \ref{lem:mono_nonnegativity}  in Table \ref{tab:objectivecomparison}. The ratios are obtained by dividing the number of such samples by the number of total samples from the region where $a(\Vector{\eta})>0$ and $b(\Vector{\eta})<0$. In \Cref{tab:objectivecomparison}, the union of all sets corresponding to the covers $\mathcal{CC}(1)$ through $\mathcal{CC}(16)$ is denoted by ``$\Sigma$.''

\begin{table*}[h!]
    \centering
        \caption{\justifying The ratio of samples where the circuit polynomial corresponding to the cover labeled $n$ is positive divided by the total number of points drawn from $[0,N]^{12}$ intersected with the semialgebraic set defined by $a(\Vector{\eta})> 0$ and $b(\Vector{\eta})<0$. The rows corresponding to the five best-performing covers $\mathcal{CC}(4)$, $\mathcal{CC}(9)$, $\mathcal{CC}(10)$, $\mathcal{CC}(12)$ and $\mathcal{CC}(15)$ including the row corresponding to the cover used in Feliu et al. \cite{DualPhosphorSONC} have been marked in green and blue, respectively. In the top row, $\Sigma$ denotes the union of all semialgebraic sets corresponding to covers $\mathcal{CC}(1)$ through $\mathcal{CC}(16)$.}

            \label{tab:objectivecomparison}

\def\arraystretch{1.6}
\setlength\tabcolsep{1.61ex}
        \bgroup
\resizebox{0.525\textwidth}{!}{    \begin{tabular}{|c" c | c | c | c|}
 \multicolumn{1}{c"}{$\#$} & $[0,0.1]^{12}$& $~~[0,1]^{12}~$&$~[0,10]^{12}~$&$~[0,100]^{12}$\\  \thickhline
 ~&~&~&~&~\\[-5mm]
 ~$\Sigma$~~&0.98490&0.98490&0.98491&0.98492 \\[2mm] \hline
 $\mathcal{CC}(1)$&0.94937&0.94941&0.94937&0.94938\\ \hline
$\mathcal{CC}(2)$&0.96239&0.96237&0.96237&0.96239\\ \hline
$\mathcal{CC}(3)$&0.96510&0.96510&0.96510&0.96512\\ \hline
\cellcolor{celltwo!23}$\mathcal{CC}(4)$&\cellcolor{celltwo!23}0.97779&\cellcolor{celltwo!23}0.97779&\cellcolor{celltwo!23}0.97780&\cellcolor{celltwo!23}0.97780\\ \hline
$\mathcal{CC}(5)$&0.96236&0.96242&0.96238&0.96240\\ \hline
$\mathcal{CC}(6)$&0.96535&0.96536&0.96537&0.96538\\ \hline
$\mathcal{CC}(7)$&0.96535&0.96536&0.96537&0.96538\\ \hline
$\mathcal{CC}(8)$&0.97170&0.97170&0.97172&0.97172\\ \hline
\cellcolor{cellthree!23}$\mathcal{CC}(9)$&\cellcolor{cellthree!23}0.97851&\cellcolor{cellthree!23}0.97852&\cellcolor{cellthree!23}0.97852&\cellcolor{cellthree!23}0.97855\\ \hline
\cellcolor{celltwo!23}$\mathcal{CC}(10)$&\cellcolor{celltwo!23}0.97993&\cellcolor{celltwo!23}0.97994&\cellcolor{celltwo!23}0.97994&\cellcolor{celltwo!23}0.97996\\ \hline
$\mathcal{CC}(11)$&0.96236&0.96242&0.96238&0.96240\\ \hline
\cellcolor{celltwo!23}$\mathcal{CC}(12)$&\cellcolor{celltwo!23}0.97995&\cellcolor{celltwo!23}0.97994&\cellcolor{celltwo!23}0.97993&\cellcolor{celltwo!23}0.97996\\ \hline
$\mathcal{CC}(13)$&0.96239&0.96237&0.96237&0.96239\\ \hline
$\mathcal{CC}(14)$&0.94937&0.94941&0.94937&0.94938\\ \hline
\cellcolor{celltwo!23}$\mathcal{CC}(15)$&\cellcolor{celltwo!23}0.98310&\cellcolor{celltwo!23}0.98310&\cellcolor{celltwo!23}0.98312&\cellcolor{celltwo!23}0.98312\\ \hline
$\mathcal{CC}(16)$&0.97324&0.97322&0.97326&0.97325\\ \hline
\end{tabular}
}
    \egroup
\end{table*}

There are substantial variations in the quality of the covers, independent of the box size. The row corresponding to the highly symmetric cover $\mathcal{CC}(9)$ that is proposed in Feliu et al. \cite{DualPhosphorSONC} is marked blue and the cover does indeed perform well. Nevertheless, there are three covers consistently outperforming it: $\mathcal{CC}(10)$, $\mathcal{CC}(12)$ and $\mathcal{CC}(15)$. The investigation of all 16 pure covers thus leads to an increase in the size of the known region of monostationarity. For the readers' convenience, the rows of the four best candidates besides $\mathcal{CC}(9)$---namely $\mathcal{CC}(4)$, $\mathcal{CC}(10)$, $\mathcal{CC}(12)$ and $\mathcal{CC}(15)$---have been colored green in the table. Intriguingly, the two covers consisting only of one-dimensional simplices on average cover a larger region than the covers that also utilize triangles.

\Cref{tab:objectivecomparison} does not offer insight into the relations between the different regions where the 16 sums of circuit polynomials are nonnegative.

One might wonder to what extent these sets overlap. They could, in theory, form an ascending chain of inclusions, or, alternatively, their union might cover the entire region of monostationarity. However, a more realistic expectation is that the outcome lies somewhere between these two extremes.

\subsection{Relative Comparison of the Pure Covers}
\label{section:containment}
Since in Feliu et al. \cite{DualPhosphorSONC}, cover $\mathcal{CC}(9)$ is used for all symbolic computations surrounding the question of monostationarity, we think it is appropriate to choose it as a baseline here. In Table \ref{tab:relative+-comparison} from the Appendix \ref{section:appendix}, each cover is therefore compared with the highly symmetric cover $\mathcal{CC}(9)$ on the same set of samples that is used in Table \ref{tab:objectivecomparison}. 

The amount of samples where the new cover detected monostationarity, whereas $\mathcal{CC}(9)$ did not, is recorded with the symbol ``$+$.'' Conversely, the number of samples where cover $\mathcal{CC}(9)$ found that the corresponding polynomial is positive and the other cover did not is recorded with a ``$-$.'' Finally, whenever neither cover could discern monostationarity, that sample is recorded under ``$0$.''

From this data, we find that most covers are not equal. While there are some covers contained in the region corresponding to cover $\mathcal{CC}(9)$, as indicated by a ``0.0'' in the ``$+$'' row, most covers have at least a small region associated to them that is not covered by $\mathcal{CC}(9)$. Unsurprisingly, the previously best covers $\mathcal{CC}(4)$, $\mathcal{CC}(10)$, $\mathcal{CC}(12)$, and $\mathcal{CC}(15)$ from Table \ref{tab:objectivecomparison} have the highest ratio in the ``$+$'' row and the lowest in their ``$0$'' and ``$-$'' rows. Notably, all their ``$+$'' rows contain larger numbers than their ``$-$'' rows. This difference is most significant in the covers $\mathcal{CC}(10)$, $\mathcal{CC}(12)$ and $\mathcal{CC}(15)$. There, the ratio of samples where only the new cover was able to deduce monostationarity relative to the points where only the previous cover $\mathcal{CC}(9)$ is able to deduce monostationarity is approximately $36\%$, $36\%$ and $82\%$, respectively. 

We can infer from \Cref{tab:relative+-comparison} that most covers are not contained in the region of monostationarity provided by cover $\mathcal{CC}(9)$. Only the semialgebraic sets corresponding to the  covers $\mathcal{CC}(1)$, $\mathcal{CC}(6)$, $\mathcal{CC}(7)$ and $\mathcal{CC}(14)$ are contained in the one given by cover $\mathcal{CC}(9)$. As the pure covers are only compared to $\mathcal{CC}(9)$, which has been introduced by Feliu et al. \cite{DualPhosphorSONC}, \Cref{tab:relative+-comparison} does not offer a comparison between any other two covers. We generate these missing binary comparisons by considering all possible ${16 \choose 2}$ combinations of two covers. The goal is to get a complete picture of how the semialgebraic sets corresponding to the pure covers $\mathcal{CC}(1)$ to $\mathcal{CC}(16)$ relate to each other. 

Two pieces of information are recorded: Whether a point is contained in the region of monostationarity provided by cover $\mathcal{CC}(i)$ but not the other and whenever a point is present only in a single cover. If there are no points contained in a cover $A$ that are not also contained in a different cover $B$, we say that $A$ is \struc{contained} in $B$. Conversely, if this property additionally holds for $B$ relative to $A$, we conclude that their covered regions are \struc{equal}. \Cref{fig:graph-containment} depicts the containment information as a Hasse diagram, the standard way to portray posets. Whenever two covers are considered equal, their labels appear in the same vertex.

\begin{figure}[h!]
    \centering
    \includegraphics[width=0.55\linewidth]{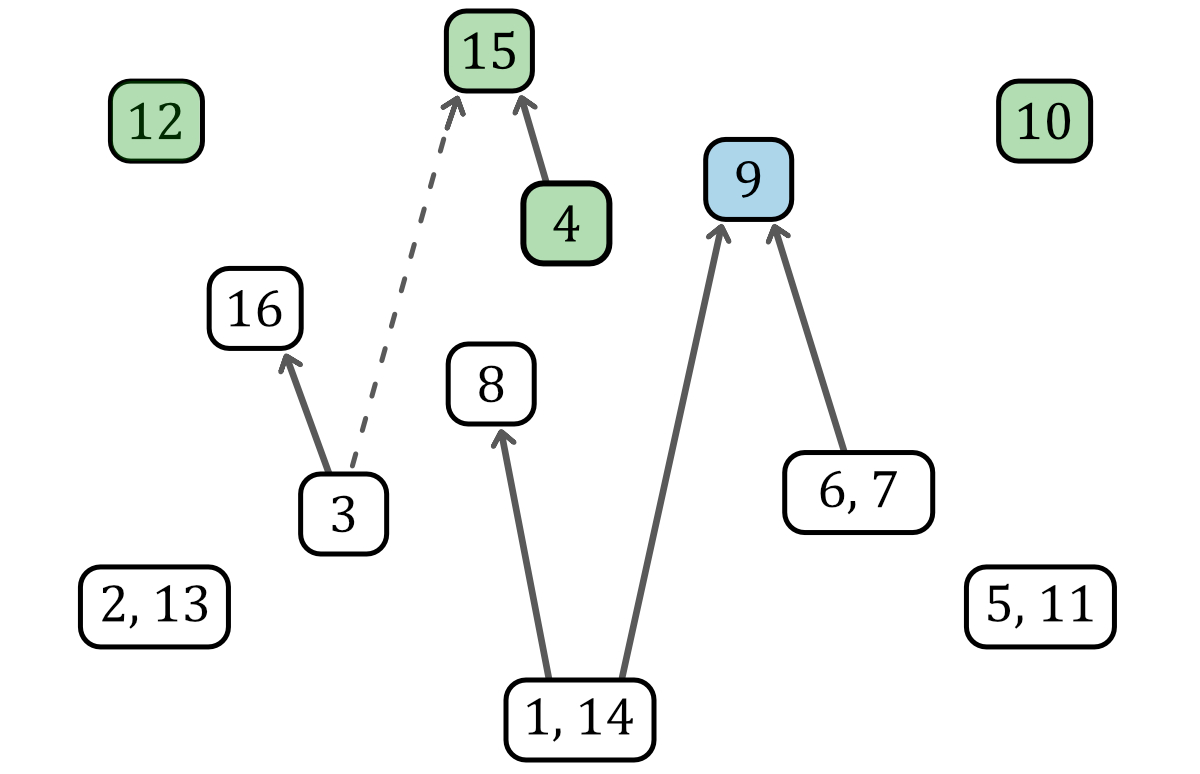}
    \caption{Hasse diagram of the covers' containment. If cover $A$ is contained in cover $B$, a directed edge $A\longrightarrow B$ is added. If additionally, $B$ is contained in $A$, we put ``$A,\,B$'' in the same vertex. We suspect that cover $\mathcal{CC}(3)$ is contained in cover $\mathcal{CC}(15)$ as well, since we only found a single point in $\mathcal{CC}(3)$ that $\mathcal{CC}(15)$ does not include. It is not unlikely that numerical inaccuracies produce this point. In this case, we mark the relation by a dashed arrow: $A\dashrightarrow B$. The covers' overall quality improves along an imagined $y$-axis (cf. \Cref{tab:objectivecomparison}). The labels $i$ in this diagram refer to pure cover $\mathcal{CC}(i)$.}
    \label{fig:graph-containment}
\end{figure}

Curiously, the fifth-best cover $\mathcal{CC}(4)$ is contained in the best cover $\mathcal{CC}(15)$. Therefore, a combination of $\mathcal{CC}(4)$ and $\mathcal{CC}(15)$ does not improve the cover's overall quality, even though Tables \ref{tab:objectivecomparison} and \ref{tab:relative+-comparison} would have suggested that. In addition, the only covers that contain points no other cover contains are $\mathcal{CC}(8)$, $\mathcal{CC}(9)$, $\mathcal{CC}(10)$, $\mathcal{CC}(12)$ and $\mathcal{CC}(15)$. While $\mathcal{CC}(8)$ contains a negligible amount of unique points, covers $\mathcal{CC}(9)$, $\mathcal{CC}(10)$, $\mathcal{CC}(12)$ and $\mathcal{CC}(15)$ for all box sizes $[0,N]$ each contain at least $0.01\%$ of the sampled points that no other cover contains. It is tempting to assume from the containment diagram in \Cref{fig:graph-containment}, that since maximal covers are ``better,'' they contain points that no other cover does. This is not true, and in fact, most pure covers are contained in the union of the other covers. Therefore, our focus going forward should be on the special covers $\mathcal{CC}(9)$, $\mathcal{CC}(10)$, $\mathcal{CC}(12)$ and $\mathcal{CC}(15)$, as they are the most distinct and provide the most variety, while objectively covering the largest regions, as well. 

\subsection{New Sufficient Conditions for Monostationarity}
\label{section:new-sufficient-conditions}
Based on the empirical results above we have identified the covers $\mathcal{CC}(4)$, $\mathcal{CC}(10)$, $\mathcal{CC}(12)$ and $\mathcal{CC}(15)$, along with cover $\mathcal{CC}(9)$ proposed by Feliu et al. \cite{DualPhosphorSONC}, to be the best pure covers to describe the region of monostationarity via SONC decompositions.
 Each of the covers gives a sufficient condition for monostationarity for a given set of parameters $\Vector{\eta}$.
 We have displayed the condition corresponding to cover $\mathcal{CC}(9)$ as derived in \cite[Thm. 3.5]{DualPhosphorSONC} in Example \ref{ex:one_circuit_in_hexagon}.
 Now, we provide new sufficient conditions for monostationarity that correspond to the best performing covers $\mathcal{CC}(4)$, $\mathcal{CC}(10)$, $\mathcal{CC}(12)$ and $\mathcal{CC}(15)$.
\begin{thm}
\label{cor:newsufficientconditions}
Let the set of parameters $\Vector{\eta}$ be given as in \eqref{eq:eta} and the polynomial $p_{\Vector{\eta}, H}$ be given as in \eqref{eq:polynomial_eta_H}. 
Assume $a(\Vector{\eta})>0$ and $b(\Vector{\eta})<0$.
It holds that $p_{\Vector{\eta}, H}$ is nonnegative and thus the CRN is monostationary if $\Vector{\eta}$ fulfills one of the following conditions:
\begin{align*}
-b(\Vector{\eta})~ &\leq~ 
    4(K_1K_4\kappa_6\kappa_9a(\Vector{\eta}))^{\frac{1}{2}}
    +2(K_2K_3\kappa_3\kappa_{12}a(\Vector{\eta}))^{\frac{1}{2}}
    \tag{$\mathcal{CC}(15)$}\\
    &+ 3(K_1K_2K_3K_4\kappa_3\kappa_6^2\kappa_9^2\kappa_{12})^{\frac{1}{3}}(K_2^{-1}(K_2K_4)^{\frac{1}{3}}+K_3^{-1}(K_1K_3)^{\frac{1}{3}})\\[1em]
-b(\Vector{\eta})~ &\leq ~
    4(K_1K_2K_3K_4\kappa_3\kappa_6\kappa_9\kappa_{12})^{\frac{1}{4}}
     (2a(\Vector{\eta}))^{\frac{1}{2}}\tag{$\mathcal{CC}(4)$}\\
    &+ 3(K_1K_2K_3K_4\kappa_3\kappa_6^2\kappa_9^2\kappa_{12})^{\frac{1}{3}}(K_2^{-1}(K_2K_4)^{\frac{1}{3}}+K_3^{-1}(K_1K_3)^{\frac{1}{3}})\\[1em]
-b(\Vector{\eta})~ &\leq ~
     4K_2^{-1}(K_1^2K_2^3K_3K_4^2\kappa_3\kappa_6^2\kappa_9^2\kappa_{12}a(\Vector{\eta}))^{\frac{1}{4}}
     +3(K_1K_4^2\kappa_6^2\kappa_9^2a(\Vector{\eta}))^{\frac{1}{3}}
     \tag{$\mathcal{CC}(10)$}\\
     &+ 2(K_2K_3\kappa_3\kappa_{12}a(\Vector{\eta}))^{\frac{1}{2}}
     + 3K_3^{-1}(K_1^2K_2K_3^2K_4\kappa_3\kappa_6^2\kappa_9^2\kappa_{12})^{\frac{1}{3}}\\[1em]
-b(\Vector{\eta})~ &\leq ~
     4K_3^{-1}(K_1^2K_2K_3^3K_4^2\kappa_3\kappa_6^2\kappa_9^2\kappa_{12}a(\Vector{\eta}))^{\frac{1}{4}}
     +3(K_1^2K_4\kappa_6^2\kappa_9^2a(\Vector{\eta}))^{\frac{1}{3}}
     \tag{$\mathcal{CC}(12)$}\\
     &+ 2(K_2K_3\kappa_3\kappa_{12}a(\Vector{\eta}))^{\frac{1}{2}}
     +3K_2^{-1}(K_1K_2^2K_3K_4^2\kappa_3\kappa_6^2\kappa_9^2\kappa_{12})^{\frac{1}{3}}
    \end{align*}
\end{thm}
\begin{proof}
    The monostationarity follows from Lemma \ref{lem:mono_nonnegativity}. We certify the nonnegativity of $p_{\Vector{\eta}, H}$ via Theorem \ref{theorem_circuit_number} and the SONC decompositions into the pure covers $\mathcal{CC}(15),\; \mathcal{CC}(4),\; \mathcal{CC}(10) \text{ and } \mathcal{CC}(12)$ respectively. 
\end{proof}
 
This result emphasizes the similarity of the covers $\mathcal{CC}(10)$ and $\mathcal{CC}(12)$. By swapping $K_2$ and $K_3$ in the first term, $K_1$ and $K_4$ in the second term as well as $K_2$ and $K_3$ in the last term, we can transform the sum of circuit numbers corresponding to cover $\mathcal{CC}(10)$ into the sum of circuit numbers of $\mathcal{CC}(12)$. In fact, the Michaelis-Menten constants defined in Equation \eqref{eq:Michaelis-Menten} can be treated as free parameters, as each $K_i$ is defined in terms of two independent parameters and one dependent parameter out of $\{\kappa_3,\kappa_6,\kappa_9,\kappa_{12}\}$. Since the dependent parameters have equal exponents in the relevant terms, the sums of circuit numbers associated with the pure covers $\mathcal{CC}(10)$ and $\mathcal{CC}(12)$ behave equivalently. This fact is further highlighted through a mirror symmetry between these covers in \Cref{fig:all16configurations} and the fact that these covers show exactly the same empirical results in Tables \ref{tab:objectivecomparison} and \ref{tab:relative+-comparison}.
\subsection{Weighted Covers}
\label{sec:weighted}

In the previous sections, only pure covers are considered, meaning that every vertex with a nonnegative coefficient is used in exactly one of the corresponding circuit polynomials. Nevertheless, this is not the only way of decomposing a polynomial into a SONC. As mentioned in Section \ref{section:dual-phosphorylation}, by weighing the vertices' coefficients, multiple covers can be combined to form \struc{weighted covers}. When computing the sum of circuit numbers for a specific cover then it is a priori not clear whether it is optimal. As laid out in \Cref{section:SONC} there exist different approaches to numerically find an optimal lower bound the minimum of a polynomial via SONC \cite{dualityofnonnegativecircuits} and SAGE \cite{Chandrasekaran:Shah:SAGE, Chandrasekaran:Murray:Wiermann} certificates. For symbolic coefficients, as is the case here, no such algorithm exists. We showed in \Cref{section:the-best-circuit-cover} that solely considering pure covers does not suffice for finding the optimal cover. By enlarging the search space to include weighted covers, it is conceivable that we find a better cover. To obtain an intuition for weighted covers, we give an example expanding Example \ref{exampleSONC}.
\begin{eg}
\label{exampleSONCWeighted}
    We add the constant term $1$ to the polynomial investigated in Example \ref{exampleSONC} and consider $g(x,y) = \struc{x^4y^2} + x^2 + y + 1 - cx^2y$ with $c \in \mathbb{R},\, c \neq 0$.
    Figure \ref{Fig:Ex:NewtonPolytopeWeighted_new} shows the only possible decomposition into a weighted cover of the Newton polytope of $p$. 
    \begin{figure}[h!]
    \centering
    \includegraphics[width=0.31\linewidth]{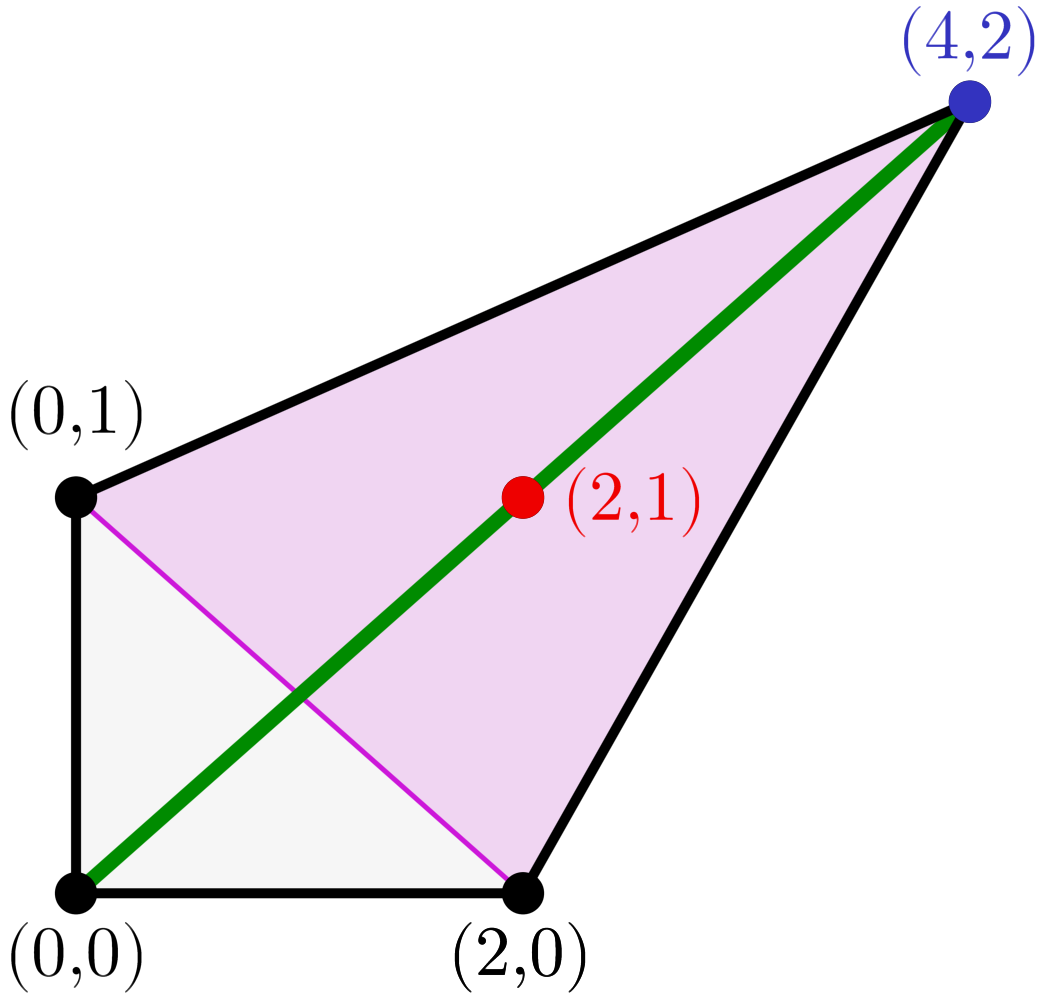}
    \caption{Example of a weighted cover of the Newton polytope corresponding to the polynomial $g$. The vertex $(4,2)$ occurs in two circuit polynomials. }
    \label{Fig:Ex:NewtonPolytopeWeighted_new}
\end{figure}
    There is no other way to decompose the Newton Polytope into simplices that contain the point $(2,1)$ in the interior and use all vertices.
    We want to certify the nonnegativity of $p$ depending on $c$ with this weighted cover.
    While we already considered the first circuit polynomial $f_1$ with support $\{(4,2), (2,0), (0,1), (2,1)\}$ in Example \ref{exampleSONC}, the second circuit polynomial $f_2$ with support $\{(0,0), (4,2), (2,1)\}$ is new.
    Since the vertex corresponding to the monomial $\struc{x^4y^2}$ now appears in both $f_1$ and $f_2$ we split its coefficient among both circuit polynomials via a weight $ w_{(4,2)}~\in~(0,1)$. 
    We partition the polynomial $g$ into $g(x,y)~=~ (\struc{w_{(4,2)}\cdot x^4y^2} + x^2 + y) + (\struc{(1-w_{(4,2)})\cdot x^4y^2} + 1) - cx^2y$.
 The circuit numbers are
    \begin{align*}
        \Theta_{f_1} = \Big(\frac{\struc{w_{(4,2)}}}{\nicefrac{1}{3}}\Big)^{\nicefrac{1}{3}}\Big(\frac{1}{\nicefrac{1}{3}}\Big)^{\nicefrac{1}{3}}\Big(\frac{1}{\nicefrac{1}{3}}\Big)^{\nicefrac{1}{3}} \text{ and } ~
        \Theta_{f_2} =\Big(\frac{\struc{(1-w_{(4,2)})}}{\nicefrac{1}{2}}\Big)^{\nicefrac{1}{2}}\Big(\frac{1}{\nicefrac{1}{2}}\Big)^{\nicefrac{1}{2}}.
    \end{align*}
    Maximizing the sum $\Theta_{f_1} + \Theta_{f_2}$ for $w_{(4,2)} \in (0,1)$, we determine the optimal weight $w_{(4,2)}\approx 0.5497$. 
    With that we can certify using Theorem \ref{theorem_circuit_number}, that $g$ is nonnegative on the positive orthant if
    \begin{align*}
        c \leq \Theta_{f_1} + \Theta_{f_2} \approx 3.7996.
    \end{align*} 
\end{eg}

Inspired by this example, we use the approach of weighted covers to investigate combinations of the pure covers $\mathcal{CC}(1)$ to $\mathcal{CC}(16)$ 
that potentially have an enlarged region of monostationarity. However, using the empirical approach from \Cref{section:empirical data} to scan the entire 15-dimensional search space with sufficiently fine discretization for a weighted cover is infeasible. Since the covers $\mathcal{CC}(4)$, $\mathcal{CC}(9)$, $\mathcal{CC}(10)$, $\mathcal{CC}(12)$, and $\mathcal{CC}(15)$ perform significantly better than the rest in both metrics that we computed (see \Cref{tab:objectivecomparison} and \ref{tab:relative+-comparison}), it seems likely that the best cover will be some combination of these five. Therefore, this section will focus on examining whether this hypothesis is true. 

Before addressing how to sample the space of weighted covers, we first recall the definition from \Cref{section:dual-phosphorylation}. For that, let us consider an integral polytope $\mathcal{P}\subset \mathbb{R}^2$ arising as the Newton polytope $\mathcal{N}(g)$ of a polynomial $g$ that is supported on $V=\text{supp}(g)$. Assume we are given $s$ pure covers $\Sigma_1,\dots,\Sigma_s$ of $\mathcal{P}$. Each vertex covered by a simplex in $\Sigma_i$ then receives a weight $w_i: V\rightarrow [0,1]$ so that the weights on each of $\mathcal{P}$'s vertices sum to 1. Therefore, each vertex comes with one linear constraint, highlighting that this decomposition has at least one solution given by one of the original pure covers.

Subsequently, we can calculate the circuit number (see Equation (\ref{eq:circuit-number})) for each of the circuit polynomials arising from the decomposition dictated by $\Sigma_i$. In the case of weighted covers, we split the coefficients by multiplying them with the respective weights $\omega_i$ (see \cite{constrainedoptimizationapproach}). By construction, the coefficient of each vertex is taken into account exactly with a factor of 1. To make this concept clear, consider the following example of finding the optimal weighted cover that is a combination of covers $\mathcal{CC}(4)$ and $\mathcal{CC}(9)$.

\begin{eg}
\label{ex:homotopy-from-cover-4-to-12}
    Consider the covers $\mathcal{CC}(4)$ and $\mathcal{CC}(9)$. In this example, we intend to investigate their intermediate weighted covers. After weighing each vertex the optimal cover that is a combination of $\mathcal{CC}(4)$ and $\mathcal{CC}(9)$ can be expressed as the weighted sum of circuit numbers
\begin{eqnarray}
\label{eq:theta-4-9}
    \Theta(4,9; \Vector{w}) &=& \frac{3}{\sqrt[3]{4}}\sqrt[3]{w_{4,\Vector{\alpha_6}}c_{\Vector{\alpha_6}}\cdot (w_{4,\Vector{\iota_2}}c_{\Vector{\iota_2}})^2}
        \,+\,\frac{3}{\sqrt[3]{4}}\sqrt[3]{(w_{4,{\Vector{\iota_1}}}c_{\Vector{\iota_1}})^2\cdot w_{4,{\Vector{\alpha_3}}}c_{\Vector{\alpha_3}}}\nonumber\\
       &~&+\,\sqrt{8}\sqrt[4]{(w_{4,{\Vector{\beta_1}}}c_{\Vector{\beta_1}})^2\cdot w_{4,{\Vector{\alpha_5}}}c_{\Vector{\alpha_5}}\cdot w_{4,\Vector{\alpha_4}} c_{\Vector{\alpha_4}}}+\sqrt{8}\sqrt[4]{w_{4,\Vector{\alpha_1}}c_{\Vector{\alpha_1}}\cdot w_{4,\Vector{\alpha_2}}c_{\Vector{\alpha_2}}\cdot (w_{4,{\Vector{\beta_2}}}c_{\Vector{\beta_2}})^2}\nonumber\\
        &~&+\, 2\sqrt{(1-w_{4,\Vector{\beta_1}})c_{\Vector{\beta_1}}\cdot (1-w_{4,\Vector{\beta_2}})c_{\Vector{\beta_2}}}\,+\, 2\sqrt{(1-w_{4,{\Vector{\iota_1}}})c_{\Vector{\iota_1}}\cdot (1-w_{4,{\Vector{\iota_2}}})c_{\Vector{\iota_2}}}\nonumber\\
        &~&+\,3\sqrt[3]{(1-w_{4,\Vector{\alpha_1}})c_{\Vector{\alpha_1}}\cdot (1-w_{4,{\Vector{\alpha_3}}})c_{\Vector{\alpha_3}}\cdot (1-w_{4,\Vector{\alpha_5}})c_{\Vector{\alpha_5}}}\\   
        &~&+\, 3\sqrt[3]{\,(1-w_{4,{\Vector{\alpha_2}}})c_{\Vector{\alpha_2}}\cdot (1-w_{4,\Vector{\alpha_6}})c_{\Vector{\alpha_6}}\cdot (1-w_{4,{\Vector{\alpha_4}}}) c_{\Vector{\alpha_4}}}\nonumber\\
   &\text{s.t.} &\;0\,\leq \,w_{4,{\Vector{\alpha_6}}},\, w_{4,{\Vector{\iota_1}}},\, w_{4,{\Vector{\iota_2}}},\, w_{4,{\Vector{\alpha_3}}},\, w_{4,{\Vector{\beta_1}}},\, w_{4,{\Vector{\alpha_5}}},\, w_{4,{\Vector{\alpha_4}}},\, w_{4,\Vector{\alpha_1}},\, w_{4,\Vector{\alpha_2}},\, w_{4,{\Vector{\beta_2}}}\, \leq\, 1\nonumber
\end{eqnarray}
for the relevant coefficients $c_{\Vector{v}}$ corresponding to vertex $\Vector{v}$ that are defined in \Cref{section:dual-phosphorylation}. Here,  the weight $w_{4,\Vector{v}}$ corresponds to cover $\mathcal{CC}(4)$ and vertex $\Vector{v}$. We obtain the expression $\Theta(4,9;\Vector{w})$ by replacing the coefficients $c_{\Vector{v}}$ in the sum of circuit numbers corresponding to the cover $\mathcal{CC}(i)$ for $i\in \{4,9\}$ by the term $w_{i,\Vector{v}}\cdot c_{\Vector{v}}$ such that $w_{4,\Vector{v}}+w_{9,\Vector{v}}=1$. Rewriting this linear equality simplifies to $w_{9,\Vector{v}}=1-w_{4,\Vector{v}}$. In this scenario, $\Theta(4)$ is exactly given by $\Theta(4,9; \Vector{0})$ and $\Theta(9)$ is equal to $\Theta(4,9; \Vector{1})$.

    \begin{figure}[h!]
        \centering
        \hfill\includegraphics[width=0.92\linewidth]{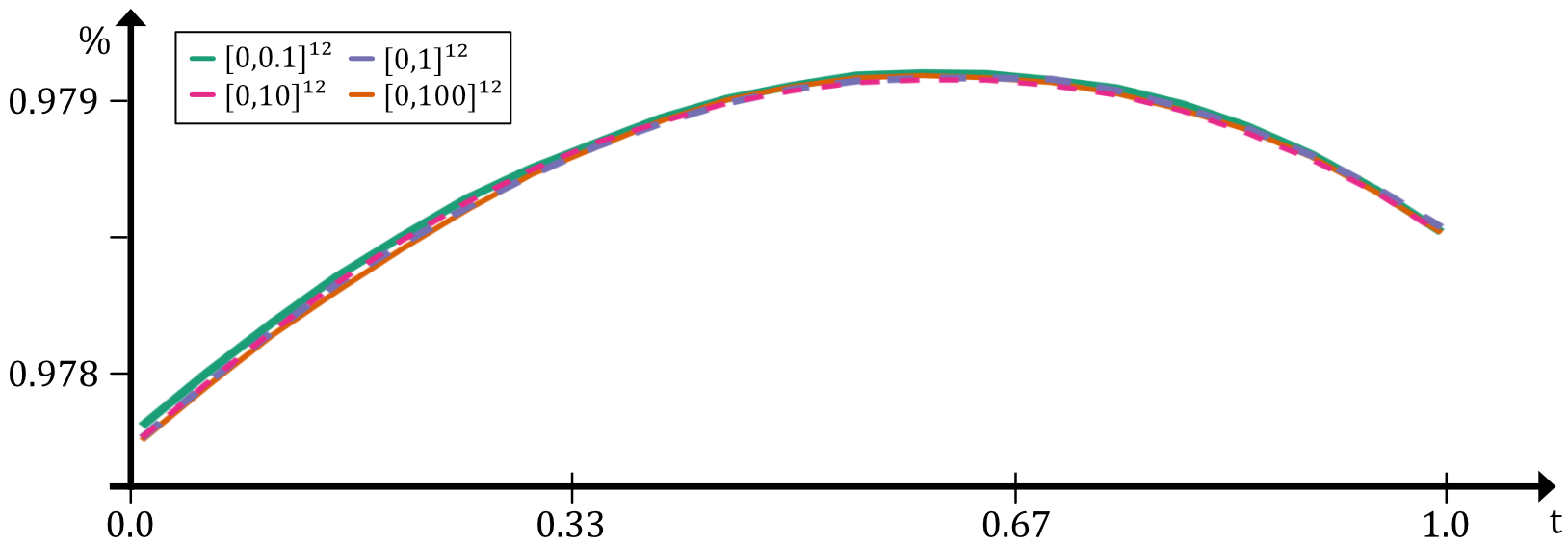}\\[2mm]
        
        \includegraphics[width=0.235\linewidth]{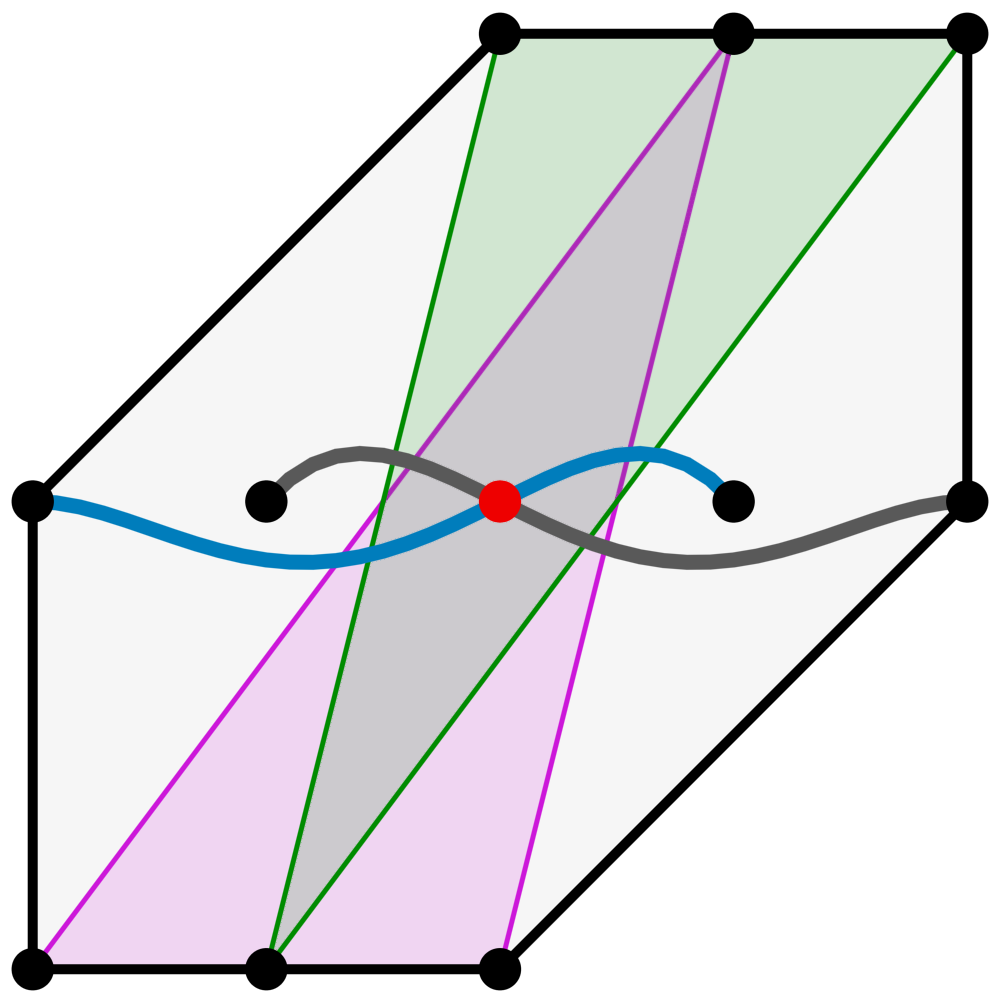}\hfill
        \includegraphics[width=0.235\linewidth]{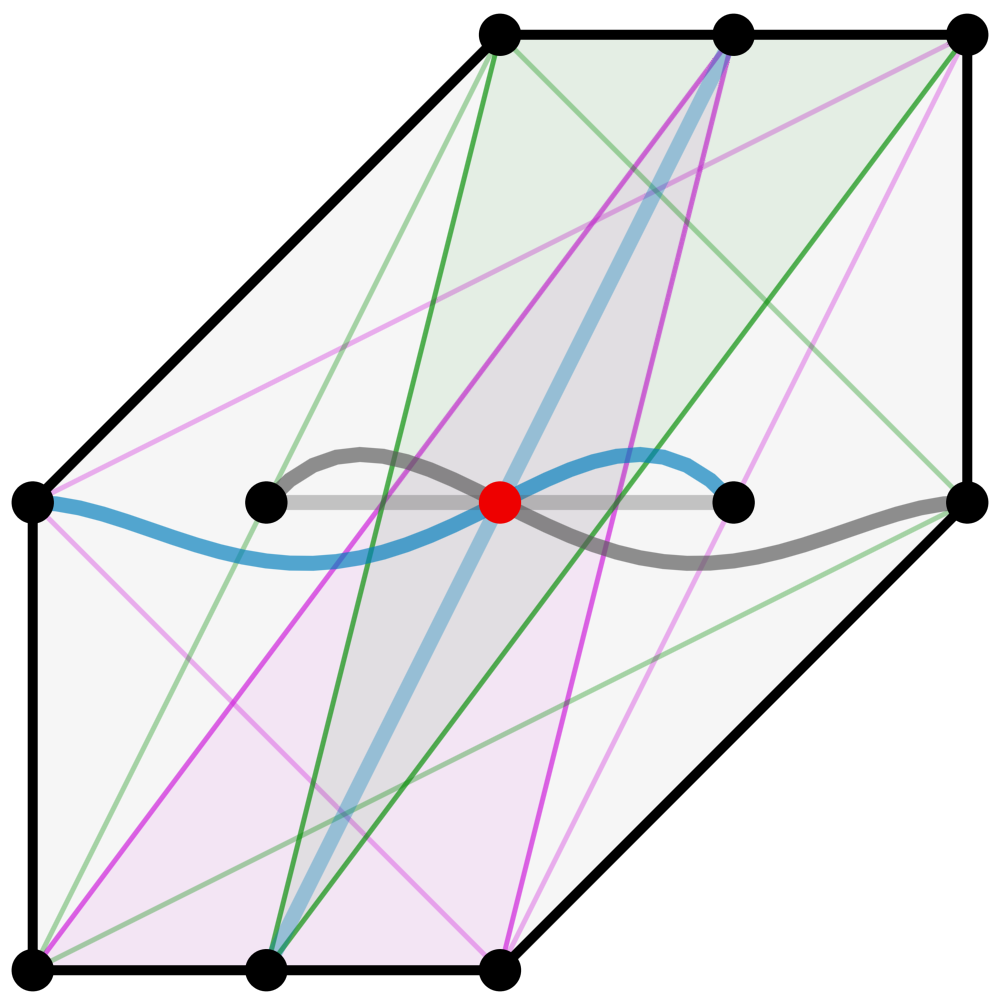}
        \hfill
        \includegraphics[width=0.235\linewidth]{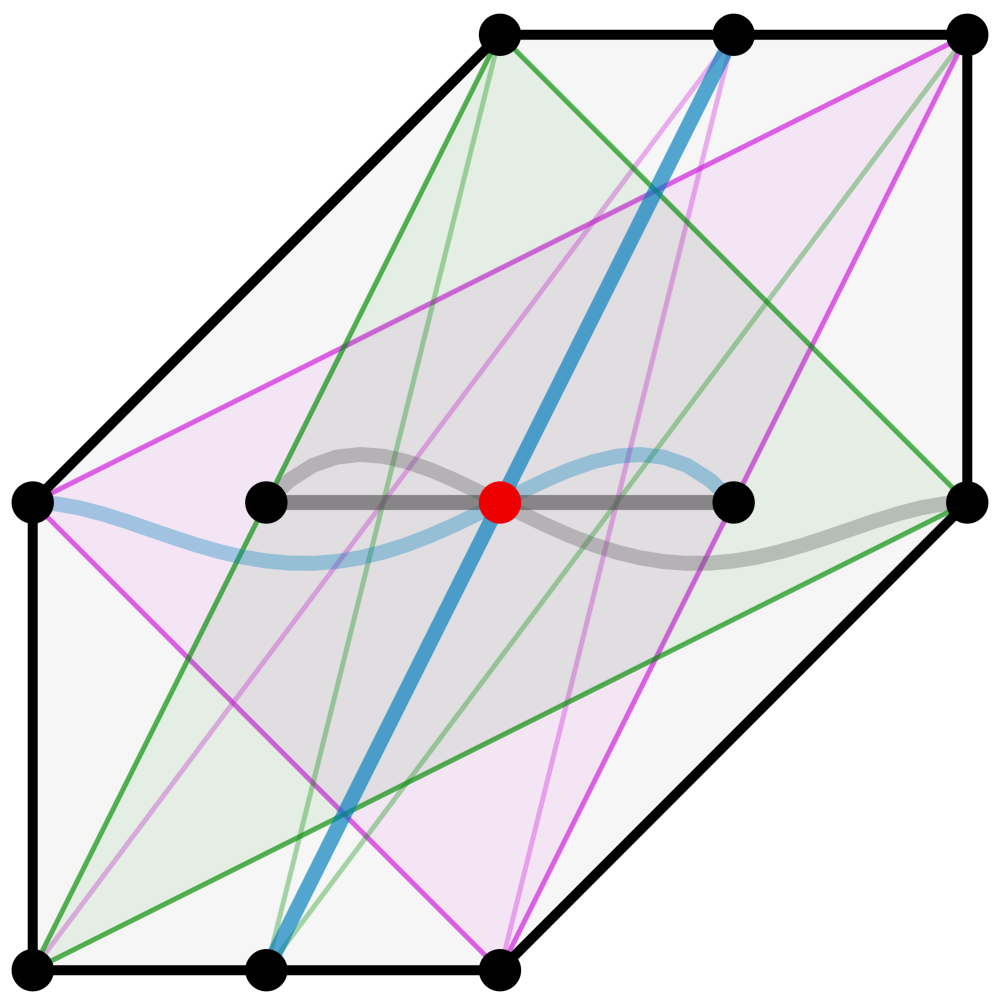}
        \hfill
        \includegraphics[width=0.235\linewidth]{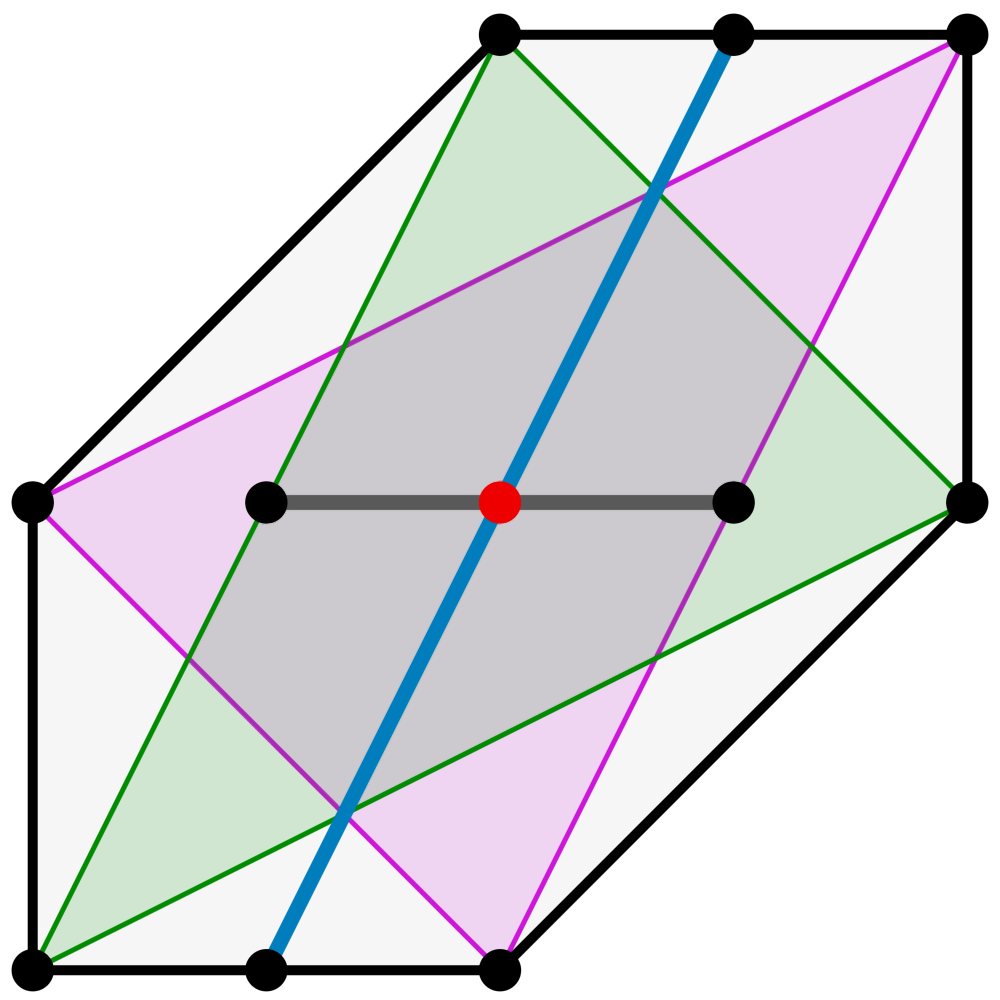}\\[2mm]
        \caption{Homotopy from cover $\mathcal{CC}(4)$ to cover $\mathcal{CC}(9)$, parametrized by $t\in [0,1]$ and discretized with constant step size $\Delta t=0.05$. Each simplex' contribution to the mixed cover is weighted by $t$ and $1-t$, respectively. The graph shows the ratio of samples corresponding to the weighted cover $H_{4,9}(t)$, where monostationarity is enabled (cf. Table \ref{tab:objectivecomparison}).}
        \label{fig:trianglehomotopy}
    \end{figure}
    Yet, optimizing this expression is infeasible, since we are dealing with symbolic coefficients $c_{\Vector{v}}$. For this example and for visualization purposes, we set all weights to be equal: $w_{4,v_i}=w_{4,v_j}$ for all $i,j$. Therefore, computing an intermediate cover reduces to weighing the coefficients in the covers $\mathcal{CC}(4)$ and $\mathcal{CC}(9)$ by $(1-t)$ and $t$, respectively, for $t\in [0,1]$. This lets us define a homotopy parametrized by $t\in[0,1]$ from cover $\mathcal{CC}(4)$ to $\mathcal{CC}(9)$. \Cref{fig:trianglehomotopy} depicts four configurations from this homotopy. 
    Since each coefficient is weighted equally, the homotopy can easily be lifted to a homotopy of the sums of circuit numbers of the form 
    \[H_{4,9}(t) = (1-t)\cdot\Theta(4)+t\cdot \Theta(9).\]
    We can prove this claim by considering the expression for  $\Theta(4,9;\Vector{w})$ in Equation \eqref{eq:theta-4-9}. Replacing $w_{4,\Vector{v}}$ by $(1-t)$ with $t\in [0,1]$ for all vertices $\Vector{v}$, we obtain the exact power of $(1-t)$ necessary to cancel out the radicals. In fact, this observation generalizes to arbitrary sums of circuit numbers in which all coefficients are weighted equally. When considering the expression for a circuit number from Equation \eqref{eq:circuit-number} weighted by $(1-t)$ given by
    \[\Theta_p = \prod_{i=0}^{r}\left(\frac{(1-t)\cdot c_{\Vector{v}}}{\lambda_i}\right)^{\lambda_{i}}\]
    for each circuit polynomial $p$ of $\mathcal{CC}(4)$ with $\sum_i\lambda_i=1$, we find that $(1-t)^{\lambda_i}$ can be factored out of each term in the product so that $\prod_i (1-t)^{\lambda_i} = (1-t)^{\left(\sum_i \lambda_i\right)}=(1-t)$.
    By \Cref{theorem_circuit_number}, the bivariate polynomial $p_{\Vector{\eta},H}$ supported on the hexagon is nonnegative if 
    this expression satisfies $H_{4,9}(t)\geq -c_{\Vector{m}}$. The graph in Figure \ref{fig:trianglehomotopy}, obtained by sampling $23\cdot 10^6$ points from the region where $a(\Vector{\eta})>0$ and $b(\Vector{\eta})<0$, highlights that it is indeed possible to obtain an improved cover by weighting the coefficients. At $t=0.6$, the ratio equals $0.97907$ (cf. Table \ref{tab:objectivecomparison}). While this value still lies below the ratio corresponding to $\mathcal{CC}(10)$, $\mathcal{CC}(12)$ and $\mathcal{CC}(15)$, it signifies an improvement of $\mathcal{CC}(4)$ and $\mathcal{CC}(9)$. 
\end{eg}
This example provides a general idea of how to investigate weighted covers and demonstrates that we can expect an improved sufficient condition from weighting. In the following, we consider the best five pure covers labeled by $\mathcal{CC}(4)$, $\mathcal{CC}(9)$, $\mathcal{CC}(10)$, $\mathcal{CC}(12)$ and $\mathcal{CC}(15)$. Among them are the pure covers with the highest symmetric differences (see Section \ref{section:containment}). We combine them into the $5 \choose 3$ subsets of three weighted covers. To simplify the computations, we choose one weight $w_1,~w_2$, and $w_3$ from $[0,1]$ for each cover analogously to \Cref{ex:homotopy-from-cover-4-to-12}. Again, we assume that $w_1+w_2+w_3=1$. This defines simplicial homotopies between the covers, parametrized by $s\in [0,1]$ and $t\in [0,s]$ such that $w_1=s$, $w_2=t$ and $w_3=1-s-t$:
\[H_{a,b,c}(s,t)~=~s\cdot \Theta(a)+t\cdot \Theta(b)+(1-s-t)\cdot \Theta(c)~~~~~\text{for}~~~a,b,c\in \{4,9,10,12,15\}\] with $\#\{a,b,c\}=3$. In this case, the complexity of the sampling grows quadratically with the chosen discretization, so we only draw $17\cdot 10^6$ samples. Since the performance of each cover is almost identical across all hypercubes $[0,N]^{12}$ (see Tables \ref{tab:objectivecomparison} and \ref{tab:relative+-comparison}), only $N=1$ is considered here. Each simplex parametrized by $s,t$ is subsequently discretized with constant grid sizes $\Delta t=\Delta s=\sfrac{1}{16}$. The resulting ratios of points where monostationarity is enabled relative to the total number of points in the semialgebraic set defined by the inequalities $a(\Vector{\eta})>0$ and $b(\Vector{\eta})<0$ are depicted in Figure \ref{fig:simplicial-cover-homotopies}. We do not sample from all possible three-element subsets, since we expect that the optimal cover contains $\mathcal{CC}(15)$, which performed significantly better than all other covers. The results of this experiment are visualized in \Cref{fig:simplicial-cover-homotopies}.

\begin{figure}[h!]
    \centering
    \includegraphics[width=0.635\linewidth]{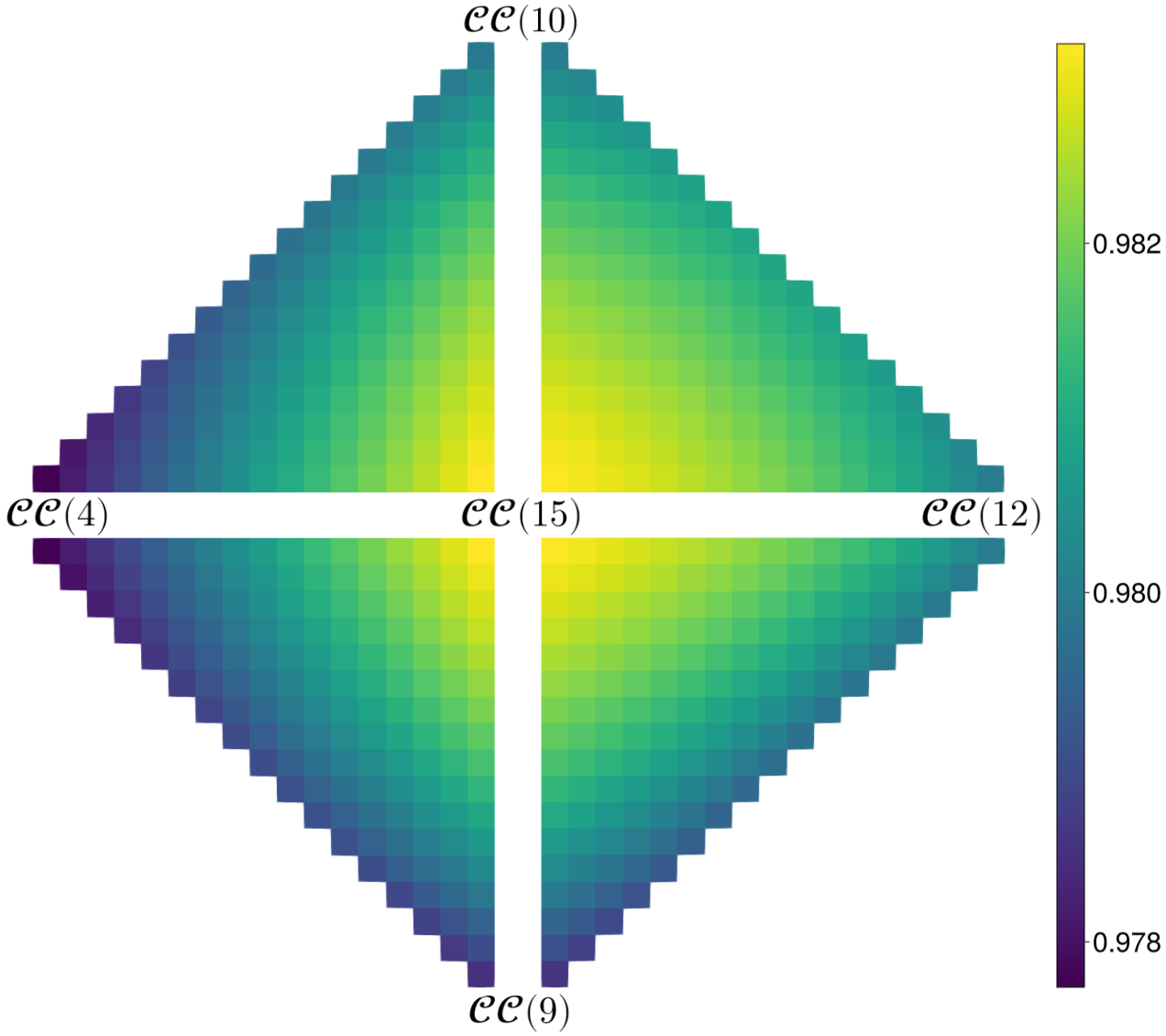}
    \caption{Using the simplicial homotopy $H_{a,b,c}(s,t)$ we compute the number of sampled points where monostationarity is enabled divided by the total number of points sampled from the region $a(\Vector{\eta})>0$ and $b(\Vector{\eta})<0$ (see Section \ref{section:empirical data}). For this experiment, we set $c=15$ and choose $a,b\in \{4,9,10,12\}$. The simplices are discretized with a grid size of $\Delta s=\Delta t=\sfrac{1}{16}$. The best cover $\mathcal{CC}(15)$ at $s=t=0$ is located in the image's center. By increasing the value of $s$ and $t$, we calculate the remaining ratios. 
    }
    \label{fig:simplicial-cover-homotopies}
\end{figure}

These experiments recover the peak between cover $\mathcal{CC}(4)$ and $\mathcal{CC}(9)$ and the linear homotopy $H_{10,12}(t)$ with endpoints $\mathcal{CC}(10)$ and $\mathcal{CC}(12)$ reveals another improved weighted cover. Despite involving three maximally distinct covers in the weighting (cf. Section \ref{section:containment}), cover $\mathcal{CC}(15)$ performs significantly better than all investigated weighted covers. This observation verifies a result from Table \ref{tab:objectivecomparison}, where the region captured by combining all covers is only $0.18\%$ larger than the region of monostationarity characterized by $\mathcal{CC}(15)$. 
Furthermore, Figure \ref{fig:simplicial-cover-homotopies} suggests that it is not sufficient to use three equally weighted pure covers. This is in opposition to our assumptions from Example \ref{ex:homotopy-from-cover-4-to-12}, where the weights for all vertices are equal. 
Reasonably discretizing the search space resulting from using all covers and non-equal weights is a challenging task with diminishing returns.

\section{Conclusion}
In this article, we highlight the connection between the nonnegativity of polynomials and the monostationarity of dual phosphorylation. The corresponding steady state variety enables the use of techniques from real algebraic geometry. By implementing an empirical approach to sample the regions of nonnegativity, we identify and compare the 16 pure covers corresponding to the hexagonal face where the monostationarity of the dual phosphorylation network is not yet completely characterized. In doing so, we establish the existence of pure covers that provide new sufficient conditions for monostationarity in \Cref{cor:newsufficientconditions}, improving the previous work by Feliu et al. \cite{DualPhosphorSONC}. This approach provides a general strategy for investigating the monostationarity of CRNs with symbolic reaction rate constants. 
In the specific case of dual phosphorylation, we can now simultaneously consider the sums of circuit numbers of all 16 pure covers to generate a much stronger sufficient condition on the network's monostationarity than previously available.

Next to the classification of all possible pure covers in \Cref{prop:16-circuit-covers}, we consider weighted covers that combine multiple pure covers by weighting the coefficients. However, it was computationally infeasible to find the optimal cover completely determining the region of monostationarity. In future work, the focus should thus lie on an in-depth investigation of weighted covers for completely classifying the region of monostationarity associated to the dual phosphorylation reaction network. This goes hand in hand with furthering the development of nonnegativity certificates for polynomials with symbolic coefficients.

As demonstrated in \Cref{prop:16-circuit-covers}, the circuit cover $\mathcal{CC}(9)$ proposed by Feliu et al. \cite{DualPhosphorSONC} is not unique. This specific cover was chosen because it uses a minimal number of circuit polynomials and admits a simple expression, since the negative point $\Vector{m}$ lies in the barycenter of the Newton polytopes of all involved circuit polynomials (cf. \cite[Remark 3.6]{DualPhosphorSONC}). However, we demonstrate in this article that $\mathcal{CC}(15)$ corresponds to a larger region of monostationarity, even though it consists of the maximal possible number of circuit polynomials. Furthermore, none of the covers $\mathcal{CC}(4)$, $\mathcal{CC}(10)$, $\mathcal{CC}(12)$ and $\mathcal{CC}(15)$ that perform well in our experiments have $\Vector{m}$ as the barycenter of all of the Newton polytopes associated with the circuit polynomials. Nonetheless, we discuss in \Cref{section:new-sufficient-conditions} that these covers admit certain symmetries. Deriving general conditions on the selection of circuit covers based on such combinatorial observations should be the subject of future research.

\bibliography{sn-bibliography}


\begin{thebibliography}{37}
\ifx \bisbn   \undefined \def \bisbn  #1{ISBN #1}\fi
\ifx \binits  \undefined \def \binits#1{#1}\fi
\ifx \bauthor  \undefined \def \bauthor#1{#1}\fi
\ifx \batitle  \undefined \def \batitle#1{#1}\fi
\ifx \bjtitle  \undefined \def \bjtitle#1{#1}\fi
\ifx \bvolume  \undefined \def \bvolume#1{\textbf{#1}}\fi
\ifx \byear  \undefined \def \byear#1{#1}\fi
\ifx \bissue  \undefined \def \bissue#1{#1}\fi
\ifx \bfpage  \undefined \def \bfpage#1{#1}\fi
\ifx \blpage  \undefined \def \blpage #1{#1}\fi
\ifx \burl  \undefined \def \burl#1{\textsf{#1}}\fi
\ifx \doiurl  \undefined \def \doiurl#1{\url{https://doi.org/#1}}\fi
\ifx \betal  \undefined \def \betal{\textit{et al.}}\fi
\ifx \binstitute  \undefined \def \binstitute#1{#1}\fi
\ifx \binstitutionaled  \undefined \def \binstitutionaled#1{#1}\fi
\ifx \bctitle  \undefined \def \bctitle#1{#1}\fi
\ifx \beditor  \undefined \def \beditor#1{#1}\fi
\ifx \bpublisher  \undefined \def \bpublisher#1{#1}\fi
\ifx \bbtitle  \undefined \def \bbtitle#1{#1}\fi
\ifx \bedition  \undefined \def \bedition#1{#1}\fi
\ifx \bseriesno  \undefined \def \bseriesno#1{#1}\fi
\ifx \blocation  \undefined \def \blocation#1{#1}\fi
\ifx \bsertitle  \undefined \def \bsertitle#1{#1}\fi
\ifx \bsnm \undefined \def \bsnm#1{#1}\fi
\ifx \bsuffix \undefined \def \bsuffix#1{#1}\fi
\ifx \bparticle \undefined \def \bparticle#1{#1}\fi
\ifx \barticle \undefined \def \barticle#1{#1}\fi
\bibcommenthead
\ifx \bconfdate \undefined \def \bconfdate #1{#1}\fi
\ifx \botherref \undefined \def \botherref #1{#1}\fi
\ifx \url \undefined \def \url#1{\textsf{#1}}\fi
\ifx \bchapter \undefined \def \bchapter#1{#1}\fi
\ifx \bbook \undefined \def \bbook#1{#1}\fi
\ifx \bcomment \undefined \def \bcomment#1{#1}\fi
\ifx \oauthor \undefined \def \oauthor#1{#1}\fi
\ifx \citeauthoryear \undefined \def \citeauthoryear#1{#1}\fi
\ifx \endbibitem  \undefined \def \endbibitem {}\fi
\ifx \bconflocation  \undefined \def \bconflocation#1{#1}\fi
\ifx \arxivurl  \undefined \def \arxivurl#1{\textsf{#1}}\fi
\csname PreBibitemsHook\endcsname

%%% 1
\bibitem[\protect\citeauthoryear{Loskot et~al.}{2019}]{10.3389/fgene.2019.00549}
\begin{botherref}
\oauthor{\bsnm{Loskot}, \binits{P.}},
\oauthor{\bsnm{Atitey}, \binits{K.}},
\oauthor{\bsnm{Mihaylova}, \binits{L.}}:
Comprehensive review of models and methods for inferences in bio-chemical reaction networks.
Frontiers in Genetics
\textbf{10}
(2019)
\end{botherref}
\endbibitem

%%% 2
\bibitem[\protect\citeauthoryear{Horn and Jackson}{1972}]{Horn1972}
\begin{barticle}
\bauthor{\bsnm{Horn}, \binits{F.}},
\bauthor{\bsnm{Jackson}, \binits{R.}}:
\batitle{General mass action kinetics}.
\bjtitle{Archive for Rational Mechanics and Analysis}
\bvolume{47}(\bissue{2}),
\bfpage{81}--\blpage{116}
(\byear{1972})
\end{barticle}
\endbibitem

%%% 3
\bibitem[\protect\citeauthoryear{Feinberg}{1972}]{Feinberg1972}
\begin{barticle}
\bauthor{\bsnm{Feinberg}, \binits{M.}}:
\batitle{Complex balancing in general kinetic systems}.
\bjtitle{Archive for Rational Mechanics and Analysis}
\bvolume{49}(\bissue{3}),
\bfpage{187}--\blpage{194}
(\byear{1972})
\end{barticle}
\endbibitem

%%% 4
\bibitem[\protect\citeauthoryear{Laurent and Kellershohn}{1999}]{LAURENT1999418}
\begin{barticle}
\bauthor{\bsnm{Laurent}, \binits{M.}},
\bauthor{\bsnm{Kellershohn}, \binits{N.}}:
\batitle{Multistability: a major means of differentiation and evolution in biological systems}.
\bjtitle{Trends in Biochemical Sciences}
\bvolume{24}(\bissue{11}),
\bfpage{418}--\blpage{422}
(\byear{1999})
\end{barticle}
\endbibitem

%%% 5
\bibitem[\protect\citeauthoryear{Ozbudak et~al.}{2004}]{Ozbudak2004}
\begin{barticle}
\bauthor{\bsnm{Ozbudak}, \binits{E.M.}},
\bauthor{\bsnm{Thattai}, \binits{M.}},
\bauthor{\bsnm{Lim}, \binits{H.N.}},
\bauthor{\bsnm{Shraiman}, \binits{B.I.}},
\bauthor{\bsnm{Oudenaarden}, \binits{A.}}:
\batitle{Multistability in the lactose utilization network of escherichia coli}.
\bjtitle{Nature}
\bvolume{427}(\bissue{6976}),
\bfpage{737}--\blpage{740}
(\byear{2004})
\end{barticle}
\endbibitem

%%% 6
\bibitem[\protect\citeauthoryear{Laurent}{2009}]{Laurent:Survey}
\begin{bchapter}
\bauthor{\bsnm{Laurent}, \binits{M.}}:
\bctitle{Sums of squares, moment matrices and optimization over polynomials}.
In: \bbtitle{Emerging Applications of Algebraic Geometry}.
\bsertitle{IMA Vol. Math. Appl.},
vol. \bseriesno{149},
pp. \bfpage{157}--\blpage{270}.
\bpublisher{Springer},
\blocation{New York}
(\byear{2009})
\end{bchapter}
\endbibitem

%%% 7
\bibitem[\protect\citeauthoryear{Blekherman et~al.}{2013}]{Blekherman:Parrilo:Thomas}
\begin{bbook}
\bauthor{\bsnm{Blekherman}, \binits{G.}},
\bauthor{\bsnm{Parrilo}, \binits{P.A.}},
\bauthor{\bsnm{Thomas}, \binits{R.R.}}:
\bbtitle{Semidefinite Optimization and Convex Algebraic Geometry}.
\bsertitle{MOS-SIAM Series on Optimization},
vol. \bseriesno{13}.
\bpublisher{SIAM and the Mathematical Optimization Society},
\blocation{Philadelphia}
(\byear{2013})
\end{bbook}
\endbibitem

%%% 8
\bibitem[\protect\citeauthoryear{Lasserre}{2015}]{Lasserre:IntroductionPolynomialandSemiAlgebraicOptimization}
\begin{bbook}
\bauthor{\bsnm{Lasserre}, \binits{J.B.}}:
\bbtitle{An Introduction to Polynomial and Semi-Algebraic Optimization}.
\bsertitle{Cambridge Texts in Applied Mathematics},
vol. \bseriesno{1}.
\bpublisher{Cambridge University Press},
\blocation{Cambridge, United Kingdom}
(\byear{2015})
\end{bbook}
\endbibitem

%%% 9
\bibitem[\protect\citeauthoryear{Theobald}{2024}]{Theobald:Book:RealAlgGeom}
\begin{bbook}
\bauthor{\bsnm{Theobald}, \binits{T.}}:
\bbtitle{Real Algebraic Geometry and Optimization}.
\bsertitle{Graduate Studies in Mathematics},
vol. \bseriesno{241}.
\bpublisher{American Mathematical Society},
\blocation{Providence, RI}
(\byear{2024})
\end{bbook}
\endbibitem

%%% 10
\bibitem[\protect\citeauthoryear{Parrilo}{2000}]{parrilo2000structured}
\begin{botherref}
\oauthor{\bsnm{Parrilo}, \binits{P.A.}}:
Structured semidefinite programs and semialgebraic geometry methods in robustness and optimization.
PhD thesis,
California Institute of Technology
(2000)
\end{botherref}
\endbibitem

%%% 11
\bibitem[\protect\citeauthoryear{Iliman and {de Wolff}}{2016}]{SONCinitial}
\begin{botherref}
\oauthor{\bsnm{Iliman}, \binits{S.}},
\oauthor{\bsnm{{de Wolff}}, \binits{T.}}:
Amoebas, nonnegative polynomials and sums of squares supported on circuits.
Research in the Mathematical Sciences
\textbf{3}
(2016)
\end{botherref}
\endbibitem

%%% 12
\bibitem[\protect\citeauthoryear{Reznick}{1989}]{ReznickAGIforms}
\begin{barticle}
\bauthor{\bsnm{Reznick}, \binits{B.}}:
\batitle{Forms derived from the arithmetic-geometric inequality}.
\bjtitle{Mathematische Annalen}
\bvolume{283},
\bfpage{431}--\blpage{464}
(\byear{1989})
\end{barticle}
\endbibitem

%%% 13
\bibitem[\protect\citeauthoryear{Pantea et~al.}{2012}]{CRN_SONC_Pantea}
\begin{barticle}
\bauthor{\bsnm{Pantea}, \binits{C.}},
\bauthor{\bsnm{Koeppl}, \binits{H.}},
\bauthor{\bsnm{Craciun}, \binits{G.}}:
\batitle{Global injectivity and multiple equilibria in uni- and bi-molecular reaction networks}.
\bjtitle{Discrete and Continuous Dynamical Systems - B}
\bvolume{17}(\bissue{6}),
\bfpage{2153}--\blpage{2170}
(\byear{2012})
\end{barticle}
\endbibitem

%%% 14
\bibitem[\protect\citeauthoryear{Iliman and {de Wolff}}{2016}]{Iliman:deWolff:GP}
\begin{barticle}
\bauthor{\bsnm{Iliman}, \binits{S.}},
\bauthor{\bsnm{{de Wolff}}, \binits{T.}}:
\batitle{Lower bounds for polynomials with simplex newton polytopes based on geometric programming}.
\bjtitle{SIAM J. Optim.}
\bvolume{26}(\bissue{2}),
\bfpage{1128}--\blpage{1146}
(\byear{2016})
\end{barticle}
\endbibitem

%%% 15
\bibitem[\protect\citeauthoryear{Papp}{2023}]{dualityofnonnegativecircuits}
\begin{barticle}
\bauthor{\bsnm{Papp}, \binits{D.}}:
\batitle{Duality of sum of nonnegative circuit polynomials and optimal sonc bounds}.
\bjtitle{Journal of Symbolic Computation}
\bvolume{114},
\bfpage{246}--\blpage{266}
(\byear{2023})
\end{barticle}
\endbibitem

%%% 16
\bibitem[\protect\citeauthoryear{Averkov}{2019}]{Gennadiy}
\begin{barticle}
\bauthor{\bsnm{Averkov}, \binits{G.}}:
\batitle{Optimal size of linear matrix inequalities in semidefinite approaches to polynomial optimization}.
\bjtitle{SIAM Journal on Applied Algebra and Geometry}
\bvolume{3}(\bissue{1}),
\bfpage{128}--\blpage{151}
(\byear{2019})
\end{barticle}
\endbibitem

%%% 17
\bibitem[\protect\citeauthoryear{Magron and Wang}{2023}]{MAGRON2023346}
\begin{barticle}
\bauthor{\bsnm{Magron}, \binits{V.}},
\bauthor{\bsnm{Wang}, \binits{J.}}:
\batitle{Sonc optimization and exact nonnegativity certificates via second-order cone programming}.
\bjtitle{Journal of Symbolic Computation}
\bvolume{115},
\bfpage{346}--\blpage{370}
(\byear{2023})
\end{barticle}
\endbibitem

%%% 18
\bibitem[\protect\citeauthoryear{Heuer and {de Wolff}}{2024}]{Heuer:deWolff}
\begin{bchapter}
\bauthor{\bsnm{Heuer}, \binits{J.}},
\bauthor{\bsnm{{de Wolff}}, \binits{T.}}:
\bctitle{Initial application of sonc to lyapunov stability of dynamical systems}.
In: \bbtitle{Proceedings of the 2024 International Symposium on Symbolic and Algebraic Computation},
pp. \bfpage{361}--\blpage{370}
(\byear{2024})
\end{bchapter}
\endbibitem

%%% 19
\bibitem[\protect\citeauthoryear{Cohen}{1989}]{cohen}
\begin{barticle}
\bauthor{\bsnm{Cohen}, \binits{P.}}:
\batitle{The structure and regulation of protein phosphatases}.
\bjtitle{Annual review of biochemistry}
\bvolume{58},
\bfpage{453}--\blpage{508}
(\byear{1989})
\end{barticle}
\endbibitem

%%% 20
\bibitem[\protect\citeauthoryear{Craciun and Feinberg}{2010}]{bb6e59ad-64dd-32f7-8a45-283020c4203c}
\begin{barticle}
\bauthor{\bsnm{Craciun}, \binits{G.}},
\bauthor{\bsnm{Feinberg}, \binits{M.}}:
\batitle{Multiple equilibria in complex chemical reaction networks: Semiopen mass action systems}.
\bjtitle{SIAM Journal on Applied Mathematics}
\bvolume{70}(\bissue{6}),
\bfpage{1859}--\blpage{1877}
(\byear{2010}).
Accessed 2024-09-10
\end{barticle}
\endbibitem

%%% 21
\bibitem[\protect\citeauthoryear{Flockerzi et~al.}{2014}]{Flockerzi2014}
\begin{barticle}
\bauthor{\bsnm{Flockerzi}, \binits{D.}},
\bauthor{\bsnm{Holstein}, \binits{K.}},
\bauthor{\bsnm{Conradi}, \binits{C.}}:
\batitle{N-site phosphorylation systems with 2n-1 steady states}.
\bjtitle{Bulletin of Mathematical Biology}
\bvolume{76}(\bissue{8}),
\bfpage{1892}--\blpage{1916}
(\byear{2014})
\end{barticle}
\endbibitem

%%% 22
\bibitem[\protect\citeauthoryear{Conradi et~al.}{2017}]{ConradiFeliuMinchevaWiuf2017}
\begin{barticle}
\bauthor{\bsnm{Conradi}, \binits{C.}},
\bauthor{\bsnm{Feliu}, \binits{E.}},
\bauthor{\bsnm{Mincheva}, \binits{M.}},
\bauthor{\bsnm{Wiuf}, \binits{C.}}:
\batitle{Identifying parameter regions for multistationarity}.
\bjtitle{PLOS Computational Biology}
\bvolume{13}(\bissue{10}),
\bfpage{1}--\blpage{25}
(\byear{2017})
\end{barticle}
\endbibitem

%%% 23
\bibitem[\protect\citeauthoryear{Conradi and Mincheva}{2014}]{ConradiMincheva2014}
\begin{barticle}
\bauthor{\bsnm{Conradi}, \binits{C.}},
\bauthor{\bsnm{Mincheva}, \binits{M.}}:
\batitle{Catalytic constants enable the emergence of bistability in dual phosphorylation}.
\bjtitle{Journal of The Royal Society Interface}
\bvolume{11}(\bissue{95}),
\bfpage{20140158}
(\byear{2014})
\end{barticle}
\endbibitem

%%% 24
\bibitem[\protect\citeauthoryear{Conradi and Mincheva}{2024}]{Conradi2024}
\begin{botherref}
\oauthor{\bsnm{Conradi}, \binits{C.}},
\oauthor{\bsnm{Mincheva}, \binits{M.}}:
In distributive phosphorylation catalytic constants enable non-trivial dynamics.
Journal of Mathematical Biology
\textbf{89}
(2024)
\end{botherref}
\endbibitem

%%% 25
\bibitem[\protect\citeauthoryear{Feliu et~al.}{2023}]{Feliu:Kaihnsa:Yueruek:deWolff:nSite}
\begin{barticle}
\bauthor{\bsnm{Feliu}, \binits{E.}},
\bauthor{\bsnm{Kaihnsa}, \binits{N.}},
\bauthor{\bsnm{Y{\"u}r{\"uk}}, \binits{O.}},
\bauthor{\bsnm{{de Wolff}}, \binits{T.}}:
\batitle{Parameter region for multistationarity in \({\boldsymbol{n-}}\)site phosphorylation networks}.
\bjtitle{SIAM Journal on Applied Dynamical Systems}
\bvolume{22}(\bissue{3}),
\bfpage{2024}--\blpage{2053}
(\byear{2023})
\end{barticle}
\endbibitem

%%% 26
\bibitem[\protect\citeauthoryear{Bihan et~al.}{2020}]{BIHAN2020367}
\begin{barticle}
\bauthor{\bsnm{Bihan}, \binits{F.}},
\bauthor{\bsnm{Dickenstein}, \binits{A.}},
\bauthor{\bsnm{Giaroli}, \binits{M.}}:
\batitle{Lower bounds for positive roots and regions of multistationarity in chemical reaction networks}.
\bjtitle{Journal of Algebra}
\bvolume{542},
\bfpage{367}--\blpage{411}
(\byear{2020})
\end{barticle}
\endbibitem

%%% 27
\bibitem[\protect\citeauthoryear{Feliu et~al.}{2022}]{DualPhosphorSONC}
\begin{barticle}
\bauthor{\bsnm{Feliu}, \binits{E.}},
\bauthor{\bsnm{Kaihnsa}, \binits{N.}},
\bauthor{\bsnm{{de Wolff}}, \binits{T.}},
\bauthor{\bsnm{Yürük}, \binits{O.}}:
\batitle{The kinetic space of multistationarity in dual phosphorylation}.
\bjtitle{Journal of Dynamics and Differential Equations}
\bvolume{34},
\bfpage{825}--\blpage{852}
(\byear{2022})
\end{barticle}
\endbibitem

%%% 28
\bibitem[\protect\citeauthoryear{Wang and Sontag}{2008}]{WangSontag}
\begin{barticle}
\bauthor{\bsnm{Wang}, \binits{L.}},
\bauthor{\bsnm{Sontag}, \binits{E.D.}}:
\batitle{On the number of steady states in a multiple futile cycle}.
\bjtitle{Journal of Mathematical Biology}
\bvolume{57},
\bfpage{29}--\blpage{52}
(\byear{2008})
\end{barticle}
\endbibitem

%%% 29
\bibitem[\protect\citeauthoryear{Dickenstein}{2016}]{dickenstein2016biochemical}
\begin{barticle}
\bauthor{\bsnm{Dickenstein}, \binits{A.}}:
\batitle{Biochemical reaction networks: An invitation for algebraic geometers}.
\bjtitle{Mathematical congress of the Americas. Contemporary Mathematics}
\bvolume{656},
\bfpage{65}--\blpage{83}
(\byear{2016})
\end{barticle}
\endbibitem

%%% 30
\bibitem[\protect\citeauthoryear{Reznick}{1978}]{ReznickPSDForms}
\begin{barticle}
\bauthor{\bsnm{Reznick}, \binits{B.}}:
\batitle{{Extremal PSD forms with few terms}}.
\bjtitle{Duke Mathematical Journal}
\bvolume{45}(\bissue{2}),
\bfpage{363}--\blpage{374}
(\byear{1978})
\end{barticle}
\endbibitem

%%% 31
\bibitem[\protect\citeauthoryear{Chandrasekaran and Shah}{2016}]{Chandrasekaran:Shah:SAGE}
\begin{barticle}
\bauthor{\bsnm{Chandrasekaran}, \binits{V.}},
\bauthor{\bsnm{Shah}, \binits{P.}}:
\batitle{Relative entropy relaxations for signomial optimization}.
\bjtitle{SIAM J. Optim.}
\bvolume{26}(\bissue{2}),
\bfpage{1147}--\blpage{1173}
(\byear{2016})
\end{barticle}
\endbibitem

%%% 32
\bibitem[\protect\citeauthoryear{Wang}{2022}]{Wang}
\begin{barticle}
\bauthor{\bsnm{Wang}, \binits{J.}}:
\batitle{Nonnegative polynomials and circuit polynomials}.
\bjtitle{SIAM Journal on Applied Algebra and Geometry}
\bvolume{6}(\bissue{2}),
\bfpage{111}--\blpage{133}
(\byear{2022})
\end{barticle}
\endbibitem

%%% 33
\bibitem[\protect\citeauthoryear{Chandrasekaran et~al.}{2021}]{Chandrasekaran:Murray:Wiermann}
\begin{barticle}
\bauthor{\bsnm{Chandrasekaran}, \binits{V.}},
\bauthor{\bsnm{Murray}, \binits{R.}},
\bauthor{\bsnm{Wiermann}, \binits{A.}}:
\batitle{Newton polytopes and relative entropy optimization}.
\bjtitle{Foundations of Computational Mathematics}
\bvolume{21},
\bfpage{1703}--\blpage{1737}
(\byear{2021})
\end{barticle}
\endbibitem

%%% 34
\bibitem[\protect\citeauthoryear{Dressler et~al.}{2016}]{constrainedoptimizationapproach}
\begin{barticle}
\bauthor{\bsnm{Dressler}, \binits{M.}},
\bauthor{\bsnm{Iliman}, \binits{S.}},
\bauthor{\bsnm{{de Wolff}}, \binits{T.}}:
\batitle{An approach to constrained polynomial optimization via nonnegative circuit polynomials and geometric programming}.
\bjtitle{Journal of Symbolic Computation}
\bvolume{91},
\bfpage{149}--\blpage{172}
(\byear{2016})
\end{barticle}
\endbibitem

%%% 35
\bibitem[\protect\citeauthoryear{Dressler et~al.}{2017}]{Dressler:Iliman:deWolff:Positivstellensatz}
\begin{barticle}
\bauthor{\bsnm{Dressler}, \binits{M.}},
\bauthor{\bsnm{Iliman}, \binits{S.}},
\bauthor{\bsnm{{de Wolff}}, \binits{T.}}:
\batitle{A {P}ositivstellensatz for {S}ums of {N}onnegative {C}ircuit {P}olynomials}.
\bjtitle{SIAM J. Appl. Algebra Geom.}
\bvolume{1}(\bissue{1}),
\bfpage{536}--\blpage{555}
(\byear{2017})
\end{barticle}
\endbibitem

%%% 36
\bibitem[\protect\citeauthoryear{Y{\"u}r{\"u}k}{2021}]{dissYuruk}
\begin{botherref}
\oauthor{\bsnm{Y{\"u}r{\"u}k}, \binits{O.}}:
On the maximal mediated set structure and the applications of nonnegative circuit polynomials.
PhD thesis,
TU Braunschweig
(2021)
\end{botherref}
\endbibitem

%%% 37
\bibitem[\protect\citeauthoryear{Israel}{1992}]{determiningsamplesizes}
\begin{botherref}
\oauthor{\bsnm{Israel}, \binits{G.D.}}:
{Determining Sample Size}.
University of Florida Cooperative Extension Service, Institute of Food and Agriculture Sciences, EDIS
\textbf{25}
(1992)
\end{botherref}
\endbibitem

\end{thebibliography}
\vspace*{6mm}
\begin{appendices}

\section{Table: Relative Comparison}\label{section:appendix}
\vspace*{-3mm}
\begin{table}[h!]
    \centering
        \caption{\justifying We drew roughly $5.2\cdot 10^7$ samples from the hypercubical region in $\mathbb{R}_{+}^{12}$ described by $a(\Vector{\eta})> 0$ and $b(\Vector{\eta})<0$. We label the times that our model performed better than cover $\mathcal{CC}(9)$ with ``$+$'', the times $\mathcal{CC}(9)$ performs better with a ``$-$'' and the samples where neither model was able to deduce monostationarity is recorded in the column ``0''. Each of these numbers is subsequently divided by the total amount of samples. All of these ratios are multiplied by $10^4$ and then rounded to two decimal digits for the sake of brevity. Every cover is compared with cover $\mathcal{CC}(9)$ (blue row) that was used in Feliu et al. \cite{DualPhosphorSONC}. For that reason, we mark the comparative ``$+$'' and ``$-$'' columns corresponding to $\mathcal{CC}(9)$ with an ``x''.}
    \label{tab:relative+-comparison}
 
    \bgroup
\def\arraystretch{1.8}
\setlength\tabcolsep{1.81ex}
    \begin{tabular}{| c " c c c | c c c | c c c | c c c|}

 \multicolumn{1}{c"}{~}& \multicolumn{3}{c|}{$[0,0.1]^{12}$}&\multicolumn{3}{c|}{$[0,1]^{12}$}& \multicolumn{3}{c|}{$[0,10]^{12}$}& \multicolumn{3}{c|}{$[0,100]^{12}$}
\\[.5mm] 

\hline
\#&+&$-$&0&+&$-$&0&+&$-$&0&+&$-$&0\\ \thickhline

$\mathcal{CC}(1)$&0.0&2.91&2.15&0.0&2.91&2.15&0.0&2.92&2.15&0.0&2.92&2.15\\ \hline

$\mathcal{CC}(2)$&0.22&1.83&1.93&0.22&1.84&1.93&0.22&1.84&1.93&0.22&1.84&1.92\\ \hline

$\mathcal{CC}(3)$&0.03&1.37&2.12&0.03&1.37&2.12&0.03&1.37&2.12&0.03&1.37&2.11\\ \hline

\cellcolor{celltwo!23}$\mathcal{CC}(4)$&\cellcolor{celltwo!23}0.43&\cellcolor{celltwo!23}0.50&\cellcolor{celltwo!23}1.72&\cellcolor{celltwo!23}0.43&\cellcolor{celltwo!23}0.50&\cellcolor{celltwo!23}1.72&\cellcolor{celltwo!23}0.43&\cellcolor{celltwo!23}0.50&\cellcolor{celltwo!23}1.72&\cellcolor{celltwo!23}0.43&\cellcolor{celltwo!23}0.50&\cellcolor{celltwo!23}1.72\\ \hline

$\mathcal{CC}(5)$&0.22&1.84&1.93&0.22&1.83&1.93&0.22&1.84&1.93&0.22&1.83&1.92\\ \hline

$\mathcal{CC}(6)$&0.0&1.32&2.15&0.0&1.32&2.15&0.0&1.32&2.15&0.0&1.32&2.15\\ \hline

$\mathcal{CC}(7)$&0.0&1.32&2.15&0.0&1.32&2.15&0.0&1.32&2.15&0.0&1.32&2.15\\ \hline

$\mathcal{CC}(8)$&0.01&0.69&2.14&0.01&0.69&2.14&0.01&0.69&2.14&0.01&0.69&2.14\\ \hline

\cellcolor{cellthree!23}$\mathcal{CC}(9)$&\cellcolor{cellthree!23}x&\cellcolor{cellthree!23}x&\cellcolor{cellthree!23}2.15&\cellcolor{cellthree!23}x&\cellcolor{cellthree!23}x&\cellcolor{cellthree!23}2.15&\cellcolor{cellthree!23}x&\cellcolor{cellthree!23}x&\cellcolor{cellthree!23}2.15&\cellcolor{cellthree!23}x&\cellcolor{cellthree!23}x&\cellcolor{cellthree!23}2.15\\ \hline

\cellcolor{celltwo!23}$\mathcal{CC}(10)$&\cellcolor{celltwo!23}0.34&\cellcolor{celltwo!23}0.19&\cellcolor{celltwo!23}1.81&\cellcolor{celltwo!23}0.34&\cellcolor{celltwo!23}0.19&\cellcolor{celltwo!23}1.81&\cellcolor{celltwo!23}0.34&\cellcolor{celltwo!23}0.19&\cellcolor{celltwo!23}1.81&\cellcolor{celltwo!23}0.34&\cellcolor{celltwo!23}0.19&\cellcolor{celltwo!23}1.81\\ \hline

$\mathcal{CC}(11)$&0.22&1.84&1.93&0.22&1.83&1.93&0.22&1.84&1.93&0.22&1.83&1.92\\ \hline

\cellcolor{celltwo!23}$\mathcal{CC}(12)$&\cellcolor{celltwo!23}0.34&\cellcolor{celltwo!23}0.19&\cellcolor{celltwo!23}1.81&\cellcolor{celltwo!23}0.34&\cellcolor{celltwo!23}0.19&\cellcolor{celltwo!23}1.81&\cellcolor{celltwo!23}0.34&\cellcolor{celltwo!23}0.20&\cellcolor{celltwo!23}1.81&\cellcolor{celltwo!23}0.34&\cellcolor{celltwo!23}0.20&\cellcolor{celltwo!23}1.81\\ \hline

$\mathcal{CC}(13)$&0.22&1.83&1.93&0.22&1.84&1.93&0.22&1.84&1.93&0.22&1.84&1.92\\ \hline

$\mathcal{CC}(14)$&0.0&2.91&2.15&0.0&2.91&2.15&0.0&2.92&2.15&0.0&2.92&2.15\\ \hline

\cellcolor{celltwo!23}$\mathcal{CC}(15)$&\cellcolor{celltwo!23}0.59&\cellcolor{celltwo!23}0.13&\cellcolor{celltwo!23}1.56&\cellcolor{celltwo!23}0.59&\cellcolor{celltwo!23}0.13&\cellcolor{celltwo!23}1.56&\cellcolor{celltwo!23}0.59&\cellcolor{celltwo!23}0.13&\cellcolor{celltwo!23}1.56&\cellcolor{celltwo!23}0.59&\cellcolor{celltwo!23}0.13&\cellcolor{celltwo!23}1.56\\ \hline
$\mathcal{CC}(16)$&0.05&0.58&2.10&0.05&0.58&2.10&0.05&0.57&2.10&0.05&0.58&2.10\\ \hline

    \end{tabular}
    \egroup
\end{table}

\end{appendices}

\end{document}